\documentclass[reqno]{amsart}
\usepackage{amsmath,amssymb,mathrsfs,amsthm,amsfonts}
\usepackage[inline]{enumitem} 
\usepackage[utf8]{inputenc} 
\usepackage[usenames,dvipsnames]{xcolor}
\usepackage{psfrag} 
\usepackage{mathbbol} 
\usepackage{hyperref}
\usepackage{subcaption}
\hypersetup{%
	colorlinks=true, linkcolor=blue,
	citecolor=ForestGreen
}

\usepackage[paper=letterpaper,margin=1.1in]{geometry}

\newtheorem{theorem}{Theorem}[section]
\newtheorem{lemma}[theorem]{Lemma}
\newtheorem{proposition}[theorem]{Proposition}
\newtheorem{corollary}[theorem]{Corollary}
\theoremstyle{definition}

\newtheorem{assumption}[theorem]{Assumption}
\newtheorem{remark}[theorem]{Remark}
\newtheorem{example}[theorem]{Example}

\numberwithin{equation}{section}
\allowdisplaybreaks

\usepackage{acronym}
\acrodef{KPZ}{Kardar--Parisi--Zhang}
\acrodef{ASEP}{Asymmetric Simple Exclusion Process}
\acrodef{TASEP}{Totally Asymmetric Simple Exclusion Process}
\acrodef{SHE}{Stochastic Heat Equation}
\acrodef{SPDE}{Stochastic Partial Differential Equation}
\acrodef{PAM}{Parabolic Anderson Model}

\renewcommand{\Pr}{ \mathbf{P} }							
\newcommand{\Ex}{ \mathbf{E} }								
\newcommand{\Exrt}{ \mathbf{E}^\rt }						
\newcommand{\ExO}{ \mathbf{E}^{\rt*} }						
\newcommand{\normrt}[2]{ \Vert #1\Vert_{\rt,#2}} 			
\newcommand{\Normrt}[2]{ \Big\Vert #1\Big\Vert_{\rt,#2}} 
\newcommand{\normO}[2]{ \Vert #1\Vert_{\rt*,#2}} 			
\newcommand{\NormO}[2]{ \Big\Vert #1\Big\Vert_{\rt*,#2}}	

\newcommand{\Ham}{ \mathcal{H} }	
\newcommand{\Sg}{ \mathcal{Q} }		
\newcommand{\Sgr}{ \mathcal{R} }	
\newcommand{\ham}{ \mathbb{H} }		
\newcommand{\sg}{ \mathbb{Q} }		
\newcommand{\sgr}{ \mathbb{R} }		
\newcommand{\sgra}{ \mathbb{R}^{\rt} }
\newcommand{\sgmg}{ \mathbb{K} }	
\newcommand{\sgmgg}{ \mathbb{k} }	
\newcommand{\HK}{ \mathcal{P} }				
\newcommand{\hk}{ \mathbb{p} }				
\newcommand{\hka}{ \mathbb{p}^{\rt} }		
\newcommand{\hkr}{ \mathbb{r} }				
\newcommand{\Id}{\text{Id}}		

\newcommand{\norm}[1]{\Vert{#1}\Vert}	
\newcommand{\hold}[1]{[{#1}]}				

\newcommand{\SgrI}{ \mathcal{K} }	
\newcommand{\Ips}{ \mathcal{U} }	
\newcommand{\sgrI}{ \mathbb{K} }	
\newcommand{\ips}{ \mathbb{U} }		

\newcommand{\Rt}{ \mathbb{A} }			
\newcommand{\Rtlim}{ \mathcal{A} }		
\newcommand{\rt}{ \mathbb{a} }			
\newcommand{\rtt}{ \widetilde{\rt} }	
\newcommand{\iid}{ \mathbb{b} }			
\newcommand{\filZ}{ \mathscr{F} }		

\newcommand{\eigf}{ \varphi }		
\newcommand{\eigv}{ \lambda }		
\newcommand{\eigsp}{ \langle \{\eigf_n\} \rangle } 
\newcommand{\Mg}{ \mathcal{M} }		
\newcommand{\Mgg}{ \mathcal{L} }	
\newcommand{\mgg}{ L }				
\newcommand{\linmgD}{ G }			
\newcommand{\LinmgD}{ \mathcal{G} }	
\newcommand{\linmgR}{ H }			
\newcommand{\bdd}{\mathcal{B}}		

\newcommand{\ind}{ \mathbf{1} }		
\newcommand{\set}[1]{ {\{#1\}} }	
\newcommand{\dist}{ \mathrm{dist} }	
\newcommand{\urt}{ u_{\Rt} }		
\newcommand{\uic}{ u_\ic }			

\newcommand{\PoiR}{P_\rightarrow}	
\newcommand{\PoiL}{P_\leftarrow}	
\newcommand{\PoiiR}{Q_\rightarrow}	
\newcommand{\PoiiL}{Q_\leftarrow}	
\newcommand{\mg}{M}					

\newcommand{\Z}{\mathbf{Z}} 
\newcommand{\T}{\mathbb{T}} 
\newcommand{\R}{\mathbf{R}} 
\newcommand{\limT}{\mathcal{T}} 
\newcommand{\Parti}{\mathcal{T}} 	
\newcommand{\Midp}{\mathrm{Mid}} 	
\newcommand{\parti}{\mathbb{T}} 	
\newcommand{\midp}{\mathrm{mid}} 	

\newcommand{\ic}{\mathrm{ic}}
\newcommand{\e}{\varepsilon}
\newcommand{\scZ}{Z}				
\newcommand{\limZ}{\mathcal{Z}} 	

\renewcommand{\hat}{\widehat}
\newcommand{\til}{\widetilde}
\renewcommand{\bar}{\overline}

\newcommand{\tilx}{\widetilde{x}}
\newcommand{\tily}{\widetilde{y}}

\usepackage{graphicx}
\newcommand*{\Cdot}{{\raisebox{-0.5ex}{\scalebox{1.8}{$\cdot$}}}} 

\title[Weakly Inhomogeneous ASEP]
{SPDE Limit of Weakly Inhomogeneous ASEP}
\author[I.\ Corwin]{Ivan Corwin}
\address{I.\ Corwin,
	Departments of Mathematics, Columbia University,
	\newline\hphantom{\quad \ I.\ Corwin}
	2990 Broadway, New York, NY 10027 USA
	}
\email{ivan.corwin@gmail.com}

\author[L.-C.\ Tsai]{Li-Cheng Tsai}
\address{L.-C.\ Tsai,
	Department of Mathematics, Rutgers University — New Brunswick
	\newline\hphantom{\quad \ L.-C. Tsai}
	110 Frelinghuysen Road, Piscataway, NJ 08854 USA
	}
\email{lctsai.math@gmail.com}
\begin{document}
\begin{abstract}
We study ASEP in a spatially inhomogeneous environment on a torus $ \T^{(N)} = \Z/N\Z$ of $ N $ sites.
A given inhomogeneity $ \rtt(x)\in(0,\infty) $, $ x\in\T $,
perturbs the overall asymmetric jumping rates $ r<\ell\in(0,1) $ at bonds,
so that particles jump from site $x$ to $x+1$ with rate $r\rtt(x)$ and from $x+1$ to $x$ with rate $\ell \rtt(x)$ (subject to the exclusion rule in both cases).
Under the limit $ N\to\infty $, we suitably tune the asymmetry $ (\ell-r) $ to zero like $N^{-\frac{1}{2}}$ and the inhomogeneity $ \rtt $ to unity,
so that the two compete on equal footing. At the level of the G\"{a}rtner (or microscopic Hopf--Cole) transform, we show convergence to a new SPDE --- the Stochastic Heat Equation with a mix of spatial and spacetime multiplicative noise.
Equivalently, at the level of the height function we show convergence to the Kardar-Parisi-Zhang equation with a mix of spatial and spacetime additive noise.

Our method applies to a general class of $ \rtt(x) $,
which, in particular, includes i.i.d., long-range correlated, and periodic inhomogeneities.
The key technical component of our analysis consists of a host of estimates on the \emph{kernel} of the semigroup $ \Sg(t):=e^{t\Ham} $
for a Hill-type operator $ \Ham:= \frac12\partial_{xx} + \Rtlim'(x) $, and its discrete analog,
where $ \Rtlim $ (and its discrete analog) is a generic H\"{o}lder continuous function.
\end{abstract}

\maketitle

\section{Introduction}
\label{intro}
In this article we study the \ac{ASEP} in a spatially inhomogeneous environment where the inhomogeneity perturbs the rate of jumps across bonds, while maintaining the asymmetry (i.e., the ratio of the left and right rates across the bond).
Quenching the inhomogeneity, we run the \ac{ASEP} and study its resulting Markov dynamics.
Even without inhomogeneities, \ac{ASEP} demonstrates an interesting scaling limit to the \ac{KPZ} equation 
when the asymmetry is tuned weakly~\cite{bertini97}.

It is ultimately interesting to determine how the inhomogeneous rates modify the dynamics of such systems, and scaling limits thereof.
In this work we tune the strengths of the asymmetry and the inhomogeneity to compete on equal levels,
and we find that the latter introduces a new spatial noise into the limiting equation.
At the level of the G\"{a}rtner (or microscopic Hopf--Cole) transform (see~\eqref{eq:Z}),
we obtain a new equation of \ac{SHE}-type, with a mix of spatial and spacetime multiplicative noise. At the level of the height function, we obtain a new equation of \ac{KPZ}-type, with a mix of spatial and spacetime additive noise.

We now define the inhomogeneous \ac{ASEP}.
The process runs on a discrete $N$-site torus $ \T^{(N)} := \Z/N\Z $ where we identify $ \T^{(N)} $ with $ \{0,1,\ldots,N-1\} $,
and, for $ x,y\in\T $, understand $ x+y $ to be mod $ N $.
To alleviate heavy notation, we will often omit dependence on $ N $ and write $ \T $ in place of $ \T^{(N)} $,
and similarly for notation to come.
For fixed homogeneous jumping rates $ r<\ell\in(0,1) $ with $ r+\ell=1 $,
and for fixed inhomogeneities $ \rtt(x)\in(0,\infty) $, $ x\in\T $,
the inhomogeneous \ac{ASEP} consists of particles performing continuous time random walks on $ \T $.
Jumping from $ x $ to $ x+1 $ occurs at rate $ \rtt(x)r $, jumping from $ x+1 $ to $ x $ currents at rate $ \rtt(x)\ell $,
and attempts to jump into occupied sites are forbidden.
See Figure~\ref{fig:asepRing}.
\begin{figure}[h]
\centering
\begin{subfigure}{.35\textwidth}
	\psfrag{L}[r]{$ \ell \rtt(x-1) $}
	\psfrag{R}[c][t]{$ r \rtt(x) $}
	\includegraphics[width=\textwidth]{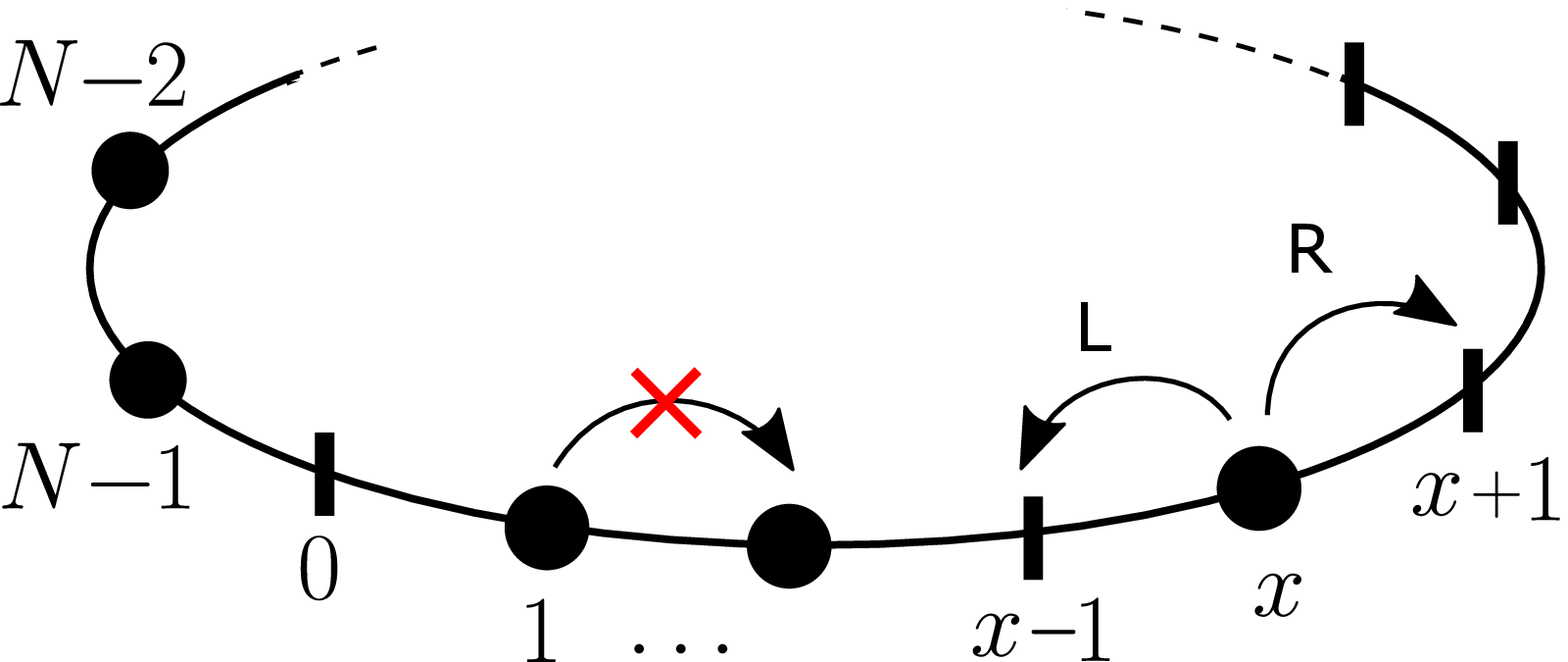}
	\caption{Inhomogeneous \ac*{ASEP}}
	\label{fig:asepRing}
\end{subfigure}
\hfil
\begin{subfigure}{.45\textwidth}
	\includegraphics[width=\textwidth]{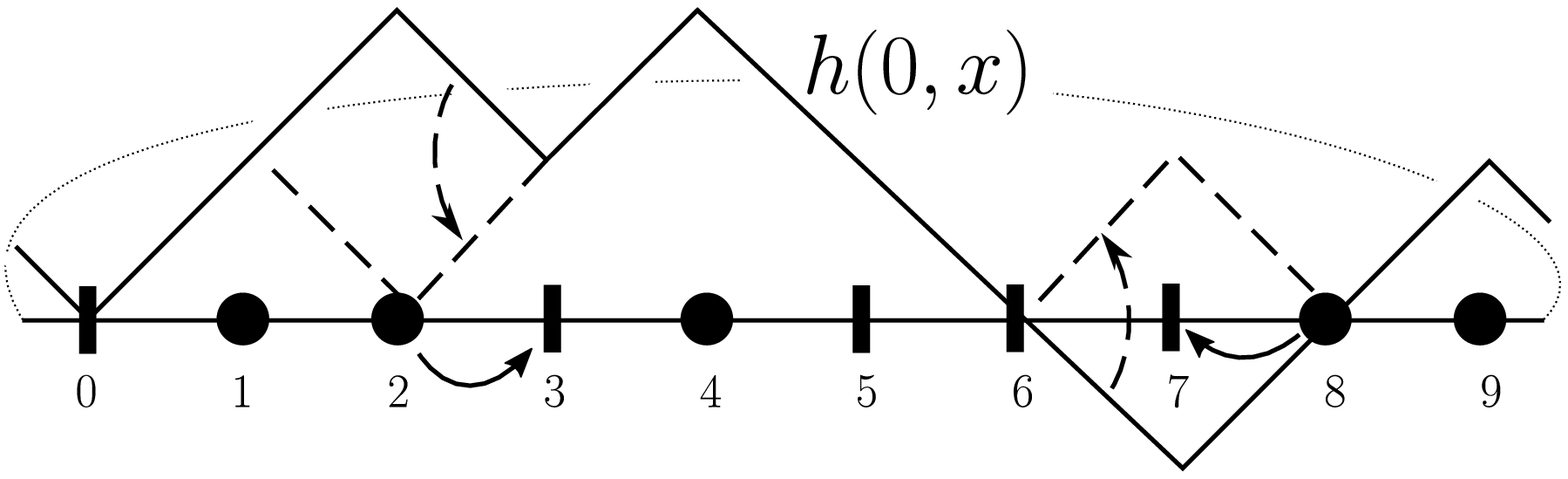}
	\caption{The height function}
	\label{fig:height}
\end{subfigure}
\caption{Inhomogeneous \ac*{ASEP} on $ \T $ and its height function. (a): The particle at $x$ jumps to $x-1$ at rate $\rtt(x-1) \ell$ or to $x+1$ at rate $\rtt(x) r$; meanwhile the particle at $1$ may not jump to the occupied site $2$. (b): The particle dynamics are coupled with a height function as shown.}
\end{figure}

We will focus on the height function (also known as integrated current), denoted $ h(t,x) $.
To avoid technical difficulties, throughout this article we assume the particle system to be \textbf{half-filled} so that $ N $ is even, and there are exactly $ \frac{N}{2} $ particles.
Under this setup, letting
\begin{align*}
	\eta(t,x) := \left\{\begin{array}{l@{,}l}
		1	&\text{ if the site } x \text{ if occupied at } t,
	\\
		0	&\text{ if the site } x \text{ if empty at } t,
	\end{array}\right.
\end{align*}
denote the occupation variables,
we define the height function $ h: [0,\infty)\times\T\to\R $ at $ t=0 $ to be
\begin{align*}
	h(0,x) &:= \sum_{0<y\leq x} \big( 2\eta(0,y) - 1 \big),
	\qquad
	x=0,1,\ldots,N-1.
\end{align*}
Then, for $ t \geq 0 $, each jump of a particle from $ x $ to $ x+1 $ decreases $ h(t,x) $ by $ 2 $,
and each jump of a particle from $ x+1 $ to $ x $ increases $ h(t,x) $ by $ 2 $,
as depicted in Figure~\ref{fig:height}.

We now provide two key definitions needed to state and prove our main results. The \textbf{G\"{a}rtner (or microscopic Hopf--Cole) transform} of inhomogeneous ASEP is defined to be
\begin{align}
	\label{eq:Z}
	Z(t,x) := \tau^{ \frac12h(t,x)}e^{\nu t},
	\qquad
	\tau := r/\ell,
	\qquad
	\nu := 1-2\sqrt{r\ell}.
\end{align}
The other key definition is \textbf{weak asymmetry} scaling in which the system size $N$ controls the asymmetry and scaling of $Z$ as
\begin{align}
	\label{eq:was}
	\ell = \tfrac12(1+N^{-\frac12}),
	\
	r = \tfrac12(1-N^{-\frac12}),
	\qquad
	\scZ_N(t,x) := Z( t N^2,xN).
\end{align}

G\"{a}rtner \cite{gaertner87} introduced his eponymous transform in the context of \emph{homogeneous} \ac{ASEP} (i.e., $ \rtt(x)\equiv 1 $), where he observed that \eqref{eq:Z} linearizes the drift of the microscopic equation, so that $ Z(t,x) $ solves a microscopic \ac{SHE}:
\begin{align}
	\label{eq:dshe}
	dZ(t,x) = \sqrt{r\ell}\Delta Z(t,x) + d\mg(t,x),
	\qquad
	\Delta Z(t,x) := Z(t,x+1)+Z(t,x-1) - 2Z(t,x),
\end{align}
where $ \mg(t,x) $ is an explicit martingale in $ t $.
Using G\"{a}rtner's transform as a starting point, Bertini and Giacomin \cite{bertini97}
showed that a \ac{SPDE} arises
under the weak asymmetry scaling. Namely,  under the scaling~\eqref{eq:was}, the process $ \scZ_N $ converges
to the solution of the \ac{SHE}:
\begin{align}
	\label{eq:SHE}
	\partial_t \limZ = \tfrac12\partial_{xx} \limZ + \xi\limZ,	
\end{align}
where $ \limZ=\limZ(t,x) $, $ (t,x)\in[0,\infty)\times\R $,
and $ \xi=\xi(t,x) $ denotes the Gaussian spacetime white noise (see, e.g.,~\cite{walsh86}). In fact, the result of~\cite{bertini97} is on the full-line $ \Z $,
and in that context $ \e\to 0 $ represents lattice spacing, which is identified with $ N^{-1} $ here. Also, the work of \cite{bertini97} assumes near stationary initial conditions similar to the ones considered here in~\eqref{eq:nearst}. Other initial conditions were considered later in \cite{ACQ}

The first key observation of our present paper is that for inhomogeneities of the form introduced above, G\"{a}rtner's transform remains essentially valid. In particular, the Laplacian term $\sqrt{r\ell}\Delta$ in the discrete \ac{SHE} \eqref{eq:dshe} is replaced by the spatially inhomogeneous operator
$$
\ham := \sqrt{r\ell} \, \rtt(x)\Delta - \nu \, \big(\rtt(x)-1\big),
$$
which involves a mixture of the \textbf{Bouchaud trap model} generator and a \textbf{parabolic Anderson model} type potential (see the sketch of the proof later in this introduction for further explanation and references for these terms). Besides this change, the martingale is also modified.

Armed with the G\"{a}rtner transform, we investigate the effect of $ \rtt(x) $ at large scales in the $ N\to\infty $ limit.
In doing so, we focus on the case where the effect of $ \rtt(x) $ is compatible with the aforementioned \ac{SPDE} limit.
A prototype of our study is when
\begin{align*}
	\rtt(x) = 1 + \tfrac{1}{\sqrt{N}} \iid(x),
	\qquad
	\{\iid(x):x\in\T\} \text{ i.i.d., bounded, with } \Ex[\iid(x)]=0.
\end{align*}
For this example of i.i.d.\ inhomogeneities, the $ N^{-\frac12} $ scaling is weak enough to have an \ac{SPDE} limit, while still strong enough to modify the nature of said limit.

To demonstrate the generality of our approach,
we will actually consider a much more general class of inhomogeneities.
Let us first prepare some notation.
For $ x,x'\in\T $, let $ [x,x']\subset \T $ denote the closed interval on $ \T $
that goes counterclockwise (see Figure~\ref{fig:asepRing} for the orientation) from $ x $ to $ x' $,
and similarly for open and half-open intervals.
With $ |I| $ denoting the cardinality of (i.e., number of points within) an interval $ I\subset\T $,
we define the \textbf{geodesic distance}
\begin{align*}
	\dist_{\T}(x,x') := |(x,x']| \wedge |(x',x]|.
\end{align*}
We will also be considering the continuum torus $ \limT := \R/\Z \simeq [0,1) $,
which is to be viewed as the $ N\to\infty $ limit of $ \frac{1}{N}\T $.

Similarly for the continuum torus $ \limT $,
we let $ [x,x']\subset\limT $ denote the interval going from $ x $ to $ x' $ counterclockwise,
let $ |[x,x']| $ denote the length of the interval,
and let $ \dist_{\limT}(x,x') $, $ x,x'\in\limT $ denote the analogous geodesic distance on $ \limT $.
For $ u\in[0,1] $, let $ C^{u}[0,1] $ denote the space of $ u $-H\"{o}lder continuous functions $ f:[0,1]\to\R $,
equipped with the norm 
\begin{align}
	\label{eq:Hold}
	\norm{f}_{C^u[0,1]} := \norm{f}_{L^\infty[0,1]}+ [f]_{C^u(\limT)},
	\qquad
	[f]_{C^u(\limT)}
	:=
	\sup_{[x, x']\subset\limT} \Big( \frac{1}{|[x,x']|^u} \Big| \int_{[x,x']\setminus\{0\}} d f(y)\Big| \Big),
\end{align}
where the integral is in the Riemann–Stieltjes sense.
The integral excludes $ 0 $ so that the possible jump of $ f $ there will not be picked up.

We now define the type of inhomogeneities to be studied.
Throughout this article, we will consider possibly random $ (\rtt^{(N)}(x))_{x\in\T} $
that may depend on $ N $.
Set $ \rt^{(N)}(x) := \rtt^{(N)}(x)-1 $, and put
\begin{align}
	\label{eq:Rt}
	\Rt^{(N)}(x,x') := -\frac12 \sum_{y\in(x,x']} \rt^{(N)}(y),
	\qquad
	x,x'\in\T.
\end{align}
As announced previously, we will often write $ \rtt^{(N)} = \rtt $, $ \rt^{(N)}=\rt $, $ \Rt^{(N)}=\Rt $, etc., to simplify notation.
When $ x=0 $, we will write $ \Rt(x):=\Rt(0,x) $.
Consider also the scaled partial sums $ \Rt_N(x,x') := \Rt(xN,x'N) $ and $ \Rt_N(x) := \Rt(xN) $,
which are linearly interpolated to be functions on $ \limT^2 $ and $ [0,1) $, respectively.
For $ f:\T^2\to\R $, we define a seminorm that is analogous to $ [\Cdot]_{C^u(\limT)} $ in~\eqref{eq:Hold}: 
\begin{align}
	\label{eq:hold}
	\hold{f}_{u,N} := \sup_{[x,x']\subset\T} \Big( \frac{1}{(|(x,x']|/N)^u} |f(x,x')| \Big).
\end{align}
Throughout this article we assume $ \{\rtt(x):x\in\T\} $ satisfies: 

\begin{assumption}\label{assu:rt}
\begin{enumerate}[label=(\alph*),leftmargin=7ex]
\item[]
\item \label{assu:rt:bdd}
For some fixed constant $ c\in(0,\infty) $,
$ \frac{1}{c} \leq \rtt(x) \leq c $.
\item \label{assu:rt:holder}
For some $ \urt>0 $, the partial sum $ \Rt_N(x,x') $ is $ \urt $-H\"{o}lder continuous:
\begin{align*}
	\lim_{\Lambda\to\infty}\liminf_{N\to\infty}
	\Pr\Big[
		\hold{\Rt_N}_{\urt,N} \leq \Lambda
	\Big]
	=1.
\end{align*}
\item \label{assu:rt:limit}For the same  $ \urt>0 $ as in \ref{assu:rt:holder}
there exists a $ C^{\urt}[0,1] $-valued process $ \Rtlim $ such that
\begin{align*}
	\sup_{ x\in[0,\frac{N-1}{N}) } |\Rt_N(x)-\Rtlim(x)| \longrightarrow_\text{P} 0,
	\qquad
	\text{as } N\to\infty,
\end{align*}
where $ \to_\text{P} $ denotes convergence in probability.
\end{enumerate}
\end{assumption}
\begin{remark}\label{rmk:rt}
\begin{enumerate}[label=(\alph*),leftmargin=7ex]
\item []
\item
Assumption~\ref{assu:rt}\ref{assu:rt:bdd} ensures the rate $ \rtt(x) $
is always nonnegative so that the process is well-defined.
\item Note that we do \emph{not} assume $ (\rt(0)+\ldots\rt(N-1))=0 $ or $ \Rtlim(1)=0 $.
\item Under Assumption~\ref{assu:rt}\ref{assu:rt:limit}, the microscopic processes $\big\{\Rt^{(N)}\big\}_{N}$ (and likewise $\big\{\rt^{(N)}\big\}_{N}$)
and limiting process $ \Rtlim $ are \emph{coupled} on the same probability space.
\end{enumerate}
\end{remark}

Here we list a few examples that fit into our working assumption~\ref{assu:rt}.
\begin{example}[i.i.d.\ inhomogeneities]
\label{ex:iid}
Consider $ \rt^{(N)}(x)=\tfrac{1}{\sqrt{N}} \iid(x) $,
where $\{\iid(x):x\in\T\}$ are i.i.d., bounded, with $ \Ex[\iid(x)]=0 $ and $ \Ex[\iid(x)^2] := \sigma^2>0 $.
Indeed, Assumptions \ref{assu:rt}\ref{assu:rt:bdd}--\ref{assu:rt:holder} are satisfied for any $ \urt\in(0,\frac12) $ (and $ N $ large enough).
The invariance principle asserts that $ \Rt^{(N)}(x/N) $ converges in distribution to $ \frac12\sigma B(x) $ in $ C[0,1] $, where $ B(x) $ denotes a standard Brownian motion.
By Skorokhod's representation theorem,
after suitable extension of the probability space, we can couple $ \{\Rt^{(N)}\}_N$ and $B$ together
so that Assumption~\ref{assu:rt}\ref{assu:rt:limit} holds.
\end{example}

\begin{example}[fractional Brownian motion]
\label{ex:fbm}
Let $ B^\alpha(x) $, $ x \geq 0 $, denote a fractional Brownian motion of a fixed Hurst exponent $ \alpha\in(0,1) $.
For $ x\in \T^{(N)}\simeq \{0,1,\ldots,N-1\} $, set $ \hat{\rt}^{(N)}(x) = B^*(\frac{x+1}{N})-B^*(\frac{x}{N}) $,
and $ \rt^{(N)}(x) := \hat{\rt}^{(N)}(x)\ind_\set{|\hat{\rt}^{(N)}(x)|<1/2} $. To be clear, we define $\hat{\rt}^{(N)}(N-1) = B^*(1)-B^*(\frac{N-1}{N})$.
The indicator $ \ind_\set{|\hat{\rt}^{(N)}(x)|<1/2} $ forces Assumption~\ref{assu:rt}\ref{assu:rt:bdd} to hold.
Since each $ \hat{\rt}^{(N)}(x) $ is a mean-zero Gaussian of variance $ N^{-2\alpha} $,
we necessarily have that
\begin{align*}
	\Pr\big[ \rt^{(N)}(x) = \hat{\rt}^{(N)}(x), \ \forall x\in\T \big] \longrightarrow 1,
	\qquad
	\text{as } N\to\infty.
\end{align*}
Given this, it is standard to verify that Assumptions~\ref{assu:rt}\ref{assu:rt:holder}--\ref{assu:rt:limit}
hold for $ \urt\in(0,\alpha) $ and $ \Rtlim=-\frac12 B^\alpha $.
\end{example}

\begin{example}[Alternating]
\label{ex:alt}
Fix any $ \delta>0 $ and let $ \rt^{(N)}(x)=N^{-\delta} $ for $ x=0,2,4,\ldots,N-2 $ and $ \rt^{(N)}(x)=-N^{-\delta} $ for $ x=1,3,\ldots,N-1 $.
It is readily verified that Assumptions~\ref{assu:rt}\ref{assu:rt:bdd}--\ref{assu:rt:limit}
hold for $ \urt\in(0,\delta] $ and $ \Rtlim\equiv 0 $.
\end{example}

Roughly speaking, our main result asserts that,
for inhomogeneous \ac{ASEP} under Assumption~\ref{assu:rt}, $ \scZ_N(t,x) $ (defined via \eqref{eq:Z} and \eqref{eq:was}) converges in distribution to the solution of the following \ac{SPDE}:
\begin{align}
	\label{eq:spde}
	\partial_t \limZ = \Ham \limZ + \xi \limZ,
	\qquad
	\Ham := \tfrac12 \partial_{xx} + \Rtlim'(x).
\end{align}

To state our result precisely,
we first recall a result from~\cite{fukushima77} on the Schr\"{o}dinger operator with a rough potential.
It is shown therein that, for any bounded Borel function $ f:[0,1]\to\R $,
the expression $ \frac12 \partial_{xx} + f'(x) $ defines a self-adjoint operator on $ L^{2}[0,1] $ with Dirichlet boundary conditions.
This construction readily generalizes to $ \limT $ (i.e., $ [0,1] $ with periodic boundary condition) considered here.
In Section~\ref{sect:Sg}, for given $ \Rtlim\in C^{\urt}[0,1] $,
we construct the semigroup $ \Sg(t)=e^{t\Ham} $ by giving an explicit formula for the kernel $ \Sg(t;x,\tilx) $.
We say that a $ C([0,\infty),C(\limT)) $-valued process $ \limZ $ is a \textbf{mild solution} of~\eqref{eq:spde} with initial condition $ \limZ^\ic \in C(\limT) $, if
\begin{align}
	\label{eq:spde:mild}
	\limZ(t,x) = \int_{\limT} \Sg(t;x,\tilx) \limZ^\ic(\tilx)d\tilx + \int_0^t \int_{\limT} \Sg(t-s;x,\tilx) \limZ(s,\tilx) \xi(s,\tilx) dsd\tilx.
\end{align}

\begin{remark}
\label{rmk:Rtfixed}
In~\eqref{eq:spde:mild}, $ \Sg(t;x,\tilx) $ is taken to be independent of the driving noise $ \xi $.
This being the case, throughout this article, for the analysis that involves the limiting \ac{SPDE}~\eqref{eq:spde}--\eqref{eq:spde:mild},
we will assume without loss of generality that $ \Sg(t;x,\tilx) $ is deterministic,
and interpret the stochastic integral $ \int(\ldots)\xi(s,x) dsdx $ in the It\^{o} sense.
\end{remark}
\noindent
Using standard Picard iteration,
we show in Proposition~\ref{prop:unique} that~\eqref{eq:spde:mild} admits at most one solution for a given $ \limZ^\ic \in C(\limT) $.
Existence follows from our result~Theorem~\ref{thm:main} in the following.

Fix $ \uic>0 $. Throughout this article we will also fix a sequence of
\emph{deterministic} initial conditions for the \ac{ASEP} height function $ \big\{ h^{\ic,(N)}(\Cdot) \big\}_N $,
and let $ Z^{\ic,(N)}(x) $ be defined in terms of $ h^{\ic,(N)} $ via \eqref{eq:Z} and \eqref{eq:was} at $t=0$.
We make the assumption that the initial conditions are \textbf{near stationary}. This is easiest stated in terms of $ Z^{\ic,(N)}$ and posits that there exists a finite constant $ c<\infty $ such that, with the shorthand notation $ \scZ^{\ic}:=\scZ^{\ic,(N)} $,
\begin{align}
	\label{eq:nearst}
	\scZ^\ic(x)
	\le c,
	\qquad
	| \scZ^\ic(x)-\scZ^\ic(x')|
	\le c \ \big(\tfrac{\dist_{\T}(x,x')}{N}\big)^{\uic},
	\qquad
	\forall x,x'\in\T,
	\
	N\in\Z_{>0}.
\end{align}
Recall the scaled process $ Z_N(t,x) $ from~\eqref{eq:was}, and similarly scale $ Z^\ic_N(x) := Z^{\ic,(N)}(xN) $.
We linearly interpolate the process $ \scZ_N(t,x) $ in $ x $ so that it is $ D([0,\infty),C(\limT)) $-valued.
We endow the space $ C(\limT) $ with the uniform norm $ \norm{\,\Cdot\,}_{C(\limT)} $ (and hence uniform topology),
and, for each $ T<\infty $, endow the space $ D([0,T],C(\limT)) $ with Skorohod's $ J_1 $-topology.
We use $ \Rightarrow $ to denote weak convergence of probability laws.
Our main result is the following:
\begin{theorem}
\label{thm:main}
Consider a half-filled inhomogeneous \ac{ASEP} on $ \T^{(N)}$,
with deterministic, near stationary initial condition described as in the preceding.
If, for some $ \limZ^\ic\in C(\limT) $,
\begin{align*}
	\norm{\scZ^\ic_N - \limZ^\ic}_{C(\limT)} \longrightarrow 0,
	\qquad
	\text{as }N\to\infty,
\end{align*}
then, under the scaling~\eqref{eq:was},
\begin{align*}
	\scZ_N \Longrightarrow \limZ \text{ in } D([0,T],C(\limT)),
	\qquad
	\text{as }N\to\infty,
\end{align*}
for each $ T<\infty $, where $ \limZ $ is the mild solution of~\eqref{eq:spde} with initial condition $ \limZ^\ic $.
\end{theorem}

\begin{remark}
Though we formulate all of our results at the level of \ac{SHE}-type equations, they can also be interpreted in terms of convergence of the \ac{ASEP} height function (under suitable centering and scaling) to a \ac{KPZ}-type equation which formally is written as
\begin{align}\label{KPZtypeformal}
	\partial_t \mathcal{H}(t,x)
	=
	\tfrac{1}{2}\partial_{xx} \mathcal{H}(t,x) - \tfrac{1}{2} \big(\partial_x \mathcal{H}(t,x)\big)^2 + \xi(t,x) - \Rtlim'(x).
\end{align}
The solution to this equation should be (as in the case where $\Rtlim'(x)\equiv 0$) defined via $\mathcal{H}(t,x) = -\log \limZ(t,x)$. One could also try to prove well-posedness of this inhomogeneous \ac{KPZ} equation directly, though this is outside the scope of our present investigation and unnecessary for our aim.
\end{remark}

\subsection*{Steps in the proof of Theorem \ref{thm:main}}

Given that Theorem~\ref{thm:main} concerns convergence at the level of $ \limZ $,
our proof naturally goes through the microscopic transform~\eqref{eq:Z}.
As mentioned earlier, for \emph{homogeneous} \ac{ASEP}, $ Z $ solves the microscopic \ac{SHE} \eqref{eq:dshe}.
On the other hand, with the presence of inhomogeneities, it was not clear at all that G\"{a}rtner's transform applies.
As noted in \cite[Remark~4.5]{borodin14}, transforms of the type~\eqref{eq:Z} are tied up with the Markov duality.
The inhomogeneous \ac{ASEP} considered here lacks a certain type of Markov duality
so that one cannot infer a useful transform from Markov duality. For specifically, referring to the notation in Remark~\ref{rmk:rt} and \cite{borodin14},
the inhomogeneous \ac{ASEP} does enjoy a Markov duality for the observable $ \til{Q}(t,\vec{x})$ (which is essentially $\eta(t,x)Z(t,x)$ in our notation), but not for $ Q(t,\vec{x}) $ (which is essentially $Z(t,x)$ in our notation).
The latter is crucial for inferring a transform of the type~\eqref{eq:Z}.

The first step of the proof is to observe that, despite the (partial) lost of Markov duality,
$ Z $ still solves an \ac{SHE}-type equation (\eqref{eq:Lang} in the following),
with two significant changes compared to~\eqref{eq:dshe}.
\begin{enumerate}[label=\roman*),leftmargin=30pt]
\item \label{enu:Bouchaud}
First the discrete Laplacian is now replaced by the generator of an inhomogeneous random walk.
Interestingly, this walk is exactly Bouchaud's model \cite{bouchaud92},
which is often studied with heavy-tail $ \rtt(x) $ (as opposed Assumption~\ref{assu:rt}) in the context of randomly trapped walks.
\item \label{enu:pontential}
Additionally, a potential term (the term $ \nu \rt(x)Z(t,x)dt $ in~\eqref{eq:Lang}) appears due to the unevenness of quenched expected growth.
For homogeneous \ac{ASEP} with near stationary initial condition, the height function grows at a constant expected speed,
and the term $ e^{\nu t} $ in~\eqref{eq:Z} is in place to balance such a constant growth.
Due to the presence of the inhomogeneity, in our case the quenched expected growth is no longer a constant and varies among sites.
This results in a fluctuating potential that acts on $ Z(t,x) $.
\end{enumerate}
The two terms in~\ref{enu:Bouchaud}--\ref{enu:pontential} together
make up an operator $ \ham $ (defined in~\eqref{eq:Lang}) of Hill-type that governs the microscopic equation.
Correspondingly, the semigroup $ \sg(t):=e^{t\ham} $ now plays the role of standard heat kernel in the case of homogeneous \ac{ASEP}.
We refer to $ \sg(t):=e^{t\ham} $ and its continuum analog $ \Sg(t) $ as \ac{PAM} semigroups.

Much of our analysis consists of estimating the kernels of $ \sg(t) $ and $ \Sg(t) $.
These estimates are crucial in order to adapt and significantly extend the core argument of \cite{bertini97}.
To prepare for the analysis of $ \Sg(t) $, in Section~\ref{sect:hk},
we establish bounds on the transition kernel of Bouchaud's walk, which was described \ref{enu:Bouchaud},
and show that the kernel is well-approximated by that of the simple random walk.
The transition kernel of Bouchaud’s walk has been studied with heavy-tail inhomogeneities in the context of trapping  models ---
see the beginning of Section~\ref{sect:hk} for more discussions of this literature.
Here we consider \emph{bounded} and \emph{vanishing} inhomogeneiiesy,
which is technically much simpler to dealt with than heavy-tail ones.
Given the vast literature (some of which is surveyed at the beginning of Section~\ref{sect:hk}) on heat kernel estimates,
it is quite possibly that some of the bounds (Proposition~\ref{prop:hk}\ref{cor:hk:hkasup}--\ref{cor:hk:hkahold::})
obtained in this paper could be derived from, or follow similarly from existing techniques.
However, for the sake of being self-contained, we provide with an elementary and short derivation via Picard iteration of the heat kernel bounds we use.

Based on the bounds in Section~\ref{sect:hk} on the kernel of Bouchaud's walk,
in Section~\ref{sect:Sgsg}, we express the PAM semigroup $ \sg(t) $ using the Feynman--Kac formula,
with Bouchaud's walk being the underlying measure.
We then expand the Feynman--Kac formula, and develop techniques to bound the resulting expansion
to obtain estimates on the kernel of $ \sg(t) $ and its continuum analog.
For the operator $ \frac12\Delta-V $ with a singular potential $ V $,
accessing the corresponding semigroup has been a classic subject of study in mathematical physics.
In particular, semigroup kernels and the Feynman--Kac formulas have been studied in \cite{mckean77,simon82},
and the expansion we use is similar to the one considered in \cite[Section 14, Chapter V]{simon79}.
The major difference is that our potential $ \Rtlim' $ has negative regularity, and hence is not function-valued.

Armed with the heat kernel bounds from  Section~\ref{sect:hk} and the semigroup kernel expansion from Section~\ref{sect:Sgsg}, the final two sections adapt the key ideas from the work of Bertini and Giacomin \cite{bertini97} into the inhomogeneous setting to prove tightness (Section \ref{sect:mom}) and identify the limiting \ac{SPDE} (Section \ref{sect:pfmain}).
\subsection*{Further directions}

There are a number of directions involving inhomogeneous \ac{ASEP} which could warrant further investigation. At a very basic level, in this article we limit our scope to half-filled systems on the torus with near stationary initial conditions so as to simplify the analysis. However, we expect similar results should be provable via our methods when one relaxes these conditions. Putting aside the weak asymmetry scaling, it is compelling to consider the nature of the long-time hydrodynamic limit (i.e., functional law of large numbers) or fluctuations (i.e., central limit type theorems) for inhomogeneous \ac{ASEP}. For the homogeneous \ac{ASEP} the hydrodynamic limit is dictated by Hamilton-Jacobi PDEs and the fluctuations characterized by the \ac{KPZ} universality class --- does any of this survive the introduction of inhomogeneities?
These questions are complicated by the lack of an explicit invariant measure for our inhomogeneous \ac{ASEP}, as well as a lack of any apparent exact solvability.

There are other types of inhomogeneities which can be introduced into \ac{ASEP} and it is natural to consider whether different choices lead to similar long-time scaling limits or demonstrate different behaviors. Our choice of inhomogeneities stemmed from the fact that upon applying G\"{a}rtner's transform, it results in an \ac{SHE}-type equation.
On the other hand, our methods seem not to apply to site (instead of bond) inhomogeneities (so out of $ x $ we have $\ell\rtt(x)$ and $r \rtt(x) $ as rates). It also does not seem to extend to non-nearest neighbor systems. A more direct approach (e.g. regularity structures \cite{Hairer13}; energy solutions \cite{GJ2014a, GP2015a}; paracontrolled distributions \cite{GIP15,GP17}; or renormalization group \cite{Kupiainen16}) at the level of the height function may eventually prove useful in dealing with these generalizations. For the homogeneous non-nearest neighbor ASEP (or other homogeneous generalizations which maintain a product invariant measure), the energy solution method  has proved quite useful for demonstrating \ac{KPZ} equation limit results --- see, for example, \cite{GJ14}.

Another type of inhomogeneities would involve jump out of $ x $ given by rates $ \rtt(x) +b $ to the left and $ \rtt(x)-b $ to the right.
A special case of this type of inhomogeneities is studied in \cite{franco16} where they consider a single slow bond (i.e, $ \rtt(x)\equiv \rtt_* $ for $ x\neq 0 $ and $ \rtt(0)<\rtt_* $). In that case, they show that the inhomogeneities preserves the product Bernoulli invariant measure (note that the inhomogeneities we consider do not preserve this property).
(In fact, the argument in \cite{franco16} for this preservation of the invariant measure may be generalized to more than just a single-site inhomogeneity.) 
Using energy solution methods, \cite{franco16} shows that depending on the strength of the asymmetry and the slow bond, one either obtains a Gaussian limit with a possible effect of the slow bond,
or the \ac{KPZ} equation without the effect of the slow bond. It would be interesting to see if this type of inhomogeneities (at every bond, not just restricted to a single site) could lead to a similar sort of \ac{KPZ} equation with inhomogeneous spatial noise such as derived herein.

%
%
\cite{covert97,rolla08,calder15} characterized the hydrodynamic limit for \ac{ASEP} and TASEP with inhomogeneities that varies at a \emph{macroscopic} scale. Those methods do not seem amenable to rough or rapidly varying parameters (such as the i.i.d.\ or other examples considered herein) and it would be interesting to determine their effect.
%
A special case of spatial inhomogeneities is to have a slow bond at the origin.
The slow bond problem is traditionally considered for the TASEP,
with particular interest in how the strength of slow-down affects the hydrodynamic limit of the flux,
see \cite{janowsky92,basu14} and the reference therein.
As mentioned previously, this problem has been further considered in the context of weakly asymmetric \ac{ASEP} in \cite{franco16}.
There are other studies of TASEP (or equivalently last passage percolation) with inhomogeneitues in
\cite{gravner02, gravner02a, lin12, emrah15, emrah16, borodin17}.
The type of inhomogeneities in those works are of a rather different nature than those considered here.
In terms of TASEP, their inhomogeneities mean that the $i^{th}$ jump of the $j^{th}$ particle occur at rate $\pi_i+\hat{\pi}_j$ for the inhomogeneity parameters $\{\pi_i\}$ and $\{\hat{\pi}_j\}$. Such inhomogeneities do not seem to result in a temporally constant (but spatially varying) noise in the limit. Thus, the exact methods which are applicable in those works do not seem likely to grant access to the fluctuations or phenomena surrounding our inhomogeneous process or limiting equation.

As mentioned previously, upon applying G\"{a}rtner's transform we obtain an \ac{SHE}-type equation with the generator of Bouchaud's walk.
Our particular result involves tuning the waiting time rate near unity,
and under such scaling the inhomogeneous walk approximates the standard random walk.
On the other hand, Bouchaud's model (introduced in \cite{bouchaud92} in relation to aging in disordered systems; see also \cite{benarous06,benarous15} for example)
is often studied under the assumption of heavy-tailed waiting parameters. In such a regime, one expects to see the effect of trapping, and in particular the FIN diffusion \cite{fontes99} is a scaling limit that exhibits the trapping effect. It would be interesting to consider a scaling limit of inhomogeneous \ac{ASEP} in which the FIN diffusion arises. See the beginning of Section \ref{sect:hk} for further discussion and references related to Bouchaud's model.

For the case $ \Rtlim'(x)=B'(x) $ (spatial white noise),
the operator $ \Ham $ (in~\eqref{eq:spde}) that goes into the \ac{SPDE}~\eqref{eq:spde} is known as Hill's operator.
There has been much interest in the spectral properties of this and similar random Schr\"{o}dinger type operator.
In particular, \cite{frisch60, halperin65, fukushima77, mckean94, cambronero99, cambronero06} studied the ground state energy in great depth,
and recently, \cite{dumaz17} proved results on the point process for lowest few energies, as well as the localization of the eigenfunctions.
On the other hand, the semigroup $ \Sg(t) := e^{t\Ham} $ is the solution operator of the (continuum) \ac{PAM} (see \cite{carmona94, koenig16} and the references therein for extensive discussion on the discrete and continuum  \ac{PAM}). A compelling challenge is to understand how this spectral information translates into the long-time behavior of our \ac{SPDE}. For instance, what can be said about the intermittency of this \ac{SPDE}?

\subsection*{Outline}
In Section~\ref{sect:hc}, we derive the microscopic (\ac{SHE}-type) equation for $ Z(t,x) $.
As seen therein, the equation is governed by a Hill-type operator $ \ham $
that involves the generator of an (Bouchaud-type) inhomogeneous walk.
Subsequently, in Sections~\ref{sect:hk} and \ref{sect:Sgsg}
we develop the necessary estimates on the transition kernel of the inhomogeneous walk and Hill-type operator.
Given these estimates, we proceed to prove Theorem~\ref{thm:main} in two steps:
by first establishing tightness of $ \{\scZ_N\}_N $ and then characterizing its limit point.
Tightness is settle in Section~\ref{sect:mom} via moment bounds.
To characterizes the limit point, in Section~\ref{sect:pfmain},
we develop the corresponding martingale problem,
and prove that the process $ Z_N(t,x) $ solves the martingale problem.

\subsection*{Acknowledgment}
We thank anonymous referees for their very useful comments that help us improving the presentation and historical treatment of results in this article.
During the conference
`Stochastic Analysis, Random Fields and Integrable Probability,
The 12th Mathematical Society of Japan, Seasonal Institute',
Takashi Kumagai generously walked us through references on heat kernel estimates in related settings.
We appreciate Kumagai's help and the hospitality of the conference organizers.

We thank Yu Gu and Hao Shen for useful discussions during the writing of the first manuscript,
and particularly acknowledge Hao Shen for pointing to us the argument in \cite[Proof of Proposition~3.8]{labbe17}.
Ivan Corwin was partially supported by the Packard Fellowship for Science and Engineering, and by the NSF through DMS-1811143 and DMS-1664650.
Li-Cheng Tsai was partially supported by the Simons Foundation through a Junior Fellowship and by the NSF through DMS-1712575.

\subsection*{Notation}
We use $ c(u,v,\ldots)\in (0,\infty) $ to denote a generic, positive, finite, deterministic constant,
that may change from line to line (or even within a line), but depends only on the designated variables $ u,v,\ldots $.
We use subscript $ N $ to denote scaled processes/spaces/functions, e.g., $ Z_N $ in~\eqref{eq:was}.
Intrinsic (not from scaling) dependence on $ N $ will be designated by superscript $ (N) $, e.g., $ \T^{(N)}$,
though, to alleviate heavy notation, we will often omit such dependence, e.g., $ \T := \T^{(N)} $.
For processes/spaces/functions that have discrete and continuum versions,
we often use the same letter but in difference fonts: math backboard for the discrete and math calligraphy for the continuum,
e.g., $ \Rt(x) $ and $ \Rtlim(x) $.
We list some of the reoccurring notation below for convenience of the readers. The left column is for discrete notation and the right is for analgous continuum notion. Some notation only occurs in one of the two contexts.

\begin{minipage}{.49\linewidth}
\begin{itemize}[leftmargin=0pt]
\item []
\item []
\item [] $\T= \Z/N\Z$: discrete torus
\item [] $ \dist_\T(\Cdot,\Cdot) $: geodesic distance on $ \T $
\item []
\item []
\item []
\item [] $ \rtt(x) $: the inhomogeneity
\item [] $ \rt(x) := \rtt(x)-1 $
\item [] $ \Rt(x,x') $: partial sum of $ \rt(x) $, see~\eqref{eq:Rt}.
\item [] $ \Rt(x) := \Rt(0,x) $
\item [] $ \Exrt[\,\Cdot \,] := \Ex[ \,\Cdot \, | \rt(x),x\in\T] $
\item [] $ \hk(t) $: random-walk semigroup on $ \T $
\item [] $ \hka(t) $: inhomogeneous random-walk semigroup on $ \T $
\item [] $ \hkr(t) := \hka(t) - \hk(t) $		
\item [] $ \ham $: the discrete \ac{PAM} operator
\item [] $ \sg(t) := e^{t\ham} $
\item [] $ \sgr(t) := \sg(t)-\hk(t) $
\item [] $ \sgra(t) := \sg(t)-\hka(t) $
\end{itemize}
\end{minipage}
\hfill
\begin{minipage}{.49\linewidth}
\begin{itemize}[leftmargin=5pt]
\item []
\item []
\item [] $ \limT := \R/\Z $ (continuum) torus
\item [] $ \dist_\limT(\Cdot,\Cdot) $: geodesic distance on $ \limT $
\item [] $ C[0,1],C(\limT) $: continuous functions on $ [0,1] $ and $ \limT $
\item [] $ C^u[0,1] $: $ u $-H\"{o}lder continuous functions on $ [0,1] $
\item [] $ H^k(\limT) $: the $ k $-th Sobolev space on $ \limT $
\item []
\item []
\item []
\item [] $ \Rtlim(x) $: limit of $ \Rt(Nx) $, see Assumption~\ref{assu:rt}\ref{assu:rt:limit}.
\item []
\item [] $ \HK(t) := e^{\frac12 t\Delta} $, the heat semigroup on $ \limT $
\item []
\item []
\item [] $ \Ham := \frac12\partial_{xx} + \Rtlim'(x) $, the continuum \ac{PAM} operator
\item [] $ \Sg(t) := e^{t\Ham} $
\item [] $ \Sgr(t) := \Sg(t)-\HK(t) $
\item []
\end{itemize}
\end{minipage}

\section{Microscopic Equation for $ Z(t,x) $}
\label{sect:hc}

In this section we derive the microscopic equation for $ Z(t,x) $.
In doing so, we view $ \{\rtt(x):x\in\T\} $ as being fixed (quenched),
and consider only the randomness due to the dynamics of our process. In deriving this equation for $Z(t,x)$ we will also treat $N$ as fixed and hence drop it from the notation.
The inhomogeneous \ac{ASEP} can be constructed as a continuous time Markov process
with a finite state space $ \{0,1\}^\T $, where $ \{0,1\} $ indicates whether a given sites is empty or occupied.
Here we build the inhomogeneous \ac{ASEP} out of graphical construction (see~\cite[Section~2.1.1]{liggett12,corwin12}).
For each $ x\in\T $, let $ \{\PoiR(t,x)\}_{t\geq 0} $ and $ \{\PoiL(t,x)\}_{t\geq 0} $ be independent Poisson processes of rates $r \rtt(x)$ and $\ell \rtt(x)$ respectively.
A particle attempts a jump across the bond $ (x,x+1) $ to the right (resp.\ left) whenever $ \PoiR(\Cdot,x) $ (resp.\ $ \PoiL(\Cdot,x) $) increases.
The jump is executed if the destination is empty, otherwise the particle stays put.
Let
\begin{align}
	\label{eq:filZ}
	\filZ(t) := \sigma(\PoiL(s,x),\PoiR(s,x),\rt(x): s\leq t, x\in\T)
\end{align}
denote the corresponding filtration.

Recall from~\eqref{eq:Z} that $ \tau:=\frac{r}{\ell} $.
Consider when a particle jumps from $ x $ to $ x+1 $.
Such a jump occurs only if $ \eta(t,x)(1-\eta(t,x+1))=1 $,
and, with $ Z(t,x) $ defined in~\eqref{eq:Z},
such a jump changes $ Z(t,x) $ by $ (\tau^{-1}-1)Z(t,x) $.
Likewise, a jump from $ x+1 $ to $ x $ occurs only if $ \eta(t,x+1)(1-\eta(t,x))=1 $,
and changes $ Z(t,x) $ by $ (\tau-1)Z(t,x) $.
Taking into account the continuous growth due to the term $ e^{\nu t} $ in~\eqref{eq:Z},
we have that
\begin{align}
	\notag
	dZ(t,x)
	=&
	\ \eta(t,x)(1-\eta(t,x+1)) (\tau^{-1}-1) Z(t,x)
		d \PoiR(t,x)
\\
	\label{eq:Lang1}
		&+\eta(t,x+1)(1-\eta(t,x)) (\tau-1) Z(t,x)
		d \PoiL(t,x)
	+ \nu Z(t,x) dt.
\end{align}
The differential in $ dZ(t,x) $ acts on the $ t $ variable.
We may extract the expected growth $ \rtt(x)rt $ and $ \rtt(x)\ell t $
from the Poisson processes $ \PoiR(\Cdot,x) $ and $ \PoiL(\Cdot,x) $,
so that the processes
\begin{align*}
	\PoiiR(t,x) := \PoiR(t,x) - \rtt(x)rt,
	\qquad
	\PoiiL(t,x) := \PoiL(t,x) - \rtt(x)\ell t
\end{align*}
are martingales.
We then rewrite~\eqref{eq:Lang1} as
\begin{align}
	\notag
	dZ(t,x)
	&=
	\Big(
		\rtt(x)\eta(t,x)(1-\eta(t,x+1)) (\tau^{-1}-1) r
		+\rtt(x)\eta(t,x+1)(1-\eta(t,x)) (\tau-1) \ell
		+\nu
	\Big)
	Z(t,x) dt + d\mg(t,x)
\\
	\label{eq:Lang2}
	&=
	\Big(
		\rtt(x)(\ell-r)\big( \eta(t,x)-\eta(t,x+1)\big)
		+\nu
	\Big)
	Z(t,x) dt + d\mg(t,x),
\end{align}
where $ \mg(t,x) $ is an $ \filZ $-martingale given by
\begin{align}
	\label{eq:mg}
	\mg(t,x) := \int_0^{t} Z(s,x)\Big( \eta(s,x)(1-\eta(s,x+1))(\tau^{-1}-1)d\PoiiR(s,x) + \eta(s,x+1)(1-\eta(s,x))(\tau-1)d\PoiiL(s,x) \Big).
\end{align}
Recall from~\eqref{eq:Z} that  $ \nu := 1-2\sqrt{\ell r} $.
Let $ \Delta f(x) := f(x+1)+f(x-1)-2f(x) $ denote discrete Laplacian.
By considering separately the four cases corresponding to $ (\eta(x),\eta(x+1)) \in \{0,1\}\times\{0,1\} $, it is straightforward to verify that
\begin{align*}
	( \ell-r ) \big(\eta(t,x)-\eta(t,x+1)\big) Z(t,x)
	=
	\sqrt{\ell r}\Delta Z(t,x)
	-
	\nu Z(t,x).
\end{align*}
Inserting this identity into~\eqref{eq:Lang2},
we obtain the following Langevin equation for $ Z(t,x) $:
\begin{align}
	\label{eq:Lang}
	dZ(t,x)
	&=	
	\ham Z(t,x) dt + d\mg(t,x),
\\
	\label{eq:ham}
	\ham &:= \sqrt{r\ell} \, \rtt(x)\Delta - \nu \, \rt(x).
\end{align}
Recall that $ \Rt(x) := \Rt(x,0) $.
From~\eqref{eq:Rt} we have $ \rt(x) = -2 \,(\Rt(x)-\Rt(x-1)) $,
and, under weak asymmetry scaling~\eqref{eq:was}, we have $ \nu = \frac1{2N} + O(N^{-2}) $.
We hence expect (and will justify) that $ \ham $ behaves like $ \Ham=\frac12\partial_{xx} + \Rtlim'(x) $.
This explains why $ \Ham $ appears in the limiting equation~\eqref{eq:spde}.
For~\eqref{eq:spde} to be the limit of~\eqref{eq:Lang},
the martingale increment $ d\mg(t,x) $ should behave like $ \xi \limZ $.
To see why this should be true, let us calculate the quadratic variation of $ \mg(t,x) $.
The collection of compensated Poisson processes $\big\{\PoiiR(\Cdot,x), \PoiiL(\Cdot,x)\big\}_{x\in \T}$ are all independent of each other. Thus, from~\eqref{eq:mg}, we have that
\begin{align}
	\notag
	&d\langle \mg(t,x), \mg(t,\tilx) \rangle\\
	\notag &\quad= \ind_\set{x=\tilx}Z^2(t,x)
	\Big( \eta(t,x)(1-\eta(t,x+1))(\tau^{-1}-1)^2\rtt(x)r + \eta(t,x+1)(1-\eta(t,x))(\tau-1)^{2}\rtt(x)\ell \Big)dt
\\
	\label{eq:qv}
	&\quad= \ind_\set{x=\tilx}Z^2(t,x)
	(r-\ell)^2 \rtt(x) \Big( \tfrac{1}{\ell}\eta(t,x) + \tfrac{1}{r}\eta(t,x+1) - \big(\tfrac1r+\tfrac1\ell\big)\eta(t,x)\eta(t,x+1)) \Big)dt,
\end{align}
where $ \ind_A(\Cdot) $ denotes the indicator function of a given set $ A $.
Under the weak asymmetry scaling~\eqref{eq:was}, $ (r-\ell)^2 = \frac1N + O(N^{-2}) $
acts as the relevant scaling factor for the quadratic variation.
In addition to this scaling factor, we should also consider the quantities that involve $ \eta(t,x) $ and $ \eta(t,x+1) $.
Informally speaking, since the system is half-filled (i.e., having $ N/2 $ particles),
we expect $ \eta(t,x) $ and $ \eta(t,x+1) $ to self-average (in $ t $) to $ \frac12 $,
and expect $ \eta(t,x)\eta(t,x+1) $ to self-average to $ \frac14 $.
With $ r,\ell\to\frac12 $ and $ \rtt(x) \to 1 $,
we expect $ d\langle \mg(t,x), \mg(t,\tilx) \rangle $ to behave like $ N^{-1} \ind_\set{x=\tilx}Z^2(t,x) dt $,
and hence $ d\mg(t,x) $ to behaves like $ \xi\limZ $, as $ N\to\infty $.

Equation~\eqref{eq:Lang} gives the microscopic equation in differential form.
For subsequent analysis, it is more convenient to work with the integrated equation.
Consider the semigroup $ \sg(t):=e^{t\ham} $,
which is well-defined and has kernel $ \sg(t;x,\tilx) $
because $ \ham $ acts on the space $ \{f:\T\to \R\} $ of \emph{finite} dimensions.
Using Duhamel's principle in~\eqref{eq:Lang} gives
\begin{align}
	\label{eq:Lang:int}
	Z(t,x) = \sum_{\tilx\in\T}\sg(t;x,\tilx)Z_\ic(\tilx) + \int_0^t \sum_{\tilx\in\T}\sg(t-s;x,\tilx)d\mg(s,\tilx).
\end{align}
More generally, initiating the process from time $ t_* \geq 0 $ instead of $ 0 $, we have
\begin{align}
	\label{eq:Lang:int:}
	Z(t,x) = \sum_{\tilx\in\T}\sg(t-t_*;x,\tilx)Z(t_*,\tilx) + \int_{t_*}^t \sum_{\tilx\in\T}\sg(t-s;x,\tilx)d\mg(s,\tilx),
	\qquad
	t \geq t_*.
\end{align}
The semigroup $ \sg(t) $ admits an expression by the Feynman--Kac formula as
\begin{align}
	\label{eq:feynmankac}
	\big( \sg(t)f \big)(s)
	=
	\Ex_x\Big[ e^{\int_0^t \nu\rt(X^\rt(s))ds} f(X^\rt(t)) \Big].
\end{align}
Hereafter $ \Ex_x[\,\Cdot\,] $ (and similarly $ \Pr_x[\,\Cdot\,] $) denotes expectation with respect to a reference process starting at $ x $.
Here the reference process $ X^\rt(t) $ is a walk on $ \T $
that attempts jumps from $ X^\rt(t) $ to $ X^\rt(t)\pm 1 $ in continuous time (each) at rate $ \sqrt{r\ell} \,\rtt(X^\rt(t)) $.

\begin{remark}
It is natural to wonder whether it is possible to directly (without appealing to the G\"{a}rtner transform and \ac{SHE}-type equation) see convergence of the height function to the \ac{KPZ}-type equation \eqref{KPZtypeformal}. While a direct proof is certainly beyond the scope of this paper, we will briefly explain at a very heuristic level where the structure of the \ac{KPZ}-type equation \eqref{KPZtypeformal} arises from the microscopic evolution of the height function.

There are two changes that can occur for the ASEP height function $h(t,x)$. Using the notation $\nabla f(x)=f(x+1)-f(x)$, we have that $h(t,x)$ can increase by $2$ at rate $\ell\rtt(x)$ provided that $\nabla h(t,x)=1$ and $\nabla h(t,x-1)=-1$; and  $h(t,x)$ can decrease by $2$ at rate $r\rtt(x)$ provided that $\nabla h(t,x)=-1$ and $\nabla h(t,x-1)=1$. Since $\nabla h$ only takes values in $\{-1,1\}$ we can encode the indicator functions as linear functions of $\nabla h$. This gives rise to the following evolution equation:
$$
dh(t,x) =  \Big( 2\ell \rtt(x) \frac{1+\nabla h(t,x)}{2}\frac{1-\nabla h(t,x-1)}{2} -  2r \rtt(x) \frac{1-\nabla h(t,x)}{2}\frac{1+\nabla h(t,x-1)}{2} \Big) dt + d\til{M}(t,x),
$$
where $ \til{M}(t,x)$ is an explicit martingale. Recalling that $\Delta f(x) = \nabla f(x)-\nabla f(x-1)$, we can rewrite the above in the suggestive form
$$
dh(t,x) =  \Big( \frac{(\ell +r)\rtt(x)}{2} \Delta h(t,x) -  \frac{(\ell-r)\rtt(x)}{2} \nabla h(t,x)\nabla h(t,x-1) +  \frac{(\ell-r)\rtt(x)}{2} \Big) dt + d\til{M}(t,x).
$$
While it is still highly non-trivial to prove convergence (with appropriate renormalization) of these terms to those in \eqref{KPZtypeformal}, it is now apparent that the new spatial noise $\Rtlim'(x)$ comes from the term $\frac{(\ell-r)\rtt(x)}{2}$.
\end{remark}

\section{Transition Probability of the Inhomogeneous Walk $ X^\rt(t) $}
\label{sect:hk}
The focus of this section and Section~\ref{sect:Sgsg} is to control the semigroup $ \sg(t) $ and its continuum counterpart $ \Sg(t) $.
As the first step, in this section we establish estimates on the transition kernel
\begin{align}
	\label{eq:hka}
	\hka(t;x,\tilx) := \Pr_{x}\big[ X^\rt(t)=\tilx \big]
\end{align}
of the inhomogeneous walk $ X^\rt(t) $.
Note that $ X^\rt(t) $ and $ \hka $ depend on $ N $ through $ \rt(x)=\rt^{(N)}(x) $ and through the underlying torus $ \T=\T^{(N)} $,
but we omit such dependence in the notation.

The kernel $ \hka $ has been studied with heavy-tail $ \rt(x) $ in the context of trapped models.
For heavy-tail $ \rt(x) $, \cite[Lemma~3.1]{cerny06} obtained bounds on $ \hka(t;x,x) $,
and \cite[Lemma~3.2]{cerny06} and \cite{cabezas15} demonstrated stretched exponential tails in large deviations of $ X^\rt(x) $,
which confirmed a prediction \cite{bertin03} based on the trapping nature of Bouchaud's walk.
Estimates on the analogous continuum kernel (i.e., for the FIN diffusion) have been obtained \cite{croydon19}.
As mentioned previously,
here we consider bounded and vanishing $ \rt(x) $, which is technically much simpler than heavy-tailed $ \rt(x) $. These bounds mentioned above essentially imply Proposition \ref{prop:hk}(a).

Due to the technical nature of how $ \hka $ enters our subsequent analysis, we need detailed bounds.
Specifically, we will derive in Proposition~\ref{prop:hk} bounds on $ \hka(t,\Cdot) $, its H\"{o}lder continuity, its gradients, and its difference between the homogeneous walk kernel.
Put in a broader context, the type of bounds we seek to obtain on $ \hka(t,\Cdot) $ and its H\"{o}lder continuity
go under the name of Nash--Aronson bounds \cite{nash58,aronson67} in elliptic PDEs (we thank one of the anonymous referees for pointing us to this literature).
These type of bounds have since been pursued and generalized in various contexts
such as Riemannian manifolds, metric measure spaces, and fractals.
We point to \cite{grigor92,saloffcoste02} and the references therein.
On Riemannian manifolds, bounds on the gradient of heat kernels have been derived in, for example, \cite{cheng81,li86}.
Modern works in probability have been investigating the analogous heat kernel in discrete settings.
For the random conductance model,
heat kernel estimates have been derived at various generality:
for bounded below conductance in \cite{barlow10},
and for general conductance with certain integrability assumptions in \cite{folz11,andres15,andres19}.
It was shown in \cite{berger08,biskup11} that anomalies in heat-kernel decay may occur without integrability assumptions.
The works \cite{deuschel19,deuschel19a} consider layered random conductance models and establish kernel estimates.
For an overview on the random conductance model we point to~\cite{biskup11}.
We also point out that a gradient estimates on Green's function of percolation clusters
has been obtained in~\cite[Theorem~6]{benjamini15}.

The starting point of our analysis is the backward Kolmogorov equation
\begin{align}
	\label{eq:bkol:}
	\partial_t \hka(t;x,\tilx) = \sqrt{r\ell} \, \rtt(x) \Delta_{x} \hka(t;x,\tilx),
	\qquad
	\hka(0;x,\tilx) = \ind_\set{\tilx}(x),
\end{align}
where $ \ind_A(\Cdot) $ denotes the indicator function of a given set $ A $.
Under the scaling~\eqref{eq:was}, we have $ \sqrt{r\ell}\to\frac12 $ as $ N\to\infty $.
Indeed, the coefficient $ \sqrt{r\ell} $ can be scaled to $ \frac12 $ by a change-of-variable $ t\mapsto 2\sqrt{\ell r}t $,
so without loss of generality, we alter the coefficient $ \sqrt{r\ell} $ in~\eqref{eq:bkol:} and consider
\begin{align}
	\tag{\ref*{eq:bkol:}'}
	\label{eq:bkol}
	\partial_t \hka(t;x,\tilx) = \tfrac12 \rtt(x) \Delta_{x} \hka(t;x,\tilx),
	\qquad
	\hka(0;x,\tilx) = \ind_\set{\tilx}(x).
\end{align}

As announced previously, we use $ c(u,v,\ldots)\in (0,\infty) $ to denote a generic, positive, finite, deterministic constant,
that may change from line to line (or even within a line), but depends only on the designated variables $ u,v,\ldots $.

Recall that $ \rtt(x)=1+\rt(x) $.
Our strategy of analyzing $ \hka $ is perturbative.
We solve~\eqref{eq:bkol} iteratively, viewing $ \rt(x) $ as a perturbation.
Such an iteration scheme begins with the unperturbed equation
\begin{align}
	\label{eq:lhe}
	\partial_t \hk(t;x,\tilx) = \tfrac12 \Delta_{x} \hk(t;x,\tilx),
	\qquad
	\hk(0;x,\tilx) = \ind_\set{\tilx}(x),
\end{align}
which is solved by the transition probability $ \hk(t;x,\tilx) = \Pr_{x}[ X(t)=\tilx ] $
of the continuous time symmetric simple random walk $ X(t) $ on $ \T $.
Here, we record some useful bounds on $ \hk $.
Let $ \nabla f(x) := f(x+1)-f(x) $ denote the forward discrete gradient.
When needed we write $ \nabla_x $ or $ \Delta_x $ to highlight which variable the operator acts on.
Given any $ u\in(0,1] $ and $ T<\infty $,
\begin{subequations}
\label{eq:hk}
\begin{align}
	\label{eq:hk:sup}
	|\hk(t;x,\tilx)| &\leq \frac{c(T)}{\sqrt{t+1}},
\\
	\label{eq:hkgd}
	|\hk(t;x,\tilx)-\hk(t,x',\tilx)| & \leq c(u,T) \frac{\dist_{\T}(x,x')^{u}}{(t+1)^{(1+u)/2}},
\\
	\label{eq:hkgd:sum}
	\sum_{\tilx\in\T} |\hk(t;x,\tilx)-\hk(t;x',\tilx)| & \leq c(u,T) \frac{\dist_{\T}(x,x')^{u}}{(t+1)^{u/2}},
\\
	\label{eq:hk:lap:sum}
	\sum_{\tilx\in\T} |\Delta_{x}\hk(t;x,\tilx)| &\leq \frac{c(T)}{t+1},
\\
	\label{eq:hk:lap:sum:}
	\sum_{x\in\T} |\Delta_{x}\hk(t;x,\tilx)| &\leq \frac{c(T)}{t+1},
\\
	\label{eq:hk:hold:sup}
	|\hk(t;x,\tilx)|\dist_{\T}(x,\tilx)^{u} &\leq c(T)(t+1)^{-(1-u)/2},
\\
	\label{eq:hk:hold}
	\sum_{\tilx\in\T}|\hk(t;x,\tilx)|\dist_{\T}(x,\tilx)^{u} &\leq c(T)(t+1)^{u/2},
\\
	\label{eq:hk:hold:}
	\sum_{\tilx\in\T}|\nabla_x\hk(t;x,\tilx)|\dist_{\T}(x,\tilx)^{u}
	 & \leq \frac{c(u,T)}{(1+t)^{(1-u)/2}},	 
\\
	\label{eq:hk:hold::}
	\sum_{\tilx\in\T}|\nabla_{\tilx}\hk(t;x,\tilx)|\dist_{\T}(x,\tilx)^{u}
	& \leq \frac{c(u,T)}{(1+t)^{(1-u)/2}},
\end{align}
\end{subequations}
for all $ x,x',\tilx\in\T $ and $ t\leq N^2T $.
These bounds~\eqref{eq:hk:sup}--\eqref{eq:hk:hold::} follow directly from known results on the analogous kernel on $ \Z $.
Indeed, with $ \hk^\Z(t;x-\tilx):=\Pr_{x}[ X^\Z(t)=\tilx ] $ denoting the transition kernel of
continuous time symmetric simple random walk $ X^\Z(t) $ on the full-line $ \Z $, we have
\begin{align}
	\label{eq:hk:hkZ}
	\hk(t;x,\tilx) = \sum_{i\in\Z} \hk^\Z(t;x-\tilx+iN),
	\qquad
	x,\tilx \in \{0,\ldots,N-1\}.
\end{align}
The full-line kernel $ \hk^\Z $ can be analyzed by standard Fourier analysis,
as in, e.g., \cite[Equation (A.11)-(A.14)]{dembo16}.
Relating these known bounds on $ \hk^\Z $ to $ \hk $ gives~\eqref{eq:hk:sup}--\eqref{eq:hk:hold::}.

Now we will start to study the perturbations around the solution to  \eqref{eq:lhe}.
Let $ \Gamma(v) $ denote the Gamma function, and let
\begin{align}
	\label{eq:Sigman}
	\Sigma_n(t) := \big\{(s_0,\ldots,s_n)\in(0,\infty)^{n+1}: s_0+\ldots+s_n=t\big\}.
\end{align}
In subsequent analysis, we will make frequent use of the Dirichlet formula
\begin{align}
	\label{eq:dirichlet}
	\int_{\Sigma_n(t)}
	\prod_{i=0}^n s_i^{v_i-1} d^n\vec{s}
	=
	t^{(v_0+\ldots+v_n)-1} \frac{\prod_{i=0}^n\Gamma(v_i)}{\Gamma(v_0+\ldots+v_n)},
	\qquad
	v_0,\ldots,v_n >0.
\end{align}
Note that the constraint in~\eqref{eq:Sigman} reduces one dimension out the $ (n+1) $-dimensional variable $ (s_0,\ldots,s_n) $.
In particular, the integration in~\eqref{eq:dirichlet} is $ n $-dimension, and we adopt the notation
\begin{align}
	\label{eq:ds}
	d^n\vec{s} = (ds_1\cdots ds_n) = (ds_0ds_2\cdots ds_n)
	=
	\cdots
	=
	\prod_{i\in\{0,\ldots,n\}\setminus\set{i_0}} ds_i,
	\qquad
	i_0\in\{0,\ldots,n\}.
\end{align}

In the following we view $ \hka $ as a perturbation of $ \hk $, and set
\begin{align}
	\label{eq:hkr}
	\hkr(t;x,\tilx) := \hka(t;x,\tilx)-\hk(t;x,\tilx).
\end{align}

\begin{lemma}
\label{lem:hk}
Given any $ u,v\in(0,1] $ and $ T<\infty $,
\begin{enumerate}[label=(\alph*),leftmargin=7ex]
\item \label{lem:hka:sup} \
	$
		\displaystyle
		|\hkr(t;x,\tilx)|
		\leq
		\frac{1}{\sqrt{t+1}}
		\sum_{n=1}^{\infty} \frac{(c(v,T)N^v\norm{\rt}_{L^\infty})^n}{\Gamma(\frac{nv+1}2)},
	$
\item \label{lem:hk:sum} \
	$
		\displaystyle
		\sum_{\tilx\in\T} |\hkr(t;x,\tilx)|
		\leq
		\sum_{n=1}^{\infty} \Big(c(T)\norm{a}_{L^\infty(\T)}\log(N+1)\Big)^n,
	$
\item \label{lem:hkagd:sum} \
	$
		\displaystyle
		\sum_{\tilx\in\T}|\hkr(t;x,\tilx)-\hkr(t;x',\tilx)|
		\leq
		\frac{\dist_{\T}(x,x')^u}{(t+1)^{u/2}}
		\sum_{n=1}^{\infty} \frac{(c(u,v,T)N^v\norm{\rt}_{L^\infty})^n}{\Gamma(\frac{2-u+nv}2)},
	$
\item \label{lem:hkagd:sup} \
	$
		\displaystyle
		|\hkr(t;x,\tilx)-\hkr(t;x',\tilx)|
		\leq
		\frac{\dist_{\T}(x,x')^u}{(t+1)^{(1+u)/2}}
		\sum_{n=1}^{\infty} \frac{(c(u,v,T)N^v\norm{\rt}_{L^\infty})^n}{\Gamma(\frac{1-u+nv}2)},
	$
\end{enumerate}
for all $ x,x',\tilx\in\T $, $ t\in[0,N^{2}T] $.
\end{lemma}
\begin{proof}
The starting point of the proof is the backward Kolmogorov equation~\eqref{eq:bkol}.
We split the inhomogeneous Laplacian $ \tfrac12 \rtt(x) \Delta_x $
into $ \frac12 \Delta_x + \frac12\rt(x) \Delta_x $,
and rewrite \eqref{eq:bkol} as
\begin{align}
	\label{eq:bkol::}
	\hka(t;x,\tilx) = \hk(t;x,\tilx)
	+
	\int_{\Sigma_1(t)} \sum_{x_1\in\T} \hk(s_0;x,x_1) \frac{\rt(x_1)}{2} \Delta_{x_1} \hka(s_1;x_1,\tilx) ds_1.
\end{align}
Through Picard iteration we obtain
\begin{align}
	\label{eq:hk:chaos}
	\hkr(t;x,\tilx) = \sum_{n=1}^\infty \hkr_n(t;x,\tilx),
\end{align}
where, under the convention $ x_0:=x $ and $ x_{n+1}:=\tilx $, and the notation~\eqref{eq:Sigman} and~\eqref{eq:ds},
\begin{align}
	\label{eq:hk:chaosn}
	\hkr_n(t;x,\tilx)
	&:=
	\int_{\Sigma_n(t)} \
	\sum_{x_1,\ldots,x_n\in\T}
	\hk(s_0;x_{0},x_1)
	\prod_{i=1}^{n} \frac{\rt(x_i)}{2} \Big( \Delta_{x_i} \hk(s_i;x_{i},x_{i+1}) \Big) d^n\vec{s}.
\end{align}
Indeed, the infinite series in~\eqref{eq:hk:chaos} converges for fixed $ (t,x) $.
To see this, in~\eqref{eq:hk:chaosn}, (crudely) bound
\begin{align*}
	|\hkr_n(t;x,\tilx)|
	\leq
	N^n \norm{\tfrac12\rt}^n_{L^\infty(\T)} \big(4\norm{\hk}_{L^\infty([0,t]\times\T)}\big)^{n+1}
	\int_{\Sigma_n(t)} \ d^n\vec{s}
	\leq
	c(N,\rt,t)^n \tfrac{1}{(n+1)!}.
\end{align*}
Given the expression~\eqref{eq:hk:chaos}--\eqref{eq:hk:chaosn},
we proceed to prove the bounds~\ref{lem:hka:sup}--\ref{lem:hkagd:sup}.

\ref{lem:hka:sup}
In~\eqref{eq:hk:chaosn}, use~\eqref{eq:hk:sup} to bound $ \hk(s_0;x_0,x_1) $ by $ \frac{c}{\sqrt{s_0}} $,
and then sum over $ x_1,\ldots,x_n $ in order, using~\eqref{eq:hk:lap:sum:}.
This yields
\begin{align}
	\label{eq:hka:sup:chaosn}
	|\hkr_n(t;x,\tilx)|
	\leq
	\big( c \norm{\rt}_{L^\infty(\T)} \big)^n
	\int_{\Sigma_n(t)} \frac{1}{\sqrt{s_0}}  \prod_{i=1}^n \frac{ds_i}{s_i+1}.
\end{align}
To bound the last expression, for the given $ v\in(0,1) $,
we write $ \frac{1}{s_i+1} \leq c(v)s_i^{v/2-1} $,
and apply the Dirichlet formula~\eqref{eq:dirichlet} with $ (v_0,\ldots,v_n)=(1/2,v/2,\ldots,v/2) $ to get
\begin{align}
	\label{eq:hka:sup:chaosn:}
	|\hkr_n(t;x,\tilx)|
	\leq
	\big( c(v) \norm{\rt}_{L^\infty(\T)} \big)^n	
	\int_{\Sigma_n(t)}
	s_{0}^{-\frac12}\prod_{i=1}^n s_i^{v/2-1} ds_i
	=
	\frac{1}{\sqrt{t}} \frac{(t^{v/2}c(v)\norm{\rt}_{L^\infty(\T)})^n}{\Gamma(\frac{nv+1}{2})}.
\end{align}
Referring back to~\eqref{eq:hka:sup:chaosn},
we see that $ |\hkr_n(t;x,\tilx)| $ is bounded by $ ( c \norm{\rt}_{L^\infty(\T)})^n $ when $ t\leq 1 $,
uniformly over $ x,\tilx\in\T $.
This being the case, by making the constant $ c(v) $ larger in~\eqref{eq:hka:sup:chaosn:},
we replace the factor $ \frac{1}{\sqrt{t}} $ with $ \frac{1}{\sqrt{t+1}} $.
Since $ t^{v/2} \leq (TN^2)^{v/2} = c(v,T)N^{v} $, summing over $ n\geq 1 $ yields the desired bound.

\ref{lem:hk:sum}
Given the expansion~\eqref{eq:hk:chaos},
our goal is to bound $ \sum_{\tilx\in\T}|\hkr_n(t;x,\tilx)| $,
for $ n=1,2,\ldots $.
To this end, sum both sides of~\eqref{eq:hk:chaosn} over $ \tilx\in\T $.
Under the convention $ x_0:=x $ and the relabeling $ x_{n+1}=\tilx $, we bound
\begin{align}
	\label{eq:hkn:sum}
	\sum_{\tilx\in\T} |\hkr_n(t;x,\tilx)|
	\leq
	\norm{\rt}^n_{L^\infty(\T)}
	\int_{\Sigma_n(t)}
	\sum_{x_1,\ldots,x_{n+1}\in\T} \hk(s_0;x_0,x_1) \prod_{i=1}^n \big| \Delta_{x_i} \hk(s_i;x_{i},x_{i+1}) \big|
	d^n \vec{s}.
\end{align}
In~\eqref{eq:hkn:sum}, sum over $ x_{n+1},\ldots,x_2,x_1 $ in order,
using the bound~\eqref{eq:hk:lap:sum} for the sum over $ x_{n+1},\ldots,x_2 $
and using $ \sum_{x_1} \hk(s_0;x_0,x_1)=1 $ for the sum over $ x_1 $.
We then obtain
\begin{align*}
	\sum_{\tilx\in\T} |\hkr_n(t;x,\tilx)|
	\leq
	\big( c\norm{\rt}_{L^\infty(\T)} \big)^n
	\int_{\Sigma_n(t)} \prod_{i=1}^n \frac{ds_i}{s_i+1}.
\end{align*}
To bound the last integral, performing a change of variable $ s'_i:= ts_i $, we see that
\begin{align*}
	\sum_{\tilx\in\T} |\hkr_n(t;x,\tilx)|
	&\leq
	\big( c\norm{\rt}_{L^\infty(\T)} \big)^n\int_{\Sigma_n(1)} \prod_{i=1}^n \frac{ds_i}{s_i+t^{-1}}
	\leq
	\big( c\norm{\rt}_{L^\infty(\T)} \big)^n
	\int_{\Sigma_n(1)}
	e^{1-s_1-\ldots-s_n} \prod_{i=1}^n \frac{ds_i}{s_i+t^{-1}}
\\
	&\leq e \big( c\norm{\rt}_{L^\infty(\T)} \big)^n \prod_{i=1}^n\int_0^\infty \frac{e^{-s_i}}{s_i+t^{-1}} ds_i
	\leq
	\big( c\norm{\rt}_{L^\infty(\T)} \big)^n(1+(\log t)_+)^n.
\end{align*}
With $ t \leq N^2T $, summing both sides over $ n\geq 1 $ gives the desired result.

\ref{lem:hkagd:sum}
Taking the difference of~\eqref{eq:hk:chaosn} for $ x=x $ and $ x=x' $,
under the relabeling $ x_{n+1}=\tilx $,
here we have
\begin{align*}
	\sum_{\tilx\in\T} |\hkr_n(t;x,\tilx)-\hkr_n(t;x',\tilx)|
	\leq
	\norm{\rt}^n_{L^\infty(\T)}
	\int_{\Sigma_n(t)}
	\sum_{x_1,\ldots,x_{n+1}\in\T} |\hk(s_0;x,x_1)-\hk(s_0;x',x_1)|
	\prod_{i=1}^n \big| \Delta_{x_i} \hk(s_i;x_{i},x_{i+1}) \big|
	d^n \vec{s}.
\end{align*}
Sum over $ x_{n+1},\ldots,x_2,x_1 $ in order,
using the bound~\eqref{eq:hk:lap:sum} for the sum over $ x_{n+1},\ldots,x_2 $,
and using the bound~\eqref{eq:hkgd:sum} for the sum over $ x_1 $.
From this we obtain
\begin{align*}
	\sum_{\tilx\in\T} |\hkr_n(t;x,\tilx)-\hkr_n(t;x',\tilx)|
	\leq
	\big( c(u)\norm{\rt}_{L^\infty(\T)} \big)^n
	\int_{\Sigma_n(t)} \frac{\dist_{\T}(x,x')^u}{s^{u/2}_0}\prod_{i=1}^n \frac{ds_i}{s_i+1}.
\end{align*}
To bound the last integral,
for the given $ v\in(0,1) $, we write $ \frac{1}{s_i+1} \leq c(v)s_i^{v/2-1} $,
and apply the Dirichlet formula~\eqref{eq:dirichlet} with $ (v_0,\ldots,v_n)=(1-u/2,v/2,\ldots,v/2) $ to get
\begin{align}
	\label{eq:hkagd:sup:chaosn:}
	|\hkr_n(t;x,\tilx)-\hkr_n(t;x',\tilx)|
	&\leq
	\dist_{\T}(x,x')^u \big( c(u,v)\norm{\rt}_{L^\infty(\T)} \big)^n
	\int_{\Sigma_n(t)} s^{-u/2}_0\prod_{i=1}^n s_i^{v/2-1}ds_i
\\
	\notag
	&=
	\dist_{\T}(x,x')^u
	\frac{1}{t^{u/2}}\frac{( t^{v/2}c(u,v)\norm{\rt}_{L^\infty(\T)})^n}{\Gamma(\frac{2-u+nv}{2})}.
\end{align}
Referring back to~\eqref{eq:hka:sup:chaosn},
we see that the l.h.s.\ of~\eqref{eq:hkagd:sup:chaosn:}
is bounded by $ ( c(u) \norm{\rt}_{L^\infty(\T)})^n $ when $ t\leq 1 $, uniformly over $ x,\tilx\in\T $.
This being the case, by making the constant $ c(u,v) $ larger in~\eqref{eq:hkagd:sup:chaosn:},
we may replace the factor $ \frac{1}{t^{u/2}} $ with $ \frac{1}{(t+1)^{u/2}} $.
Since $ t^{v/2} \leq (TN^2)^{v/2} = c(v,T)N^{v} $, summing the result over $ n\geq 1 $ concludes the desired bound.

\ref{lem:hkagd:sup}
Taking the difference of~\eqref{eq:hk:chaosn} for $ x=x $ and $ x=x' $,
here we have
\begin{align*}
	|\hkr_n(t;x,\tilx)-\hkr_n(t;x',\tilx)|
	\leq
	\norm{\rt}^n_{L^\infty(\T)}
	\int_{\Sigma_n(t)}
	\sum_{x_1,\ldots,x_{n}\in\T} |\hk(s_0;x,x_1)-\hk(s_0;x',x_1)|
	\prod_{i=1}^n \big| \Delta_{x_i} \hk(s_i;x_{i},x_{i+1}) \big|
	d^n \vec{s}.
\end{align*}
Use~\eqref{eq:hkgd} to bound
the expression $ | \hk(s_0;x,x_1) - \hk(s_0,x',x_1) | $
by $ c(u)\dist_{\T}(y,y')^u (s_0)^{-\frac{1+u}2} $, and then sum over $ x_1,\ldots,x_{n} $ using~\eqref{eq:hk:lap:sum:}.
We then obtain
\begin{align*}
	|\hkr_n(t;x,\tilx)-\hkr_n(t;x',\tilx)|
	\leq
	\big( c(u) \norm{\rt}_{L^\infty(\T)} \big)^n
	\int_{\sum_{n}(t)} \frac{\dist_{\T}(x,x')^u}{s_0^{(1+u)/2}}
	\prod_{i=1}^n \frac{dt_i}{s_i+1}.
\end{align*}
To bound the last expression, for the given $ v\in(0,1) $, we write $ \frac{1}{s_i+1} \leq c(v)s_i^{v-1} $,
and apply the Dirichlet formula~\eqref{eq:dirichlet} with $ (v_0,\ldots,v_n)=((1-u)/2,v/2,\ldots,v/2) $ to get
\begin{align}
	\label{eq:hkagd:sup:chaosn::}
	|\hkr_n(t;x,\tilx)-\hkr_n(t;x',\tilx)|
	&\leq
	\dist_{\T}(x,x')^{u}
	\big(c(u,v) \norm{\rt}^n_{L^\infty(\T)} \big)^n 	
	\int_{\Sigma_n(t)}
	s_{0}^{-(1+u)/2}\prod_{i=1}^n s_i^{v/2-1} ds_i
\\
	\notag
	&=
	\frac{\dist_{\T}(x,x')^v}{t^{\frac{1+u}{2}}}
	\frac{(t^{v/2}c(u,v))^n}{\Gamma(\frac{1-u+nv}{2})}.
\end{align}
Referring back to~\eqref{eq:hka:sup:chaosn},
we see that the l.h.s.\ of~\eqref{eq:hkagd:sup:chaosn::}
is bounded by $ ( c(u) \norm{\rt}_{L^\infty(\T)})^n $ when $ t\leq 1 $, uniformly over $ x,\tilx\in\T $.
This being the case, by making the constant $ c(u,v) $ larger in~\eqref{eq:hkagd:sup:chaosn::},
we replace the factor $ \frac{1}{t^{(1+u)/2}} $.
Since $ t^{v/2} \leq (TN^2)^{v/2} = c(v,T)N^{v} $, summing the result over $ n\geq 1 $ concludes the desired bound.
\end{proof}

We now incorporate Lemma~\ref{lem:hk} with the assumed properties of $ \rt(x) $ from Assumption~\ref{assu:rt}.
To simplify notation, we will say that events $ \{\Omega_{\Lambda,N}\}_{\Lambda,N} $ hold `with probability $ \to_{\Lambda,N} 1 $' if
\begin{align}
	\label{eq:wp1}
	\lim_{\Lambda\to\infty} \liminf_{N\to\infty} \Pr\big[ \Omega_{\Lambda,N} \big] = 1.
\end{align}
\begin{proposition}\label{prop:hk}
For given  $ T<\infty $, $ u\in(0,1] $ and $ v\in(0,\urt) $,
the following hold with probability $ \to_{\Lambda,N} 1 $:
\begin{enumerate}[label=(\alph*),leftmargin=7ex]
\item \label{cor:hk:hkasup} \
	$
		\displaystyle
		|\hka(t;x,\tilx)| \leq \frac{1}{\sqrt{t+1}}\Lambda,
		\qquad
		\forall t\in[0,N^2T],
		\
		x,\tilx\in\T,
	$
\item \label{cor:hk:hkagdsup} \
	$
		\displaystyle
		|\hka(t;x,\tilx)-\hka(t;x',\tilx)| \leq \frac{\dist_{\T}(x,x')^u}{(t+1)^{(1+u)/2}}\Lambda,
		\qquad
		\forall t\in[0,N^2T],
		\
		x,x'\in\T,
	$
\item \label{cor:hk:hkagdsum} \
	$
		\displaystyle
		\sum_{\tilx\in\T}|\hka(t;x,\tilx)-\hka(t;x',\tilx)| \leq \frac{\dist_{\T}(x,x')^u}{(t+1)^{u/2}}\Lambda,
		\qquad
		\forall t\in[0,N^2T],
		\
		x,x'\in\T,
	$
\item \label{cor:hk:hkahold:sup} \
	$
		\displaystyle
		\hka(t;x,\tilx)\dist_{\T}(x,\tilx)^v \leq (t+1)^{-(1-v)/2}\Lambda,
		\qquad
		\forall t\in[0,N^2T],
		\
		x,\tilx\in\T,	
	$
\item \label{cor:hk:hkahold} \
	$
		\displaystyle
		\sum_{\tilx\in\T}\hka(t;x,\tilx)\dist_{\T}(x,\tilx)^v \leq (t+1)^{v/2}\Lambda,
		\qquad
		\forall t\in[0,N^2T],
		\
		x\in\T,	
	$
\item \label{cor:hk:hkahold:} \
	$
		\displaystyle
		\sum_{\tilx\in\T}|\nabla_x\hka(t;x,\tilx)|\Big(\frac{\dist_{\T}(x,\tilx)}{N}\Big)^v \leq \Lambda,
		\qquad
		\forall t\in[0,N^2T],
		\
		x\in\T,	
	$
\item \label{cor:hk:hkahold::} \
	$
		\displaystyle
		\sum_{\tilx\in\T}|\nabla_{\tilx}\hka(t;x,\tilx)|\Big(\frac{\dist_{\T}(x,\tilx)}{N}\Big)^v \leq \Lambda,
		\qquad
		\forall t\in[0,N^2T],
		\
		x\in\T,		
	$
\item \label{cor:hk:hkrsup} \
	$
		\displaystyle
		|\hkr(t;x,\tilx)| \leq \frac{N^{-v}}{\sqrt{t+1}}\Lambda,
		\qquad
		\forall t\in[0,N^2T],
		\
		x,\tilx \in\T,	
	$
\item \label{cor:hk:hkrsum} \
	$
		\displaystyle
		\sum_{\tilx\in\T} |\hkr(t;x,\tilx)| \leq N^{-v}\Lambda,
		\qquad
		\forall t\in[0,N^2T],
		\
		x\in\T,	
	$
\item \label{cor:hk:hkrgsum} \
	$
		\displaystyle
		\sum_{\tilx\in\T}|\hkr(t;x,\tilx)-\hkr(t;x',\tilx)| \leq \frac{(\dist_{\T}(x,x'))^{u}}{(t+1)^{u/2}} N^{-v}\Lambda,
		\qquad
		\forall t\in[0,N^2T],
		\
		x,x'\in\T,
	$
\item \label{cor:hk:hkrgsup} \
	$
		\displaystyle
		|\hkr(t;x,\tilx)-\hkr(t;x',\tilx)| \leq \frac{(\dist_{\T}(x,x'))^{u}}{(t+1)^{(u+1)/2}} N^{-v}\Lambda,
		\qquad
		\forall t\in[0,N^2T],
		\
		x,x',\tilx\in\T.
	$
\end{enumerate}
\end{proposition}
\begin{proof}
Recall the definition of $ \Rt(x,x') $ from~\eqref{eq:Rt} and recall the seminorm $ \hold{\,\Cdot\,}_{\urt,N} $ from~\eqref{eq:hold}.
With $ \rt(x)=\Rt(0,x)-\Rt(0,x-1) $, we have
\begin{align}
	\label{eq:rtbd}
	|\rt(x)| \leq N^{-\urt} \hold{ \Rt_N }_{\urt,N}.
\end{align}
In particular, under Assumption~\ref{assu:rt}\ref{assu:rt:holder},
$ \norm{\rt}_{L^\infty(\T)} \leq N^{-\urt}\Lambda $ with probability $ \to_{\Lambda,N} 1 $.
This being the case, taking $ v'=\urt-v $ in Lemma~\ref{lem:hk},
and summing over $ n\geq 1 $ therein, we see that events \ref{cor:hk:hkrsup}--\ref{cor:hk:hkrgsup} hold with probability $ \to_{\Lambda,N} 1 $.

%
%
With $ \hka=\hk+\hkr $, \ref{cor:hk:hkasup}--\ref{cor:hk:hkagdsum}
follow by combining the bounds (we just showed they hold with probability $ \to_{\Lambda,N} 1 $) in \ref{cor:hk:hkrsup}, \ref{cor:hk:hkrgsum} and \ref{cor:hk:hkrgsup} with those in \eqref{eq:hk:sup}--\eqref{eq:hkgd:sum}.
As for~\ref{cor:hk:hkahold:sup}--\ref{cor:hk:hkahold::},
with $ \hka=\hk+\hkr $ and $ \frac{\dist_\T(x,x')}{N} \leq 1 $,
we write
\begin{align}
	\label{eqpf:hkhold:sup}
	\hka(t;x,\tilx)\dist_\T(x,x')^v
	&\leq
	\hk(t;x,\tilx)\dist_\T(x,x')^v
	+
	|\hkr(t;x,\tilx)| \, N^v,
\\
	\label{eqpf:hkhold}
	\sum_{\tilx\in\T}\hka(t;x,\tilx)\dist_\T(x,x')^v
	&\leq
	\sum_{\tilx\in\T}\hk(t;x,\tilx)\dist_\T(x,x')^v
	+
	\sum_{\tilx\in\T}|\hkr(t;x,\tilx)| \, N^v,
\end{align}
and, for $ y=x,\tilx $,
\begin{align}
	\notag
	\sum_{\tilx\in\T}|\nabla_{y}\hka(t;x,\tilx)|\Big(\frac{\dist_\T(x,x')}{N}\Big)^{v}
	&\leq
	\sum_{\tilx\in\T}|\nabla_{y}\hk(t;x,\tilx)|\Big(\frac{\dist_\T(x,x')}{N}\Big)^{v}
	+
	\sum_{\tilx\in\T}|\nabla_{y}\hkr(t;x,\tilx)|
\\
	\label{eqpf:hkhold:}
	&\leq
	\sum_{\tilx\in\T}|\nabla_{y}\hk(t;x,\tilx)|\Big(\frac{\dist_\T(x,x')}{N}\Big)^{v}
	+
	\sum_{\tilx\in\T}2|\hkr(t;x,\tilx)|.
\end{align}
Applying~\eqref{eq:hk:hold:sup}--\eqref{eq:hk:hold::} as well as \ref{cor:hk:hkrsup} and \ref{cor:hk:hkrsum} (that we have already established in this proposition), we can bound the corresponding terms in~\eqref{eqpf:hkhold:sup}--\eqref{eqpf:hkhold:}.
From this and by using $ (t+1)^{-1/2} \leq (t+1)^{-(1-v)/2} $ we obtain the desired results for~\ref{cor:hk:hkahold:sup}--\ref{cor:hk:hkahold::}.
\end{proof}

\section{The Semigroups $ \Sg(t) $ and $ \sg(t) $}
\label{sect:Sgsg}
Our goal in this section is to establish the relevant properties of
the semigroups $ \Sg(t)=e^{t\Ham} $ and $ \sg(t)=e^{t\ham} $.
In particular, in Section~\ref{sect:Sg},
for any given potential $ \Rtlim' $, we will \emph{construct} $ \Sg(t)=e^{t\Ham} $ and establish bounds
using integration by parts techniques.
Then, in Section~\ref{sect:sg}, we generalize these techniques to the microscopic setting
to establish bounds on $ \sg(t) $.

As mentioned in the introduction,
we will utilize an expansion of the Feynman--Kac formula,
which is similar to the expansion considered in \cite[Section 14, Chapter V]{simon79}
for singular but function-valued potentials.
With the potential $ \Rtlim'(x) $ being non-function-valued,
we need to extract the smoothing effect of the heat semigroup to compensate the roughness of $ \Rtlim'(x) $.
This is done by integration by parts in Lemmas~\ref{lem:Ips} and \ref{lem:ips}.

\subsection{Macroscopic}
\label{sect:Sg}
Recall that $ \Ham=\frac12\partial_{xx}+\Rtlim'(x) $.
As previously explained in Remark~\ref{rmk:Rtfixed},
for the analysis within this subsection (that pertains to the \emph{limiting} \ac{SPDE}),
the randomness of $ \Rtlim $ plays no role,
and we will assume without loss of generality $ \Rtlim $ is a deterministic function in $ C^{\urt}[0,1] $.

We begin by recalling the classical construction of $ \Ham $ from~\cite{fukushima77}.
Note that, even though~\cite{fukushima77} treats $ \Ham $ on the closed interval $ [0,1] $ with Dirichlet boundary condition,
the (relevant) argument carries through for $ \limT $ as well.
Write $ H^1(\limT):\{f\in\limT\to\R: f,f'\in L^2(\limT)\} $ for the Sobolev space,
equipped with the norm $ \norm{f}^2_{H^1(\limT)} := \norm{f}^2_{L^2(\limT)} + \norm{f'}^2_{L^2(\limT)} $.
For $ f,g\in L^2(\limT) $, write $ \langle f,g\rangle = \langle f,g\rangle_{L^2(\limT)}:= \int_{\limT} fg dx $
for the inner product in $ L^2(\limT) $,
and similarly $ \langle f,g\rangle_{H^1(\limT)} := \int_{\limT} (fg + f'g') dx $.
Consider the symmetric quadratic form
\begin{align*}
	F_\Rtlim: H^1(\limT)\times H^1(\limT)\to\R,
	\qquad
	F_\Rtlim(f,g): = \tfrac12 \langle f',g' \rangle - f(1)g(1) \Rtlim(1) + \int_{0}^1 (f'g+fg')(x) \Rtlim(x) dx.
\end{align*}
If $ \Rtlim $ were smooth, integration by parts gives $ F_\Rtlim(f,g)= -\langle f,\Ham g\rangle  $.

We now appeal to \cite[Definition 12.14]{grubb08} to define $ \Ham $ to be the operator associated to $ F_\Rtlim $.
Let $ D(\Ham) $ denote the domain of $ \Ham $.
From the preceding definition we have that
\begin{align}
	\label{eq:ham:form}
	D(\Ham) \subset H^1(\limT);
	\qquad\quad
	- \langle \Ham f, g \rangle = F_\Rtlim(f,g),
	\qquad
	\forall
	(f,g)\in D(F_\Rtlim)\times H^1(\limT).
	\qquad
\end{align}
On $ \limT $ we have the following elementary bound
\begin{align}
	\label{fn:LinfinH1}
	\norm{\,\Cdot\,}_{L^\infty(\limT)} \leq \sqrt{2}\norm{\,\Cdot\,}_{H^1(\limT)}
\end{align}
To see this, for $ x,y\in\limT $, write $ |f(x)| \leq |f(y)| + \int_{[x,y]} |f'(\tily)| d\tily $,
and integrate in $ y\in\limT $ to get $ |f(x)| \leq \norm{f}_{L^1(\limT)} + \norm{f'}_{L^1(\limT)} \leq \sqrt{2}\norm{f}_{H^1(\limT)} $.
Now, with $ \Rtlim $ being bounded, using \eqref{fn:LinfinH1}
it is readily checked (see~\cite[Lemma~1]{fukushima77}) that
\begin{align*}
	F_\Rtlim(f,g) + c \langle f, g\rangle_{L^2(\limT)} \geq \tfrac{1}{c} \langle f, g\rangle_{H^1(\limT)},
	\qquad\quad
	F_\Rtlim(f,g) \leq c \langle f, g\rangle_{H^1(\limT)},	
	\qquad
	f,g\in H^1(\limT),
\end{align*}
for some constant $ c=c(\Rtlim) $ depending only on $ \Rtlim $.
Given these properties, and that $ F_\Rtlim $ is symmetric,
it then follows that (see \cite[Theorem 12.18, Corollary 12.19]{grubb08}) $ \Ham $ is a self-adjoint, closed operator,
with $ D(\Ham) $ being dense in $ L^2(\limT) $.

Having constructed $ \Ham $, we now turn to the semigroup $ \Sg(t)=e^{t\Ham} $.
\emph{Heuristically}, the semigroup should be given by the Feynman--Kac formula
\begin{align*}
	\big( \Sg(t) f \big)(x)
	``="
	\Ex_x\Big[ e^{\int_0^t \Rtlim'(B(s))ds} f(B(t)) \Big],
\end{align*}
where $ B $ denotes a Brownian motion on $ \limT $ starting from $ x $.
The issue with this formula is that,
under our assumptions, $ \Rtlim \in C^{\urt}[0,1] $ is not necessarily differentiable,
namely the potential $ \Rtlim' $ may not be function-valued.
As mentioned previously, for function-valued potentials, 
such Feynman--Kac formula have been rigorously made sense in \cite{mckean77,simon79,simon82}.
Continuing, for the moment, with the informal Feynman--Kac formula,
we Taylor-expand the exponential function $ \exp(\int_0^t \Rtlim'(B(s))ds) $,
and exchange the expectation $ \Ex_x[\Cdot] $ with the integrals.
This yields
\begin{align*}
	(\Sg(t) f)(x)
	&``="
	\Ex_x\Big[ \sum_{n=0}^\infty \frac{1}{n!}\int_{[0,t]^n} \Big(\prod_{i=1}^n\Rtlim'(B(t_i))dt_i\Big) f(t) \Big]
\\
	&=
	\Ex_x\Big[ \sum_{n=0}^\infty \int_{0<t_1<\ldots<t_n<t} \Big(\prod_{i=1}^n\Rtlim'(B(t_i))dt_i\Big) f(t) \Big]
	``="
	\int_{\limT} \Sg(t;x,\tilx)f(x) dx,
\end{align*}
where $ \Sg $ is defined as follows.
With the notation $ \Sigma_{n}(t) $ from~\eqref{eq:Sigman}, $ d^n\vec{s} $ from~\eqref{eq:ds},
the convention $ x_0:=x $, $ \tilx:=x_{n+1} $,
and with
\begin{align}
	\label{eq:HK}
	\HK(t;x,\tilx) = \sum_{i\in\Z}\frac{1}{\sqrt{2\pi t}} e^{-\frac{|x-\tilx+i|^2}{2t}},
	\qquad
	x,\tilx \in[0,1)
\end{align}
denoting the standard heat kernel on $ \limT $, we define
\begin{align}
	\label{eq:Sgker}
	\Sg(t;x,\tilx)
	&:=
	\HK(t;x,\tilx) + \sum_{n=1}^\infty\Sgr_n(t;x,\tilx),
\\
	\label{eq:Sgr}
	\Sgr_n(t;x,\tilx)
	&:=
	\int_{\Sigma_n(t)} \SgrI_n(\vec{s};x,\tilx) d^n\vec{s},
\\
	\label{eq:SgrI}
	\SgrI_n(\vec{s};x,\tilx)
	&=
	\SgrI_n(s_0,\ldots,s_n;x,\tilx)
	:=
	\int_{\limT^n} \prod_{i=0}^n\HK(s_i;x_i,x_{i+1})\ \prod_{i=1}^n d\Rtlim(x_i).
\end{align}

Note that, for each fixed $ (s_0,\ldots,s_n)\in \Sigma_n(t) $, the function
$ \prod_{i=0}^n\HK(s_i;x_i,x_{i+1}) $ is $ C^\infty(\limT^{n+2}) $ in $ (x_0,\ldots,x_{n+1}) $,
so~\eqref{eq:SgrI} is actually a well-defined Riemann–Stieltjes integral.
Namely, despite the heuristic nature of the preceding calculations,
the function $ \Sg(t;x,\tilx) $ in \eqref{eq:Sgker} is well-defined as long as
the summations and integrals in \eqref{eq:Sgker}--\eqref{eq:Sgr} converge absolutely.
This construction will be carried out in Proposition~\ref{prop:Sg}:
there we \emph{define} $ \Sg(t) $ via~\eqref{eq:Sgker}--\eqref{eq:SgrI},
check that the result defines a bounded operator for each $ t \in[0,\infty) $,
and verify that the result is indeed the semigroup generated by $ \Ham $.

\begin{remark}
\label{rmk:notchaos}
In the case when $ \Rtlim $ is equal to a Brownian motion $ B $,
one can also consider the chaos expansion of $ \Sg(t;x,\tilx) $ (see, e.g.,~\cite{janson97}).
That is, for each $ t,x,\tilx $, one views $ \Sg(t;x,\tilx) $ as a random variable (with randomness over $ B $),
and decompose it into terms that belongs to $ n $-th order Wiener chaoses of $ B $.
Such an expansion has been carried out in \cite{gu18} for \ac{PAM} in two dimensions,
and it is conceivable that their method carries over in one dimension.
We clarify here that our expansion~\eqref{eq:Sgker}--\eqref{eq:SgrI} here is \emph{not} the chaos expansion.
For example, it is readily checked that $ \Ex[\Sgr_1(t;x,\tilx)\Sgr_2(t;x,\tilx)] \neq 0 $,
where the expectation is taken with respect to $ B $.
\end{remark}

To prepare for the construction in Proposition~\ref{prop:Sg},
in Lemma~\ref{lem:ips} we establish an integration-by-parts identity,
and, based on this identity,  in Lemmas~\ref{lem:Ipsbd}--\ref{lem:SgrIbd},
we obtain bounds on the relevant integrals in \eqref{eq:Sgker}--\eqref{eq:Sgr}.
To state Lemma~\ref{lem:Ips}, we begin with some notation.
Recall that we write $ [x,\tilx] $, $ x,\tilx\in\limT $, for the interval on $ \limT $
that goes counterclockwise from $ x $ to $ \tilx $.
For given $ y_1\neq y_2\in\limT $, let $ \Midp_2(y,\tily)\in[y_1,y_2] $ denote the midpoint of the interval $ [y_1,y_2] $,
and let $ \Midp_1(y_1,y_2)\in[y_2,y_1] $ denote the midpoint of the interval $ [y_2,y_1] $.
Set $ \Parti_1(y_1,y_2) := [\Midp_1(y_1,y_2),\Midp_2(y_1,y_2)) \subset \limT $
and $ \Parti_2(y_1,y_2) := [\Midp_2(y_1,y_2),\Midp_1(y_1,y_2)) \subset \limT $.
Indeed, $ \Parti_1(y_1,y_2), \Parti_2(y_1,y_2) $ form a partition of $ \limT $, with the property
\begin{align}
	\label{eq:Parti}
	\dist_{\limT}(y_j,x) \leq \dist_{\limT}(y_{j+1},x), \ \forall x\in \Parti_j(y_1,y_2),
	\qquad
	\forall j=1,2,
	\qquad
	\text{where } y_3 := y_1.
\end{align}
See Figure~\ref{fig:Midpt} for an illustration.
\begin{figure}[h]
\centering
	\psfrag{X}{$ y_1 $}
	\psfrag{Y}{$ y_2 $}
	\psfrag{T}[c]{$ \Parti_1(y_1,y_2) $}
	\psfrag{S}{$ \Parti_2(y_1,y_2) $}
	\psfrag{M}[r]{$ \Midp_1(y_1,y_2) \quad $}
	\psfrag{N}{$ \Midp_2(y_1,y_2) $}
	\includegraphics[width=.4\textwidth]{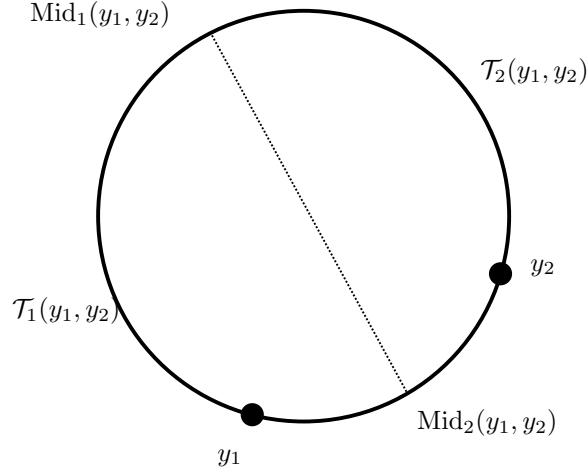}
	\caption{The points $ \Midp_1(y_1,y_2), \Midp_2(y_1,y_2) $ are equidistant from $y_1$ and $y_2$; the intervals $ \Parti_1(y_1,y_2), \Parti_2(y_1,y_2) $ are composed of all points closer to $y_1$ and $y_2$ respectively.}
	\label{fig:Midpt}
\end{figure}

Recall that $ [x_1,x_2] \subset \limT $ denotes the interval going from $ x_1 $ to $ x_2 $ counterclockwise.
We define the macroscopic analog of $ \Rt(x_1,x_2) $ (see~\eqref{eq:Rt}) as
\begin{align}
	\label{eq:RtLim}
	\Rtlim(x_1,x_2)
	=
	\int_{[x_1,x_2]\setminus\{0\}} d \Rtlim(x)
	= \left\{\begin{array}{l@{,}l}
		\Rtlim(x_2)-\Rtlim(x_1)	&\text{ when } x_1\leq x_2\in[0,1),
	\\
		\Rtlim(x_2)-\Rtlim(0) + \Rtlim(1)-\Rtlim(x_1)	&\text{ when }  x_2<x_1\in[0,1).
	\end{array}\right.
\end{align}
Note that the integral excludes $ 0 $ so that the possible jump of $ \Rtlim(x) $ there will not be picked up by $ \Rtlim(x_1,x_2) $. 
Hereafter we adopt the standard notation $ f(x)|^{x=b}_{x=a} := f(b)-f(a) $.
Lemma~\ref{lem:Ips} gives a integration-by-parts formula for $ \Ips $.
\begin{lemma}
\label{lem:Ips}
For $ y_1\neq y_2\in\limT $, set
\begin{align}
\label{eq:Ips:}
\begin{split}
	\Ips&(s,s';y_1,y_2)
	:=
	\sum_{j=1}^2
	\Bigg(
		\HK(s;y_1,x) \Rtlim(y_j,x) \HK(s';x,y_2)\big|_{x=\Midp_j(y_1,y_2)}^{x=\Midp_{j+1}(y_1,y_2)}
\\
		&-
		\int_{\Parti_j(y_1,y_2)} (\partial_{x}\HK(s;y_1,x)) \Rtlim(y_j,x) \HK(s';x,y_2) dx
		-
		\int_{\Parti_j(y_1,y_2)} \HK(s;y_1,x) \Rtlim(y_j,x) \partial_{x}\HK(s';x,y_2) dx
	\Bigg),
\end{split}
\end{align}
where, by convention, we set $ \Midp_{3}(\tily,y) :=\Midp_{1}(\tily,y)  $. Then we have
\begin{align}
	\label{eq:Ips}
	\SgrI_n(\vec{s};x,\tilx)
	=
	\int_{\limT^{n+1}}
	\HK(\tfrac{s_0}{2};x,y_1) dy_1
		\Big( \prod_{i=1}^n \Ips(\tfrac{s_{i-1}}{2},\tfrac{s_{i}}{2};y_{i},y_{i+1}) dy_{i+1} \Big)
	\HK(\tfrac{s_{n}}{2};y_{n+1},\tilx).
\end{align}
\end{lemma}
\begin{remark}
The value of $ \Ips(s,s';y_1,y_2) $ at $ y=\tily $ in~\eqref{eq:Ips} is irrelevant since the set has zero Lebesgue measure.
\end{remark}

\begin{proof}
In~\eqref{eq:SgrI},
use the semigroup property $ \HK(s_i;x_i,x_{i+1})=\int_{\limT} \HK(\frac{s_i}{2};x_i,y_i)\HK(\frac{s_i}{2};y_i,x_{i+1}) dy_i $
to rewrite
\begin{align}
	\label{eqlem:Ips}
	&\SgrI_n(\vec{s};x,\tilx)
	=
	\int_{\limT^{n+1}}
	\HK(\tfrac{s_0}{2};x,y_1) dy_1
		\Big( \prod_{i=1}^n \til{\Ips}_i(y_i,s_{i-1},y_{i+1},s_i) dy_{i+1} \Big)
	\HK(\tfrac{s_{n}}{2};y_{n+1},\tilx),
\\
	\label{eqlem:Ips:}
	&\til{\Ips}_i(y_i,s_{i-1},y_{i+1},s_i)
	:=
	\int_{\limT} \HK(\tfrac{s_{i-1}}{2};y_{i},x) \, d\Rtlim(x) \, \HK(\tfrac{s_{i}}{2};x,y_{i+1}).
\end{align}
The integral in \eqref{eqlem:Ips:} is over $ x\in\limT $ in the Riemann–Stieltjes sense.
Split the integral into integrals over $ \Parti_1(y_i,y_{i+1}) $ and $ \Parti_2(y_{i},y_{i+1}) $.
This gives $ \til{\Ips}_i=\til{\Ips}_{i,1}+\til{\Ips}_{i,2} $,
\begin{align}
	\notag
	\til{\Ips}_{i,j}(y_i,s_{i-1},y_{i+1},s_i)
	&:= \int_{\Parti_j(y_i,y_{i+1})} \HK(\tfrac{s_{i-1}}{2};y_{i},x) \, d\Rtlim(x) \, \HK(\tfrac{s_{i}}{2};x,y_{i+1})
\\
	\label{eq:Ips1}
	&= \int_{\Parti_j(y_{i},y_{i+1})} \HK(\tfrac{s_{i-1}}{2};y_{i},x) \, d\Rtlim(y_{i+j-1},x) \, \HK(\tfrac{s_{i}}{2};x,y_{i+1}),
\end{align}
where the equality in~\eqref{eq:Ips1} follows
since $ y_i $ and $ y_{i+1} $ are \emph{fixed} (the integral is in $ x $).
Then, in~\eqref{eq:Ips1}, integrate by parts (in $ x $), and add the results for $ j=1,2 $ together.
This gives $ \til{\Ips}_{i}= \Ips(\frac{s_{i-1}}{2},\frac{s_{i}}{2};y_i,y_{i+1}) $.
Inserting this back into~\eqref{eqlem:Ips} completes the proof.
\end{proof}

Equation~\eqref{eq:Ips} expresses $ \SgrI_n $ in terms of $ \Ips $.
We proceed to establish bounds on the latter.
Here we list a few bounds on $ \HK(t;x,x') $ that will be used in the subsequent analysis.
They are readily checked from the explicit expression~\eqref{eq:HK} of $ \HK $ (in the spirit of \eqref{eq:hk}).
Given any $ u\in(0,1] $ and $ T<\infty $, for all $ x,x',\tilx\in\limT $ and $ s\in[0,T] $,
\begin{subequations}
\begin{align}
	\label{eq:HK:int}
	\int_{\limT} \HK(s;x,\tilx) d\tilx &=1,
\\
	\label{eq:HKgx:int}
	\int_{\limT} |\partial_{\tilx}\HK(s;x,\tilx)| \dist_{\limT}(x,\tilx)^u d\tilx &\leq c(u,T) s^{-\frac{1-u}{2}},
\\
	\label{eq:HKgx':int}
	\int_{\limT} |\partial_{x}\HK(s;x,\tilx)| \dist_{\limT}(x,\tilx)^u d\tilx &\leq c(u,T) s^{-\frac{1-u}{2}},
\\
	\label{eq:HK:hold}
	\int_{\limT} \HK(s;x,\tilx) \dist_{\limT}(x,\tilx)^u d\tilx &\leq c(T) s^{-\frac{u}2},
\\
	\label{eq:HK:uni}
	\HK(s;x,\tilx) &\leq c(T) s^{-\frac12},
\\
	\label{eq:HK:uni:}
	\HK(s;x,\tilx) \dist_{\limT}(x,\tilx)^u &\leq c(T) s^{-\frac{1-u}2},
\\
	\label{eq:HKg:int}
	\int_{\limT} |\HK(s;x,\tilx)-\HK(s;x',\tilx)| d\tilx &\leq c(u,T) \dist_{\limT}(x,x')^{u}s^{-\frac{u}{2}},
\\
	\label{eq:HKg:uni}
	|\HK(s;x,\tilx)-\HK(s;x',\tilx)| &\leq c(u,T)s^{-\frac{1+u}{2}}.
\end{align}
\end{subequations}

\begin{lemma}
\label{lem:Ipsbd}
Given any $ v\in(0,\urt) $ and $ T<\infty $, for all $s,s' \in [0,T]$,
\begin{align*}
	\int_{\limT} |\Ips(s,s';y,y')| dy'
	\leq
	c(v,T) \norm{\Rtlim}_{C^{\urt}[0,1]} \big( {s}^{-(1-v)/2}+{s'}^{-(1-v)/2} \big).
\end{align*}
\end{lemma}
\begin{proof}
Recall the definition of $ \norm{\Cdot}_{C^{u}[0,1]} $ and $ [\Cdot]_{C^u(\limT)} $ from~\eqref{eq:Hold}.
From the expression~\eqref{eq:RtLim} of $ \Rtlim(x_1,x_2) $ and the property~\eqref{eq:Parti},
it is straightforward to check that
\begin{align*}
	|\Rtlim(y_j,x)|
	\leq
	\dist_{\limT}(y_j,x)^{\urt} \, [ \Rtlim ]_{C^{\urt}(\limT)}
	\leq
	\dist_{\limT}(y_j,x)^{v} \, \norm{ \Rtlim }_{C^{\urt}[0,1]},
	\qquad
	\forall x\in\Parti_j(y_1,y_2),
	\
	j=1,2.
\end{align*}
Inserting this bound into~\eqref{eq:Ips:} gives
\begin{align}
	\label{eqlem:Ipsbd:1}
	|\Ips(s,s';y_1,y_2)|
	\leq
	\norm{ \Rtlim }_{C^{\urt}[0,1]}
	\sum_{j=1}^2
	\Bigg(
		&\sum_{x\in\{\Midp_j(y_1,y_2)\}_{j=1}^2}  \HK(s;y_1,x) \, \dist_{\limT}(y_j,x)^v \, \HK(s';x,y_2)
\\
	\label{eqlem:Ipsbd:2}
	&+
		\int_{\Parti_j(y_1,y_2)} \,\big|\partial_{x}\HK(s;y_1,x)\big| \dist_{\limT}(y_j,x)^v \,\HK(s';x,y_2) dx
\\
	\label{eqlem:Ipsbd:3}
	&+
		\int_{\Parti_j(y_1,y_2)} \HK(s;y_1,x) \, \dist_{\limT}(y_j,x)^v \big|\partial_{x}\HK(s;y_2,x)\big|  \,  dx
	\Bigg).
\end{align}
In~\eqref{eqlem:Ipsbd:2} use~\eqref{eq:Parti}
to bound $ \dist_{\limT}(y_j,x) $ by $ \dist_{\limT}(y_1,x) $,
and in~\eqref{eqlem:Ipsbd:3} use~\eqref{eq:Parti} to bound $ \dist_{\limT}(y_j,x) $ by $ \dist_{\limT}(y_2,x) $.
This way $ \dist_{\limT} $ has the same $ y $ variable as $ \partial_{x}\HK $.
We now have
\begin{align}
\label{eq:Ips:bd:}
	&|\Ips(s,s';y_1,y_2)|
	\leq
	\norm{ \Rtlim }_{C^{\urt}[0,1]}
	\Bigg(
	2
	\sum_{x\in\{\Midp_j(y_1,y_2)\}_{j=1}^2}  \HK(s;y_1,x) \dist_{\limT}(y_j,x)^v \HK(s';x,y_2)
\\
\label{eq:Ips:bd::}
		&
		+
		\int_{\limT} \, \big|\partial_{x}\HK(s;y_1,x)\big| \dist_{\limT}(y_1,x)^v \, \HK(s';x,y_2) dx
		+
		\int_{\limT} \HK(s;y_1,x) \, \dist_{\limT}(y_2,x)^v \big|\partial_{x}\HK(s;x,y_2)\big| \, dx
	\Bigg).
\end{align}
Integrate~\eqref{eq:Ips:bd:}--\eqref{eq:Ips:bd::} over $ y_2\in\limT $,
use \eqref{eq:HK:int}, \eqref{eq:HK:uni}--\eqref{eq:HK:uni:} to bound the terms in~\eqref{eq:Ips:bd:},
and use \eqref{eq:HK:int}--\eqref{eq:HKgx':int} to bound the terms in~\eqref{eq:Ips:bd::}.
We then conclude the desired result
\begin{align*}
	\int_{\limT} |\Ips(s,s';y_1,y_2)| dy_2
	\leq
	c(v,T) \norm{\Rtlim}_{C^{\urt}[0,1]} \big( s^{-(1-v)/2} + {s'}^{-(1-v)/2} + s^{-(1-v)/2} + {s'}^{-(1-v)/2} \big).
\end{align*}
\end{proof}

Based on Lemmas~\ref{lem:Ips}--\ref{lem:Ipsbd}, we now establish bounds on $ \SgrI_n $.
Recall the notation $ \Sigma_{n}(t) $ from~\eqref{eq:Sigman} and $ d^n\vec{s} $ from~\eqref{eq:ds}.
\begin{lemma}
\label{lem:SgrIbd}
Given any $ u\in(0,1] $ and $ v\in(0,\urt) $, we have,
for all $ x,x',\tilx\in\limT $, $ t\in[0,T] $, and $ n \geq 1 $,
\begin{enumerate}[label=(\alph*),leftmargin=7ex]
\item \label{lem:SgrI:int}
	$
	\displaystyle
	\int_{\Sigma_n(t)} \int_{\limT} |\SgrI_n(\vec{s};x,\tilx)| d^n\vec{s}\, d\tilx
	\leq
	t^{\frac{(1+v)n}{2}}\frac{ (c(v,T) \norm{\Rtlim}_{C^{\urt}[0,1]} )^n}{\Gamma(\frac{(1+v)n+2}{2})},
	$
\item \label{lem:SgrI:gint}
	$
	\displaystyle	
	\int_{\Sigma_n(t)} \int_{\limT} |\SgrI_n(\vec{s};x,\tilx)-\SgrI_n(\vec{s};x',\tilx)| d^n\vec{s}\,d\tilx
	\leq
	\dist_{\limT}(x,x')^{u} t^{\frac{(1+v)n-u}{2}} \frac{(c(u,v,T) \norm{\Rtlim}_{C^{\urt}[0,1]})^n }{\Gamma(\frac{(1+v)n+2-u}{2})},
	$
\item \label{lem:SgrI:sup}
	$
	\displaystyle
	\int_{\Sigma_n(t)} |\SgrI_n(\vec{s};x,\tilx)| d^n\vec{s}
	\leq
	t^{\frac{(1+v)n-1}{2}} \frac{ (c(v,T) \norm{\Rtlim}_{C^{\urt}[0,1]} )^n}{\Gamma(\frac{(1+v)n+1}{2})},
	$
\item \label{lem:SgrI:gsup}
	$
	\displaystyle
	\int_{\Sigma_n(t)} |\SgrI_n(\vec{s};x,\tilx)-\SgrI_n(\vec{s};x',\tilx)| d^n\vec{s}
	\leq
	\dist_{\limT}(x,x')^{u} t^{\frac{(1+v)n-1-u}{2}} \frac{(c(u,v,T) \norm{\Rtlim}_{C^{\urt}[0,1]} )^n }{\Gamma(\frac{(1+v)n+1-u}{2})}.
	$
\end{enumerate}
\end{lemma}
\begin{proof}
The proof begins with the given expression~\eqref{eq:Ips} of $ \SgrI_n $:
\begin{align}
	\begin{split}
	\label{eqSglem:SgrIbd:Ips}
	&\SgrI_n(\vec{s};x,\tilx)
	=
	\int_{\limT^{n+1}}
	\HK(\tfrac{s_0}{2};x,y_1) \, dy_1 \, \Ips(\tfrac{s_{0}}{2},\tfrac{s_{1}}{2};y_{1},y_{2}) \, dy_2
	\cdots
\\
	&\hphantom{\SgrI_n(\vec{s};x,\tilx)=\int_{\limT^{n+1}}\HK(\tfrac{s_0}{2};x,y_1) \, dy_1 \,}
	\, dy_{n} \, \Ips(\tfrac{s_{n-1}}{2},\tfrac{s_{n}}{2};y_{n},y_{n+1}) \, dy_{n+1} \,
	\HK(\tfrac{s_{n}}{2};y_{n+1},\tilx),
	\end{split}
\\
	\begin{split}
	&\SgrI_n(\vec{s};x,\tilx) - \SgrI_n(\vec{s};x',\tilx)
	=
	\int_{\limT^{n+1}}
	\big( \HK(\tfrac{s_0}{2};x,y_1)- \HK(\tfrac{s_0}{2};x',y_1) \big) \, dy_1
\\
	&
	\label{eqSglem:SgrIbd:Ipsg:}	
	\hphantom{\SgrI_n(\vec{s};x,\tilx) - \SgrI_n(\vec{s};x',\tilx)=}
	\Ips(\tfrac{s_{0}}{2},\tfrac{s_{1}}{2};y_{1},y_{2}) \, dy_2
	\cdots
	\, dy_{n} \, \Ips(\tfrac{s_{n-1}}{2},\tfrac{s_{n}}{2};y_{n},y_{n+1}) \, dy_{n+1} \,
	\HK(\tfrac{s_{n}}{2};y_{n+1},\tilx).
	\end{split}
\end{align}

For~\ref{lem:SgrI:int}--\ref{lem:SgrI:gint},
integrate \eqref{eqSglem:SgrIbd:Ips}--\eqref{eqSglem:SgrIbd:Ipsg:} over $ \tilx,y_{n+1},\ldots, y_1\in\limT $ \emph{in order}.
Use~\eqref{eq:HK:int} for the integral over $ \tilx $,
use Lemma~\ref{lem:Ipsbd} subsequently for the integrals over $ y_{n},\ldots,y_{2} $,
and use~\eqref{eq:HK:int} and \eqref{eq:HKg:int} for the integral over $ y_1 $.
We then obtain
\begin{subequations}
\label{eqSglem:}
\begin{align}
	\label{eqSglem:a}
	\int_{\limT} |\SgrI_n(\vec{s};x,\tilx)| d\tilx
	&\leq
	\big( c(v,T) \norm{\Rtlim}_{C^{\urt}[0,1]} \big)^n \prod_{i=1}^n \big( s_{i-1}^{-(1-v)/2}+s_{i}^{-(1-v)/2} \big),
\\
	\label{eqSglem:b}	
	\int_{\limT} |\SgrI_n(\vec{s};x,\tilx)-\SgrI_n(\vec{s};x',\tilx)| d\tilx
	&\leq
	\dist_{\limT}(x,x')^{u}
	\big( c(u,v,T)\norm{ \Rtlim }_{C^{\urt}[0,1]} \big)^n s_0^{-\frac{u}{2}} \prod_{i=1}^n \big( s_{i-1}^{-(1-v)/2}+s_{i}^{-(1-v)/2} \big).
\end{align}
\end{subequations}
For~\ref{lem:SgrI:sup}--\ref{lem:SgrI:gsup},
in~\eqref{eqSglem:SgrIbd:Ips}--\eqref{eqSglem:SgrIbd:Ipsg:},
use~\eqref{eq:HK:uni} to bound $ \HK(\frac{s_n}{2};y_{n+1},\tilx) $ by $ c s_{n}^{-1/2} $,
and then integrate the result over $ y_{n+1},\ldots, y_1\in\limT $ in order.
Similarly to the preceding, we have
\begin{align}
	\tag{\ref*{eqSglem:}c}
	|\SgrI_n(\vec{s};x,\tilx)|
	&\leq
	\big( c(u,T)\norm{ \Rtlim }_{C^{\urt}[0,1]} \big)^n \prod_{i=1}^n \big( s_{i-1}^{-(1-u)/2}+s_{i}^{-(1-u)/2} \big) s_n^{-\frac12},
\\
	\tag{\ref*{eqSglem:}d}
	\label{eqSglem:d}
	|\SgrI_n(\vec{s};x,\tilx)-\SgrI_n(\vec{s};x',\tilx)|
	&\leq
	\dist_{\limT}(x,x')^{u}
	\big( c(u,v,T)\norm{ \Rtlim }_{C^{\urt}[0,1]} \big)^n s_0^{-\frac{u}{2}} \prod_{i=1}^n \big( s_{i-1}^{-(1-v)/2}+s_{i}^{-(1-v)/2} \big) s_n^{-\frac12}.
\end{align}
Next we will explain how to integrate the r.h.s.\ of~\eqref{eqSglem:a}--\eqref{eqSglem:d} to establish the desired bounds.
We begin with~\eqref{eqSglem:a}.
Expand the $ n $-fold product on the r.h.s.\ of~\eqref{eqSglem:a} into a sum of size $ 2^n $:
\begin{align}
	\label{eq:sumexpansion}
	\prod_{i=1}^n \big( s_{i-1}^{-(1-v)/2}+s_{i}^{-(1-v)/2} \big)
	=
	\sum_{\vec{b}}
	\prod_{i=0}^n s_{i}^{-\frac{1-v}{2}(\ind\set{b_{i-1/2}=i}+\ind\set{b_{i+1/2}=i})},	
\end{align}
where the sum goes over $ \vec{b}=(b_{1/2},b_{3/2},\ldots,b_{n-1/2})\in \{0,1\}\times\{1,2\}\times\cdots\times\{n-1,n\} $,
with the convention that $ b_{-1/2}:=-1 $ and $ b_{n+1/2}:=n+1 $.
Insert~\eqref{eq:sumexpansion} into~\eqref{eqSglem:a}, and integrate both sides over $ \vec{s}\in\Sigma_n(t) $.
%
With the aid of the Dirichlet formula~\eqref{eq:dirichlet}, we obtain
\begin{align*}
	\int_{\Sigma_n(t)\times\limT} |\SgrI_n(\vec{s};x,\tilx)| d^n\vec{s} d\tilx
	\leq
	\big( c(v,T)\norm{ \Rtlim }_{C^{\urt}[0,1]} \big)^n
	\sum_{\vec{b}}
	\frac{t^{(1+v)n/2} \prod_{i=0}^n\Gamma\big(1-\frac{1-v}{2}(\ind\set{b_{i-1/2}=i}+\ind\set{b_{i+1/2}=i})\big)}{\Gamma(\frac{1+v}{2}n+1)}.
\end{align*}
Since $ \Gamma(x) $ is decreasing for $ x\in(0,1] $,
we have $ \Gamma(1-\frac{1-v}{2}(\ind\set{b_{i-1/2}=i}+\ind\set{b_{i+1/2}=i})) \leq \Gamma(v) $.
From this we conclude the desired result for~\ref{lem:SgrI:int}:
\begin{align*}
	\int_{\Sigma_n(t)\times\limT} |\SgrI_n(\vec{s};x,\tilx)| d^n\vec{s} d\tilx
	\leq
	\big(  c(v,T)\norm{ \Rtlim }_{C^{\urt}[0,1]}  \big)^n 2^{n}  \frac{t^{(1+v)n/2}\Gamma(v)^{n+1}}{\Gamma(\frac{1+v}{2}n+1)}
	\leq
	t^{\frac{(1+v)n}{2}}\frac{ (c(v,T) \norm{\Rtlim}_{C^{\urt}[0,1]} )^n}{\Gamma(\frac{(1+v)n+2}{2})}.
\end{align*}
As for~\ref{lem:SgrI:gint}--\ref{lem:SgrI:gsup},
integrating \eqref{eqSglem:b}--\eqref{eqSglem:d} over $ \vec{s}\in\Sigma_n(t) $,
with the aid of the Dirichlet formula~\eqref{eq:dirichlet}, one obtains the desired results via the same procedure as in the preceding.
We do not repeat the argument.
\end{proof}

Given Lemma~\ref{lem:SgrIbd}, we now construct the semigroup $ \Sg(t) $,
whereby verifying the heuristic given earlier.
Recall $ \Sgr_n $ from~\eqref{eq:Sgr}.
\begin{proposition}
\label{prop:Sg}
Fix  $ u\in(0,1] $ and $ T<\infty $.
The series $ \Sgr(t;x,\tilx) := \sum_{n=1}^\infty \Sgr_n(t;x,\tilx) $ converges uniformly over $ x,\tilx\in\limT $ and $ t\in[0,T] $,
and satisfies the bounds
\begin{enumerate}[label=(\alph*),leftmargin=7ex]
\item \label{prop:Sg:int}
	$
		\displaystyle
		\int_{\Sigma_n(t)\times\limT} |\Sgr(t;x,\tilx)| d\tilx
		\leq
		c(T),
	$
\item \label{prop:Sgg:int}
	$
		\displaystyle
		\int_{\Sigma_n(t)\times\limT} |\Sgr(t;x,\tilx)-\Sgr(t;x',\tilx)|  d\tilx
		\leq
		\dist_{\limT}(x,x')^u c(u,T),
	$
\item \label{prop:Sg:sup}
	$
		\displaystyle
		|\Sgr(t;x,\tilx)|
		\leq
		c(T),
	$
\item \label{prop:Sgg:sup}
	$
		\displaystyle
		|\Sgr(t;x,\tilx)-\Sgr(t;x',\tilx)|
		\leq
		\frac{\dist_{\limT}(x,x')^u}{t^{u/2}} c(u,T),
	$
\end{enumerate}
for all $ x,\tilx\in\limT $ and $ t\in[0,T] $.
Furthermore, with $ \Sg(t;x,\tilx) := \HK(t;x,\tilx)+\Sgr(t;x,\tilx) $ (as in~\eqref{eq:Sgker}),
\begin{align*}
	\big( \Sg(t) f \big)(x) := \int_{\limT} \Sg(t;x,\tilx) f(\tilx) d\tilx
\end{align*}
defines an operator $ \Sg(t): C(\limT)\to C(\limT) $ for each $ t\in[0,\infty) $, and $ \Sg(t) $ is, in fact, the semigroup of $ \Ham $.
\end{proposition}
\begin{proof}
By assumption, $ \norm{\Rtlim}_{C^{\urt}[0,1]}<\infty $.
For given $ v\in(0,\urt) $, $ \delta>0 $, and $ c<\infty $, $ \sum_{n=1}^\infty \frac{c^n}{\Gamma(v n+\delta)}<\infty $.
From these two observations the claimed bounds \ref{prop:Sg:int}--\ref{prop:Sgg:sup} follow straightforwardly from Lemma~\ref{lem:SgrIbd}.

It remains to check that the so defined operators $ \Sg(t) $, $ t\geq 0 $, form the semigroup of $ \Ham $.
Fixing $ t,s \in[0,\infty) $, we begin by checking the semigroup property.
Writing $ \Sgr_0(t;x,y):=\HK(t;x,y) $ to streamline notation,
we have
\begin{align}
	\label{eqprop:Sg:s+t}
	\big( \Sg(t) \Sg(s) \big)(x,\tilx)
	:=
	\int_{\limT} \Big( \sum_{n=0}^\infty \Sg_n(t;x,y) \Big)\Big( \sum_{n=0}^\infty \Sg_n(s;y,\tilx) \Big) dy
	=
	\sum_{n=0}^\infty \sum_{n_1+n_2=n}\int_{\limT} \Sg_{n_1}(t;x,y) \Sg_{n_2}(s;y,\tilx) dy.
\end{align}
Here we rearranged the product of two infinite sums into iterated sums,
which is permitted granted the bounds from Lemma~\ref{lem:SgrIbd}.

Fix $ n\geq 0 $, and consider generic $ n_1,n_2\geq 0 $ with $ n_1+n_2=n $.
From the given expressions~\eqref{eq:Sgr}--\eqref{eq:SgrI} of $ \Sgr_n $,
we have
\begin{align*}
	\int_{\limT} &\Sg_{n_1}(t;x,y) \Sg_{n_2}(s;y,\tilx) dy
\\
	&=
	\int_{\limT^{n+1}\times\Sigma_{n_1}(t)\times\Sigma_{n_2}(s)}
	\Big( \prod_{i=0}^{n_1}\HK(t_i;x_i,x_{i+1})\ \prod_{i=1}^n d\Rtlim(x_i) \Big)
	dy
	\Big( \prod_{i=0}^{n_2}\HK(s_i;x'_i,x'_{i+1})\ \prod_{i=1}^n d\Rtlim(x'_i) \Big)
	d^{n_1}\vec{t} \ d^{n_2}\vec{s},
\end{align*}
with the convention $ x_0:=x $, $ x_{n_1}:=y $, $ x'_0:=y $, and $ x_{n_2}:=\tilx $.
Integrate over $ y $, using $ \int_{\limT} \HK(t_{n_1};x_{n_1-1},y)\HK(s_{0};y,x'_1)dy = \HK(t_{n_1}+s_0;x_{n_1-1},x'_1) $.
Renaming variables as $ (x'_1,\ldots,x'_{n_2}):=(x_{n_1+1},\ldots,x_{n}) $
and $ (t_{n_1}+s_0,s_1,\ldots,s_{n_2}):=(t_{n_1},\ldots,t_{n}) $, we obtain
\begin{align}
	\label{eqprop:Sg:1}
	\int_{\limT} \Sg_{n_1}(t;x,y) \Sg_{n_2}(s;y,\tilx) dy
	=
	\int_{\limT^{n}\times\Sigma_{n}(t+s)}
	\Big( \prod_{i=0}^{n_1+n_2}\HK(t_i;x_i,x_{i+1})\ \prod_{i=1}^{n_1+n_2} d\Rtlim(x_i) \Big)
	\ind_{\Sigma'_{n_1,n_2}(t,s)}(\vec{t}\,)
	\,
	d^{n}\vec{t},
\end{align}
where
$
	\Sigma'_{n_1,n_2}(t,s)
	:=
	\{
		t_0+\ldots+t_{n_1-1}<t, \ t_{n_1+1}+\ldots+t_{n_1+n_2}<s
	\}.
$
It is straightforward to check that
\begin{align*}
	\sum_{n_1+n_2=n}\ind_{\Sigma'_{n_1,n_2}(t,s)}(\vec{t}\,) = \ind_{\Sigma_{n}(t+s)}(\vec{t}\,),
	\quad
	\text{for Lebesgue almost every } \vec{t} \in (0,\infty)^n.
\end{align*}
Given this property, we sum \eqref{eqprop:Sg:1} over $ n_1+n_2=n $ to obtain
\begin{align*}
	\sum_{n_1+n_2=n} \int_{\limT} \Sg_{n_1}(t;x,y) \Sg_{n_2}(s;y,\tilx) dy
	=
	\int_{\limT^{n}\times\Sigma_{n}(t+s)}
	\Big( \prod_{i=0}^{n_1+n_2}\HK(t_i;x_i,x_{i+1})\ \prod_{i=1}^{n_1+n_2} d\Rtlim(x_i) \Big)
	\,
	d^{n}\vec{t}
	=
	\Sg_n(t+s;x,\tilx).
\end{align*}
Inserting this back into~\eqref{eqprop:Sg:s+t} confirms the semigroup property: $ \Sg(t)\Sg(s)=\Sg(t+s) $.

We now turn to showing that $ \lim_{t\downarrow 0} \frac{1}{t} (\Sg(t)g-g) = \Ham g $, for all $ g\in D(\Ham) $.
Recall that $ \Ham $ satisfies \eqref{eq:ham:form}.
This being the case, it suffices to show
\begin{align}
	\label{eqprop:Sg:generator}
	\lim_{t\downarrow 0} \tfrac{1}{t} \big( \langle f,\Sg(t)g\rangle -\langle f, g \rangle \big)
	=
	-F_\Ham(f,g)
	:=
	-\tfrac12 \langle f',g' \rangle + f(1)g(1) \Rtlim(1) - \int_{0}^1 (f'g+fg')(x) \Rtlim(x) dx,
\end{align}
for all $ f,g\in H^1(\limT) $.
The operator $ \Sg(t) $, by definition, is given by the series~\eqref{eq:Sgker}.
This being the case, we consider separately the contribution from each term in the series.
First, for the heat kernel, with $ f,g\in H^1(\limT) $, it is standard to show that
\begin{align}
	\label{eqprop:Sg:HK}
	\lim_{t\downarrow 0} \tfrac{1}{t} \big( \langle f,\HK(t)g\rangle -\langle f, g \rangle \big)
	=
	-\tfrac12 \langle f',g' \rangle.
\end{align}

Next we turn to the $ n=1 $ term.
Recall the given expressions~\eqref{eq:Sgr}--\eqref{eq:SgrI} for $ \Sgr_1 $.
With the notation $ \phi(t;x) := \int_{\limT} \HK(t;x,y)\phi(y)dy $ for a given function $ \phi $,
we write
\begin{align*}
	\frac{1}{t} \langle f,\Sgr_1(t)g\rangle
	=
	\frac{1}{t} \int_{0}^t \Big( \int_{\limT} f(s;x) d\Rtlim(x) g(t-s;x) \Big) ds. 	
\end{align*}
Integrate by parts in $ x $ gives
\begin{align}
	\label{eqprop:Sg:Sgr1}
	\frac{1}{t} \langle f,\Sgr_1(t)g\rangle
	=
	\frac{1}{t} \int_{0}^t \Bigg(
		\Rtlim(1)f(s;1) g(t-s;1)
		-
		\int_{0}^1 \big( (\partial_x f(s;x)) g(t-s;x) + f(s;x) (\partial_x g(t-s;x)) \big) \Rtlim(x) dx
	\Bigg) ds.
\end{align}
For $ \phi\in H^1(\limT) $, it is straightforward to check that
$ \Vert \phi(t;\Cdot) - \phi(\Cdot) \Vert_{H^1(\limT)} \to 0 $ as $ t\downarrow 0 $.
Also, with $ \limT $ having unit (and hence finite) Lebesgue measure,
$ L^2 $-norms and $ L^\infty $-norms are controlled by the $ H^1 $-norms:
\begin{align*}
	\norm{\psi}_{L^2(\limT)},\norm{\psi}_{L^\infty(\limT)} \leq\norm{\psi}_{H^1(\limT)},
\end{align*}
so $ \Vert \phi(t;\Cdot) - \phi(\Cdot) \Vert_{L^2(\limT)} \to 0 $
and $ \Vert \phi(t;\Cdot) - \phi \Vert_{L^\infty(\limT)} \to 0 $, as $ t\downarrow 0 $.
Using these properties for $ \phi=f,g $ in~\eqref{eqprop:Sg:Sgr1},
together with $ \Rtlim \in L^\infty[0,1] $, we send $ t\downarrow 0 $ to obtain
\begin{align}
	\label{eqprop:Sg:Sgr1:}
	\lim_{t\downarrow 0} \tfrac{1}{t} \langle f,\Sgr_1(t)g\rangle
	=
	\Rtlim(1)f(1) g(1)
	-
	\int_{0}^1 \big( (f'g + fg \big)(x) \Rtlim(x) dx.
\end{align}

Finally we consider the $ n\geq 2 $ terms.
Given the expressions~\eqref{eq:Sgr}--\eqref{eq:SgrI} for $ \Sgr_n $,
we write
\begin{align}
	\label{eqprop:Sg:Sgrn:1}
	\langle f,\Sgr_n(t)g\rangle
	=
	\int_{\limT^2} f(x) \Big( \int_{\Sigma_n(t)} \SgrI_n(\vec{s};x,\tilx) \Big) g(\tilx) dxd\tilx.
\end{align}
With $ f,g\in H^1(\limT) $, we have $ \Vert f \Vert_{L^\infty(\limT)},\Vert g \Vert_{L^\infty(\limT)} <\infty $.
Thus, in~\eqref{eqprop:Sg:Sgrn:1}, bound $ f(x) $ and $ g(x') $ by their supremum,
and use Lemma~\ref{lem:SgrIbd}\ref{lem:SgrI:int} for fixed $ v\in(0,\urt) $.
This gives
\begin{align*}
	|\langle f,\Sgr_n(t)g\rangle|
	\leq
	\Vert f \Vert_{L^\infty(\limT)} \Vert g \Vert_{L^\infty(\limT)}
	t^{\frac{(1+v)n}{2}}\frac{ (c(v,T) \norm{\Rtlim}_{C^{\urt}[0,1]})^n}{\Gamma(\frac{(1+v)n+2}{2})}.	
\end{align*}
Sum this inequality over $ n\geq 2 $, and divide the result by $ t $.
This gives, for all $ t\leq 1 $,
\begin{align}
	\label{eqprop:Sg:Sgrn:2}
	\frac{1}{t}\sum_{n\geq 2} |\langle f,\Sgr_n(t)g\rangle|
	\leq
	c(f,g,\urt) t^{v/2}.
\end{align}
The r.h.s.\ of~\eqref{eqprop:Sg:Sgrn:2} indeed converges to $ 0 $ as $ t\downarrow 0 $.

Combining \eqref{eqprop:Sg:HK}, \eqref{eqprop:Sg:Sgr1:}, and \eqref{eqprop:Sg:Sgrn:2}
concludes the desired result~\eqref{eqprop:Sg:generator}.
\end{proof}

We close this subsection by showing the uniqueness of mild solutions~\eqref{eq:spde:mild} of \eqref{eq:spde}.
(Recall that existence follows from Theorem~\ref{thm:main}.)
The argument follows standard Picard iteration the same way as for the \ac{SHE}.
\begin{proposition}
\label{prop:unique}
For any given $ \limZ^\ic \in C^{\uic}(\limT) $
and a fixed $ \Rtlim \in C^{\urt}[0,1] $,
there exists at most one $ C([0,\infty),C(\R)) $-valued mild solution~\eqref{eq:spde:mild}.
\end{proposition}
\begin{proof}
Let $ \limZ\in C([0,\infty),C(\T)) $ be a mild solution~\eqref{eq:spde:mild} solving \eqref{eq:spde}.
Iterating~\eqref{eq:spde:mild} $ m $-times gives
\begin{align*}
	\limZ(t,x) =
	\sum_{n=0}^m\limZ_n(t,x)
	+
	\mathcal{W}_m(t,x),
\end{align*}
where, with the notation $ [0,t]^n_{<} := \{ (t_1,\ldots,t_n)\in(0,\infty)^n : 0<t_1<\ldots<t_n<t_{n+1}:=t\} $,
\begin{align*}
	\limZ_n(t,x)
	&:=
	\int_{[0,t]^{n}_{<}\times\limT^{n+1}}
	\Big( \prod_{i=1}^n \Sg(t_{i+1}-t_i;x_{i+1},x_i) \xi(t_i,x_i) dt_idx_i \Big) \Sg(t_1;x_1,x_0) \limZ^\ic(x_0) dx_0,
\\
	\mathcal{W}_m(t,x)
	&:=
	\int_{[0,t]^{m+1}_{<}\times\limT^{m+1}}
	\Big( \prod_{i=1}^{m+1} \Sg(t_{i+1}-t_i;x_{i+1},x_i) \xi(t_i,x_i) dt_idx_i \Big) \limZ(t_1,x_1).	
\end{align*}
For given $ \Lambda<\infty $, let $ \tau_{\Lambda} := \inf\{t \geq 0: \sup_{x\in\limT} \limZ(t,x)^2 > \Lambda \} $
denote the first hitting of $ \limZ^2 $ at $ \Lambda $.
Recall that $ \Sg(t) $ is deterministic since $ \Rtlim $ is assumed to be deterministic (throughout this section).
Evaluating the second moment of $ \mathcal{W}_m(t\wedge\tau_\Lambda,x) $ gives
\begin{align*}
	\Ex\big[\mathcal{W}_m(t\wedge\tau_\Lambda,x)^2\big]
	&=
	\Ex\Big[
		\int_{[0,\tau_\Lambda]^{m+1}_{<}\times\limT^{m+1}}
		\Big( \prod_{i=1}^{m+1} \Sg^2(t_{i+1}-t_i;x_{i+1},x_i) dt_idx_i \Big) \limZ^2(t_1,x_1)
		\Big]
\\
	&\leq
	\Lambda
	\int_{[0,t]^{m+1}_{<}\times\limT^{m+1}}
	\prod_{i=1}^{m+1} \Sg^2(t_{i+1}-t_i;x_{i+1},x_i) dt_idx_i.
\end{align*}
Further applying bounds from Proposition~\ref{prop:Sg}\ref{prop:Sg:int}, \ref{prop:Sg:sup} gives
\begin{align*}
	\Ex\big[\mathcal{W}_m(t\wedge\tau_\Lambda,x)^2\big]
	&\leq
	\Lambda c(t)^{m+1}
	\int_{[0,t]^{m+1}_{<}\times\limT^{m+1}}
	\prod_{i=1}^{m+1} dt_idx_i
	=
	\Lambda \frac{c(t)^{m+1}}{(m+1)!}.
\end{align*}
Sending $ m\to\infty $ gives $ \Ex[\mathcal{W}_m(t\wedge\tau_\Lambda,x)^2] \to 0 $.
With $ \limZ $ being $ C([0,\infty)\times\limT) $ by assumption,
we have $ \Pr[ \tau_{\Lambda} > t ]\to 1 $, as $ \Lambda\to\infty $.
Hence, after passing to a suitable sequence $ \Lambda_m\to\infty $,
we conclude $ \mathcal{W}_m(t,x) \to_\text{P} 0 $, as $ m\to\infty $, for each fixed $ (t,x) $.
This gives
\begin{align*}
	\limZ(t,x) =
	\lim_{m\to\infty} \sum_{n=0}^m\limZ_n(t,x),
\end{align*}
for each $ (t,x) $.
Since each $ \limZ_n $ is a function of $ \limZ^\ic $ and $ \xi $, uniqueness of $ \limZ(t,x) $ follows.
\end{proof}

\subsection{Microscopic}
\label{sect:sg}
Our goal is to bound the kernel $ \sg(t;x,\tilx) $ of the microscopic semigroup.
Recall the operator $ \ham $ defined in~\eqref{eq:ham} and the definition of $ \nu $ from~\eqref{eq:Z}.
Various operators and parameters (e.g., $ \ham $, $ \nu $) depends on $ N $,
but, to alleviate heavy notation, we often omit the dependence in notation.
Under weak asymmetry scaling~\eqref{eq:was},
\begin{align}
	\label{eq:nuN}
	\nu=\tfrac{1}{N} + O(\tfrac{1}{N^2}).
\end{align}
Set $ f=\ind_\set{\tilx} $ in the Feynman--Kac formula~\eqref{eq:feynmankac} to get
\begin{align*}
	\sg(t;x,\tilx)
	=
	\big( \sg(t)\ind_\set{\tilx} \big)(x)
	=
	\Ex_x\Big[ e^{\int_0^t \nu\rt(X^\rt(s))ds} \ind_\set{\tilx}(X^a(t)) \Big],
\end{align*}
where $ X^\rt(t) $ denotes the inhomogeneous walk defined in Section~\ref{sect:hc}.
Taylor-expanding the exponential function, and exchanging the expectation with the sums and integrals yields
\begin{align}
	\notag
	\sg(t;x,\tilx)
	&=
	\Ex_{x}\big[\ind_\set{\tilx}(X^\rt(t))\big]
	+
	\sum_{n=1}^\infty
	\int_{\Sigma_n(t)} \Ex_{x}\Big[ \prod_{i=1}^n \nu \, \rt(X^\rt(s_0+\ldots+s_{i-1})) \ind_\set{\tilx}(X^\rt(t)) \Big] d^n\vec{s}
\\
	\label{eq:sgker}
	&=
	\hka(t;x,\tilx)
	+
	\sum_{n=1}^\infty \sgr_n(t;x,\tilx),
\end{align}
where we assume (and will show below that) the sum in~\eqref{eq:sgr} converges, and the term $ \sgr_n(t;x,\tilx) $ is defined as
\begin{align}
	\label{eq:sgr}
	\sgr_n(t;x,\tilx)
	&:=
	\int_{\Sigma_{n}(t)} \sgrI_n(\vec{s};x,\tilx) d^n\vec{s},
\\
	\label{eq:sgrI}
	\sgrI_n(\vec{s};x,\tilx)
	&:=
	\sum_{x_1,\ldots,x_n\in\T}
	\prod_{i=0}^n\hka(s_i;x_i,x_{i+1}) \prod_{i=1}^n\nu \rt(x_i).
\end{align}
Note that, unlike in the continuum case in Section~\ref{sect:Sg},
the expansions~\eqref{eq:sgker} here is rigorous,
provided that the sums and integrals in~\eqref{eq:sgr}--\eqref{eq:sgrI} converge.

We now proceed to establish bounds that will guarantee the convergence of the sums and integrals.
Our treatment here parallels Section~\ref{sect:Sg},
starting with a summation-by-part formula.

Similarly to our treatment in Section~\ref{sect:Sg}, here we need to partition $ \T $ into two pieces
according to a given pair $ y,\tily\in\T $.
Unlike in the macroscopic (i.e., continuum) case, here we cannot ignore $ y=\tily $.
Given $ y,\tily \in \T $, we define
\begin{align*}
	\parti_1(y,\tily) := \big\{ x\in \T: \dist_{\T}(x,y) \leq \dist_{\T}(x,\tily) \wedge \tfrac{N}2 \big\},
	\qquad
	\parti_2(y,\tily) := \big\{ x\in \T: \dist_{\T}(x,\tily) < \dist_{\T}(x,y) \wedge \tfrac{N}2 \big\}.
\end{align*}
The intervals $ \parti_1(y,\tily) $ and $ \parti_2(y,\tily) $
are the macroscopic analog of $ \Parti_1(y,\tily) $ and $ \Parti_2(\tily,y) $, respectively.
In particular, $ \parti_1(y,\tily) $ and $ \parti_2(y,\tily) $ partition $ \T $ into two pieces, with
\begin{align}
	\label{eq:parti:dist}
	\begin{split}
	&\dist_{\T}(y_1,x) \leq \dist_{\T}(y_2,x)+1 \leq 2\dist_{\T}(y_2,x),\ \forall x\in \parti_1(y_1,y_2),
\\	
	&\dist_{\T}(y_2,x) \leq \dist_{\T}(y_1,x),\ \forall x\in \parti_2(y_1,y_2).
\end{split}
\end{align}
Write $ \midp_1(y,\tily),\midp_2(y,\tily)\in\T  $ for the boundary points of $ \parti_1(y,\tily) $ and $ \parti_2(y,\tily) $.
More precisely, $ \parti_1(y,\tily) = [ \midp_1(y,\tily),\midp_2(y,\tily) ) $
and $ \parti_2(y,\tily) = [ \midp_2(y,\tily),\midp_1(y,\tily) ) $.
Recall the definition of $ \Rt(y,x) $ from~\eqref{eq:Rt}.
\begin{lemma}
\label{lem:ips}
Set
\begin{align}
\label{eq:ips:}
\begin{split}
	\ips(s,s';y_1,y_2)
	:=
	\sum_{j=1}^2
	\Big(&
		\hka(s;y_1,x) \, \nu \, \Rt(y_j,x) \hka(s';x-1,y_2)\big|_{x=\midp_j(y_1,y_2)}^{x=\midp_{j+1}(y_1,y_2)+1}
\\
		&-
		\int_{\parti_1(y_1,y_2)} \big(\nabla_{x}\hka(s;y_1,x)\big) \, \nu \, \Rt(y_j,x) \hka(s';x+1,y_2) dx
\\
		&-
		\int_{\parti_2(y_1,y_2)} \hka(s;y_1,x) \, \nu \, \Rt(y_j,x) \nabla_{x}\hka(s';x,y_2) dx
	\Big),
\end{split}
\end{align}
where, by convention, we let $ \midp_{3}(y_1,y_2) := \midp_{1}(y_1,y_2) $.
Then we have that
\begin{align}
	\label{eq:ips}
	\sgrI_n(\vec{s};x,\tilx)
	=
	\sum_{y_1,\ldots,y_{n+1}\in\T}
	\hka(\tfrac{s_0}{2};x,y_1)
		\Big( \prod_{i=1}^n \ips(\tfrac{s_{i-1}}{2},\tfrac{s_{i}}{2};y_{i},y_{i+1})  \Big)
	\hka(\tfrac{s_{n}}{2};y_{n+1},\tilx).
\end{align}
\end{lemma}
\begin{proof}
Use the semigroup property $ \hka(s_i;x_i,x_{i+1})=\sum_{y_i\in\T} \hka(\frac{s_i}{2};x_i,y_i)\hka(\frac{s_i}{2};y_i,x_{i+1}) dy_i $ in~\eqref{eq:sgrI} to rewrite
\begin{align}
	\label{eqlem:ips}
	&\sgrI_n(\vec{s};x,\tilx)
	=
	\sum_{y_1,\ldots,y_{n+1}\in\T}
	\hka(\tfrac{s_0}{2};x,y_1)
	\,\til{\ips}_1 (s_{0},y_1,s_1,y_{2})
	 \cdots \til{\ips}_n(s_{n},y_i,s_n,y_{n+1})
	\,
	\hka(\tfrac{s_{n}}{2};y_{n+1},\tilx),
\\
	\label{eqlem:ips:}
	&\til{\ips}_i(s_{i-1},y_i,s_i,y_{i+1})
	:=
	\sum_{x\in\T} \hka(\tfrac{s_{i-1}}{2};y_{i},x) \, \nu \, \rt(x) \, \hka(\tfrac{s_{i}}{2};x,y_{i+1}).
\end{align}
In~\eqref{eqlem:ips:}, divide the sum over $ \T $ into sums over $ \parti_1(y_i,y_{i+1}) $ and $ \parti_2(y_i,y_{i+1}) $ to get
$ \til{\ips}_i=\til{\ips}_{i,1}+\til{\ips}_{i,2} $, where
\begin{align*}
	\til{\ips}_{i,j}(s_{i-1},y_i,s_i,y_{i+1})
	&:=
	\sum_{x\in\parti_1(y_i,y_{i+1})} \hka(\tfrac{s_{i-1}}{2};y_{i},x) \, \nu \, \rt(x) \, \hka(\tfrac{s_{i}}{2};x,y_{i+1})
\\
	&=
	\sum_{x\in\parti_2(y_i,y_{i+1})} \hka(\tfrac{s_{i-1}}{2};y_{i},x) \, \nu
	\,
	\big( \nabla_{x}\Rt(y_{i+j-1},x-1) \big) \, \hka(\tfrac{s_{i}}{2};x,y_{i+1}).
\end{align*}
Apply summation by parts
\begin{align*}
	\sum_{x\in[x_1,x_2)} f(x)\nabla g(x-1) = - \sum_{x=[x_1,x_2)} \big( \nabla f(x) \big) g(x) + f(x_2+1)g(x_2) - f(x_1)g(x_1-1)
\end{align*}
with $ f(x)=\Rt(y_{i+j-1},x) $ and $ g(x)=\hka(\tfrac{s_{i-1}}{2};y_{i},x)\hka(\tfrac{s_{i}}{2};x,y_{i+1}) $
for $ j=1,2 $, and add the results together.
We then conclude $ \til{\ips}_{i}(s_{i-1},y_i,s_i,y_{i+1})= \ips(\frac{s_{i-1}}{2},\frac{s_{i}}{2};y_i,y_{i+1}) $.
Inserting this back to~\eqref{eqlem:ips} completes the proof.
\end{proof}

Given the summation-by-parts formula \eqref{eq:ips:}, we proceed to establish bounds on $ \ips $.
Unlike in the macroscopic case, where we assume $ \Rtlim $ to be deterministic,
the treatment of microscopic semigroup needs to address the randomness of $ \rt $.
Recall the terminology `with probability $ \to_{\Lambda,N} 1 $' from~\eqref{eq:wp1}.
\begin{lemma}
\label{lem:ipsbd}
Given any $ v\in(0,\urt) $ and  $ T<\infty $, the following holds with probability $ \to_{\Lambda,N} 1 $:
\begin{align*}
	\sum_{y'\in\T} |\ips(s,s';y,y')|
	\leq
	\frac{ \Lambda \, c(v,T) }{ N^{1+v} } \big( (1+s)^{-(1-v)/2}+(1+s')^{-(1-v)/2} \big),
	\qquad
	\forall s,s'\in[0,N^2T], y\in \T.
\end{align*}
\end{lemma}
\begin{proof}
Recall the definition of the seminorm $ \hold{\,\Cdot\,}_{u,N} $ from~\eqref{eq:hold}.
With $ v\leq \urt $,
we have $ |\Rt(y_j,x)| \leq (\frac{|(y_j,x]|}{N})^{v} \hold{\Rt}_{\urt,N} $.
Further, by~\eqref{eq:parti:dist}, we have $ |(y_j,x]| \leq 2\dist_{\T}(y_1,y_2) $, for all $ x\in\parti_j(y_1,y_2) $.
Hence
\begin{align*}
	|\Rt(y_j,x)|
	\leq
	2\dist_{\T}(y_j,x)^{v} N^{-v} \, \hold{ \Rt_N }_{\urt},
	\qquad
	\forall x\in\parti_j(y_1,y_2).
\end{align*}
Using this bound in~\eqref{eq:ips:}, together with $ |\nu| \leq \frac{c}{N} $ (from~\eqref{eq:nuN}), we obtain
\begin{align}
	\label{eqlem:ipsbd:1}
	|\ips(s,s';y_1,y_2)|
	\leq
	\frac{ \hold{\Rt}_{\urt,N} }{ N^{1+v} }
	\sum_{j=1}^2
	\Bigg(&
		\sum_{ x\in\{\midp_i(y_1,y_2)\}_{i=1}^2 } \hka(s;y_1,x) \, \dist_{\T}(y_j,x)^{v} \, \hka(s';x,y_2)
\\
	\label{eqlem:psbd:2}
	&+
		\sum_{x\in\parti_1(y_1,y_2)}  \big|\nabla_{x}\hk(s;y_1,x)\big| \dist_{\T}(y_j,x)^{v} \, \hka(s';x+1,y_2)
\\
	\label{eqlem:psbd:3}
	&
	+
		\sum_{x\in\parti_2(y_1,y_2)} \hka(s;y_1,x) \, \dist_{\T}(y_j,x)^{v} \big|\nabla_{x}\hka(s;x,y_2)\big|
	\Bigg).
\end{align}
In~\eqref{eqlem:psbd:2}, use~\eqref{eq:parti:dist} to bound
$ \dist_{\T}(y_j,x)^{v} $ by $ 2 \dist_{\T}(y_1,x)^v $,
and in~\eqref{eqlem:psbd:3}, use~\eqref{eq:parti:dist} to bound
$ \dist_{\T}(y_j,x)^{v} $ by $ 2 \dist_{\T}(y_2,x)^v $.
This gives
\begin{align}
\label{eq:ips:bd:}
\begin{split}
	&|\ips(s,s';y_1,y_2)|
	\leq
	\frac{ \hold{\Rt}_{\urt,N} }{ N^{1+v} }
	\Bigg(
	2 \sum_{ x\in\{\midp_i(y_1,y_2)\}_{i=1}^2 }
	\hka(s;y_1,x) \, \dist_{\T}(y_j,x)^{v} \, \hka(s';x,y_2)
\\
	&
	+
	2\sum_{x\in\T} |\nabla_{x}\hka(s;y_1,x)\big| \dist_{\T}(y_1,x)^{v} \, \hka(s';x+1,y_2)
	+
	2\sum_{x\in\T} \hka(s;y_1,x) \, \dist_{\T}(y_2,x)^{v} \, \big|\nabla_{x}\hka(s;x,y_2)\big|
	\Bigg).
\end{split}
\end{align}
For the kernel $ \hk(t,x,\tilx) $ of the homogeneous walk,
we indeed have $ \sum_{\tilx\in\T} \hk(t,x,\tilx) = 1 $.
Since $ \hka=\hk+\hkr $, the preceding identity together with the bound from Proposition~\ref{prop:hk}\ref{cor:hk:hkrsum}
yields, with probability $ \to_{\Lambda,N} 1 $,
\begin{align}
	\label{eq:hk:hka:summ}
	\sum_{\tilx\in\T} \hka(t,x,\tilx) \leq 2.
\end{align}
Now sum~\eqref{eq:ips:bd:} over $ y_2\in\T $,
and use \eqref{eq:hk:hka:summ} and Proposition~\ref{prop:hk}\ref{cor:hk:hkasup}, \ref{cor:hk:hkahold:sup}--\ref{cor:hk:hkahold::} to bound the result.
This gives
\begin{align*}
	&\sum_{y_2\in\T}|\ips(s,s';y_1,y_2)|
	\leq
	\frac{ c(v,T) \hold{\Rt}_{\urt,N} }{ N^{1+v} }
	\big(
		(s+1)^{-(1-v)/2} + (s'+1)^{-(1-v)/2} + (s+1)^{-(1-v)/2} + (s'+1)^{-(1-v)/2}
	\big).
\end{align*}
Recalling Assumption~\ref{assu:rt}\ref{assu:rt:holder} on $ \hold{\Rt}_{\urt,N} $,
we conclude the desired result.
\end{proof}

Based on Lemmas~\ref{lem:ips}--\ref{lem:ipsbd}, we now establish bounds on $ \sgrI_n $.
\begin{lemma}
\label{lem:sgrIbd}
Given any $ u\in(0,1] $, $ v\in(0,\urt) $, and $ T<\infty $, the following hold with probability $ \to_{\Lambda,N} 1 $:
\begin{enumerate}[label=(\alph*),leftmargin=7ex]
\item \label{lem:sgrI:sum}
	$
	\displaystyle
	\sum_{\tilx\in\T} \int_{\Sigma_n(t)} |\sgrI_n(\vec{s};x,\tilx)| d^n\vec{s}
	\leq
	(tN^{-2})^{\frac{1+v}2} \frac{ \Lambda^n}{\Gamma(\frac{(1+v)n+2}{2})}
	$,
	\qquad $ t\in[0,N^2T] $, $ \forall x\in\T $,  $ n\in\Z_{>0} $,
\item \label{lem:sgrI:gsum}
	$
	\displaystyle	
	\sum_{\tilx\in\T}
	\int_{\Sigma_n(t)} |\sgrI_n(\vec{s};x,\tilx)-\sgrI_n(\vec{s};x',\tilx)| d^n\vec{s}\
	\leq
	\frac{\dist_{\T}(x,x')^{u}}{(t+1)^{u/2}} (tN^{-2})^{\frac{1+v}2} \frac{\Lambda^n }{\Gamma(\frac{(1+v)n+2-u}{2})},
	$
	\newline $ t\in[0,N^2T] $, $ \forall x\in\T $, $ n\in\Z_{>0} $,
\item \label{lem:sgrI:sup}
	$
	\displaystyle
	\int_{\Sigma_n(t)} |\sgrI_n(\vec{s};x,\tilx)| d^n\vec{s}
	\leq
	\frac{(tN^{-2})^{\frac{1+v}2} }{(t+1)^{1/2}} \frac{\Lambda^n}{\Gamma(\frac{(1+v)n+1}{2})},
	$
	\qquad $ t\in[0,N^2T] $, $ \forall x\in\T $, $ n\in\Z_{>0} $,
\item \label{lem:sgrI:gsup}
	$
	\displaystyle
	\int_{\Sigma_n(t)} |\sgrI_n(\vec{s};x,\tilx)-\sgrI_n(\vec{s};x',\tilx)| d^n\vec{s}
	\leq
	\frac{\dist_{\T}(x,x')^{u}}{(t+1)^{(1+u)/2}}(tN^{-2})^{\frac{1+v}2} \frac{\Lambda^n}{\Gamma(\frac{(1+v)n+1-u}{2})},
	$
	\qquad $ t\in[0,N^2T] $, $ \forall x\in\T $,  $ n\in\Z_{>0} $.
\end{enumerate}
\end{lemma}
\begin{proof}
The proof follows by the same line of calculation as in the proof of Lemma~\ref{lem:SgrIbd},
with $ \hka $, $ \ips $, $ \sgrI_n $ replacing $ \HK $, $ \Ips $, $ \SgrI $,
and with sums replacing integrals accordingly.
In particular, in place of~\eqref{eqSglem:a}--\eqref{eqSglem:d}, here we have,
with probability $ \to_{\Lambda,N} 1 $,
\begin{subequations}
\label{eqsglem:}
\begin{align}
	\label{eqsglem:a}
	\sum_{\tilx\in\Z} |\sgrI_n(\vec{s};x,\tilx)| d\tilx
	&\leq
	\big( N^{-1-v}\Lambda \big)^n \prod_{i=1}^n \big( s_{i-1}^{-(1-v)/2}+s_{i}^{-(1-v)/2} \big),
\\
	\label{eqsglem:b}	
	\sum_{\tilx\in\Z} |\sgrI_n(\vec{s};x,\tilx)-\sgrI_n(\vec{s};x',\tilx)| d\tilx
	&\leq
	\big( N^{-1-v}\Lambda \big)^n \dist_{\T}(x,x')^{u} s_0^{-\frac{u}{2}} \prod_{i=1}^n \big( s_{i-1}^{-(1-v)/2}+s_{i}^{-(1-v)/2} \big),
\\	
	|\sgrI_n(\vec{s};x,\tilx)|
	&\leq
	\big( N^{-1-v}\Lambda \big)^n \prod_{i=1}^n \big( s_{i-1}^{-(1-u)/2}+s_{i}^{-(1-u)/2} \big) s_n^{-\frac12},
\\
	\label{eqsglem:d}
	|\sgrI_n(\vec{s};x,\tilx)-\sgrI_n(\vec{s};x',\tilx)|
	&\leq
	\dist_{\T}(x,x')^{u}
	\big( N^{-1-v}\Lambda \big)^n s_0^{-\frac{u}{2}} \prod_{i=1}^n \big( s_{i-1}^{-(1-v)/2}+s_{i}^{-(1-v)/2} \big) s_n^{-\frac12}.
\end{align}
\end{subequations}
Given~\eqref{eqsglem:a}--\eqref{eqsglem:d},
the rest of the proof follows by applying the Dirichlet formula~\eqref{eq:dirichlet}.
We omit repeating the argument.
\end{proof}

We now proceed to establish bounds on $ \sgr $.
In the following, we will often decompose $ \sgr $ into the sum of $ \hk $, the kernel of the homogeneous walk on $ \T $,
and a remainder term $ \sgr := \sg-\hk $.
Recall from~\eqref{eq:hkr} that $ \hka(t;x,\tilx)=\hk(t;x,\tilx)+\hkr(t;x,\tilx) $.
Referring to \eqref{eq:sgker}, we see that
\begin{align*}
	\sgr(t;x,\tilx) := \sg(t;x,\tilx)-\hk(t;x,\tilx) = \hkr(t;x,\tilx) + \sum_{n=1}^\infty \sgr_n(t;x,\tilx).
\end{align*}

\begin{proposition}
\label{prop:sg}
Fix  $ u\in(0,1] $, $ v\in(0,\urt) $ and $ T<\infty $.
The following hold with probability $ \to_{\Lambda,N} 1 $:
\begin{enumerate}[label=(\alph*),leftmargin=7ex]
\item \label{prop:sg:sum}
	$
		\displaystyle
		\sum_{\tilx\in\T} \sg(t;x,\tilx)
		\leq
		\Lambda,
	$
	\qquad $ t\in[0,N^2T] $, $ \forall x\in\T $,
\item \label{prop:sg:sup}
	$
		\displaystyle
		\sg(t;x,\tilx)
		\leq
		\frac{\Lambda}{\sqrt{t+1}},
	$
	\qquad $ t\in[0,N^2T] $, $ \forall x,\tilx\in\T $,
\item \label{prop:sg:gsup}
	$
		\displaystyle
		|\sg(t;x,\tilx)-\sg(t;x',\tilx)|
		\leq
		\frac{\dist_{\T}(x,x')^u}{(t+1)^{(u+1)/2}}\Lambda,
	$
	\qquad  $ t\in[0,N^2T] $, $ \forall x,x',\tilx\in\T $,
\item \label{prop:sg:gsum}
	$
		\displaystyle
		\sum_{\tilx\in\T}|\sg(t;x,\tilx)-\sg(t;x',\tilx)|
		\leq
		\frac{\dist_{\T}(x,x')^u}{(t+1)^{u/2}}\Lambda,
	$
	\qquad  $ t\in[0,N^2T] $, $ \forall x,x'\in\T $,
\item \label{prop:sgr:sum}
	$
		\displaystyle
		\sum_{\tilx\in\T} |\sgr(t;x,\tilx)|
		\leq
		(tN^{-2})^{v}\Lambda,
	$
	\qquad $ t\in[0,N^2T] $, $ \forall x\in\T $,
\item \label{prop:sgr:gsum}
	$
		\displaystyle
		\sum_{\tilx\in\T}|\sgr(t;x,\tilx)-\sgr(t;x',\tilx)|
		\leq
		\Big( \frac{\dist_{\T}(x,x')^u}{(1+t)^{u/2}}N^{-v} + \Big(\frac{\dist_{\T}(x,x')}{N}\Big)^u \Big)\Lambda,
	$
	\quad $ t\in[0,N^2T] $, $ \forall x,x'\in\T $,
\item \label{prop:sgr:gsup}
	$
		\displaystyle
		|\sgr(t;x,\tilx)-\sgr(t;x',\tilx)|
		\leq
		\Big( \frac{\dist_{\T}(x,x')^u}{(1+t)^{(u+1)/2}}N^{-v} + \frac{(\dist_{\T}(x,x')/N)^u}{(1+t)^{1/2}} \Big)\Lambda,
	$
	\quad $ t\in[0,N^2T] $, $ \forall x,x',\tilx\in\T $,
\item \label{prop:sg:loc}
	$
		\displaystyle
		\sup_{t'\in[t,t+1]}\sg(t';x,\tilx)
		\leq
		\Lambda \sg(t+1;x,\tilx),
	$
	\qquad $ t\in[0,N^2T] $, $ \forall x,\tilx\in\T $.
\end{enumerate}
\end{proposition}
\begin{proof}
Let $ \til{\sgr}(t;x,\tilx) := \sum_{n\geq 1} \sgr_n(t;x,\tilx) $.
Summing the r.h.s.\ of Lemma~\ref{lem:sgrIbd}\ref{lem:sgrI:sum}--\ref{lem:sgrI:gsup} gives,
with probability $ \to_{\Lambda,N} 1 $,
\begin{enumerate}[label=(\Roman*),leftmargin=7ex]
\item \label{sgr:sum}
	$
	\displaystyle
	\sum_{\tilx\in\T} |\til{\sgr}(t;x,\tilx)|
	\leq
	(tN^{-2})^{\frac{1+v}2} \Lambda
	$,
%
\item \label{sgr:gsum}
	$
	\displaystyle	
	\sum_{\tilx\in\T}
	|\til{\sgr}(t;x,\tilx)-\til{\sgr}(t;x,\tilx)|
	\leq
	\frac{\dist_{\T}(x,x')^{u}}{(t+1)^{u/2}} (tN^{-2})^{\frac{1+v}2} \Lambda
	$,
%
\item \label{sgr:sup}
	$
	\displaystyle
	|\til{\sgr}(t;x,\tilx)|
	\leq
	\frac{(tN^{-2})^{\frac{1+v}2} }{(t+1)^{1/2}} \Lambda
	$,
%
\item \label{sgr:gsup}
	$
	\displaystyle
	|\til{\sgr}(t;x,\tilx)-\til{\sgr}(t;x',\tilx)|
	\leq
	\frac{\dist_{\T}(x,x')^{u}}{(t+1)^{(1+u)/2}}(tN^{-2})^{\frac{1+v}2} \Lambda
	$.
\end{enumerate}

Given that $ \sg(t)=\hka(t)+\til{\sgr}(t) $:
\begin{itemize}
\item \ref{prop:sg:sum} follows by combining $ \sum_{\tilx} \hka(t;x,\tilx)=1 $
	and \ref{sgr:sum};
\item \ref{prop:sg:sup} follows by combining Proposition~\ref{prop:hk}\ref{cor:hk:hkasup}
	and \ref{sgr:sup};
\item \ref{prop:sg:gsup} follows by combining Proposition~\ref{prop:hk}\ref{cor:hk:hkagdsup}
	and \ref{sgr:gsup};
\item \ref{prop:sg:gsum} follows by combining Proposition~\ref{prop:hk}\ref{cor:hk:hkagdsum}
	and \ref{sgr:gsum}.
\end{itemize}

Given that $ \sgr(t)=\hkr(t)+\til{\sgr}(t) $,
\begin{itemize}
\item \ref{prop:sgr:sum} follows by combining Proposition~\ref{prop:hk}\ref{cor:hk:hkrsum}
	and \ref{sgr:sum} (note that $ (tN^{-2})^{(v+1)/2} \leq c(T)(tN^{-2})^{v/2} $).
\item With $ t\leq N^2T $, we have $ \frac{1}{(t+1)^{u/2}} (tN^{-2})^{\frac{1+v}2} \leq c(T) N^{-v} $.
	Hence, by \ref{sgr:gsum}, with probability $ \to_{\Lambda,N} 1 $,
	\begin{align*}
		\sum_{\tilx\in\T} |\sgr(t;x,\tilx)| \leq \frac{\dist_{\T}(x,x')^u}{N^u} \Lambda.
	\end{align*}
	Combining this with Proposition~\ref{prop:hk}\ref{cor:hk:hkrgsum} gives~\ref{prop:sgr:gsum}.
\item Similarly to the preceding, by \ref{sgr:gsup}, with probability $ \to_{\Lambda,N} 1 $,
	\begin{align*}
		\sum_{\tilx\in\T} |\sgr(t;x,\tilx)-\sgr(t;x',\tilx)| \leq
		\frac{(\dist_{\T}(x,x')/N)^u}{(1+t)^{1/2}} \Lambda.
	\end{align*}
	Combining this with Proposition~\ref{prop:hk}\ref{cor:hk:hkrgsup} gives~\ref{prop:sgr:gsup}.
\item Finally, to show~\ref{prop:sg:loc}, we fix $ t'\in[t,t+1] $, and set $ \delta:=t+1-t'\leq 1 $.
With $ \sg(t;x,y) \geq 0 $, we write
\begin{align}
	\label{eq:sgloc:}
	\sg(t+1;x,\tilx) = \sum_{y\in\T} \sg(\delta;x,y) \sg(t';y,\tilx) \geq \sg(\delta;x,x) \sg(t';x,\tilx).
\end{align}
Given that $ \delta \leq 1 $, we indeed have $ \hk(\delta;x,x) \geq \Pr_x[ X(s)=x, \forall s\in[0,1] ] \geq \frac{1}{c} $.
With $ \sg(\delta)=\hk(\delta)+\hkr(\delta)+\til{\sgr}(\delta) $,
combining the preceding lower bound on $ \hk(\delta;x,x) $ with Proposition~\ref{prop:hk}\ref{cor:hk:hkrsup} and \ref{sgr:sup},
we now have, with probability $ \to_{\Lambda,N}1 $, $ \sg(\delta;x,x) \geq \frac1c - N^{-v}\Lambda \to \frac1c >0 $.
Inserting this back into~\eqref{eq:sgloc:} yields~\ref{prop:sg:loc}.
\end{itemize}
\end{proof}

We conclude this section by establishing the convergence of the microscopic semigroup $ \sg(t) $
to its macroscopic counterpart $ \Sg(t) $.
Recall from Assumption~\ref{assu:rt}\ref{assu:rt:limit} that $ \Rtlim $ and $ \Rt $ are \emph{coupled}.
The semigroups $ \Sg(t) $ and $ \sg(t) $ being constructed from $ \Rtlim $ and $ \Rt $,
the coupling in Assumption~\ref{assu:rt}\ref{assu:rt:limit} induces a coupling of $ \Sg(t) $ and $ \sg(t) $.

\begin{proposition}
\label{prop:sgtoSg}
Set $ \sg_N(t;x,\tilx) := N\sg(tN^2;Nx,N\tilx) $, and linearly interpolate in $ x $ and $ \tilx $
so that $ \sg_N(t;x,\tilx) $ defines a kernel on $ \limT $.
Given any $ T<\infty $, $ u>0 $, and $ f\in C(\limT) $, we have that
\begin{align*}
	\sup_{x\in\limT} \sup_{t\in[0,T]} \Big| \big( \Sg(t)f-\sg_N(t)f \big)(x) \Big| \longrightarrow_\text{P} 0.
\end{align*}
\end{proposition}
\begin{proof}
Set $ \hka_N(t;x,\tilx):= N\hka(N^2t;Nx,N\tilx) $,
$ \hk_N(t;x,\tilx):= N\hk(N^2t;Nx,N\tilx) $,
$ \hkr_N(t;x,\tilx):= N\hkr(N^2t;Nx,N\tilx) $,
and $ \sgr_{n,N}(t;x,\tilx) := N\sgr_n(tN^2;Nx,N\tilx) $,
and linearly interpolate these kernels in $ x $ and $ \tilx $.
Recall from~\eqref{eq:Sgker} and \eqref{eq:sgker} that $ \Sg(t) $ and $ \sg(t) $ are given in terms of $ \Sgr_n(t) $ and $ \HK(t) $, and $ \sgr_n(t) $ and $ \hka $, respectively.
Finally, recall that $ \hka(t)=\hk(t)+\hkr(t) $.
We may write
\begin{align*}
	\Big| \big( \Sg(t)f-\sg_N(t)f \big)(x) \Big|
	\leq
	&\Big| \int_{\limT} \big( \HK(t;x,\tilx) - \hk_N(t,x,\tilx) \big) f(\tilx) d\tilx \Big|
	+
	\norm{f}_{L^\infty(\limT)} \int_{\limT} |\hkr_N(t;x,\tilx)| d\tilx
\\
	&+
	\norm{f}_{L^\infty(\limT)}
	\sum_{n=1}^\infty \sup_{x\in\limT} \sup_{t\in[0,T]} \int_{\limT}|\Sgr_n(t;x,\tilx)-\sgr_{n,N}(t;x,\tilx)|d\tilx.
\end{align*}
Given that $ f\in C(\limT) $,
with the aid of Lemma~\ref{lem:hk}, it is standard to check that:
\begin{align*}
	\sup_{x\in\limT} \int_{\limT} \big( \HK(t;x,\tilx) - \hk_N(t;x,\tilx) \big) f(\tilx) d\tilx \longrightarrow 0,
	\qquad
	\textrm{as } N\to\infty.
\end{align*}
By Proposition~\ref{prop:hk}\ref{cor:hk:hkrsum}, we have
\begin{align*}
	\sup_{x\in\limT} \sup_{t\in[0,T]} \int_{\T} |\hkr_N(t;x,\tilx)| d\tilx \longrightarrow_\text{P} 0,
	\qquad
	\textrm{as } N\to\infty.
\end{align*}
Further, by Lemmas~\ref{lem:SgrIbd}\ref{lem:SgrI:int} and \ref{lem:sgrIbd}\ref{lem:sgrI:sum}, we have,
with probability $ \to_{\Lambda,N} 1 $,
\begin{align*}
	\sum_{n \geq 1}\sup_{x\in\limT} \sup_{t\in[0,T]} \int_{\limT}|\Sgr_n(t;x,\tilx)|d\tilx <\Lambda,
	\qquad
	\textrm{and}
	\qquad
	\sum_{n \geq 1}\sup_{x\in\T} \sup_{t\in[0,N^2T]} \sum_{\tilx\in\T}|\sgr_{n,N}(t;x,\tilx)| <\Lambda.	
\end{align*}
Given this, it suffices to check, for each fixed $ n \geq 1 $,
\begin{align*}
	\sup_{x\in\limT} \sup_{t\in[0,T]} \int_{\limT}\big| \Sgr_n(t;x,\tilx)-\sgr_{n,N}(t;x,\tilx) \big|d\tilx \longrightarrow_\text{P} 0,
	\qquad
	\textrm{as } N\to\infty.	
\end{align*}
Such a statement is straightforwardly checked (though tedious)
from the given expressions \eqref{eq:Sgr}--\eqref{eq:SgrI}, \eqref{eq:Ips}
and \eqref{eq:sgr}--\eqref{eq:sgrI}, \eqref{eq:ips} of $ \Sgr_n $ and $ \sgr_{n} $,
together with the aid of Lemmas~\ref{lem:hk}, \ref{lem:SgrIbd}, and \ref{lem:sgrIbd}.
We omit the details here.
\end{proof}

\section{Moment bounds and tightness}
\label{sect:mom}
Recall that $ \scZ_N(t,x)=Z(tN^{2},xN) $ denotes the scaled process in~\eqref{eq:Z}.
The goal of this section is to show the tightness of $ \{\scZ_N\}_N $.
For the case of homogeneous \ac{ASEP},
tightness is shown by establishing moment bounds on $ Z_N $
through iterating the microscopic equation (analogous to~\eqref{eq:Lang:int}); see \cite[Section~4]{bertini97} and also \cite[Section~3]{corwin18}.
Here we proceed under the same general strategy.
A major difference here is that the kernel $ \sg(t;x,x') $ (that governs the microscopic equation~\eqref{eq:Lang:int}) is itself random.
We hence proceed by conditioning.
For given $ u\in(0,1] $, $ v\in(0,\urt) $, $ \Lambda,T<\infty $, let
\begin{align}
	\label{eq:Omega}
	\Omega(u,v,\Lambda,T,N):=\{
		\text{properties in Proposition~\ref{prop:sg} hold and } \hold{\Rt}_{\urt,N} \leq \Lambda
	\}.
\end{align}
%
%
Recall $ \mg(s,x) $ from~\eqref{eq:mg}.
\begin{lemma}
\label{lem:BDG}
Fix $ k>1 $.
Write $ \Exrt[\,\Cdot \,]:=\Ex[ \, \Cdot \,  | \rt(x),x\in\T] $
for the conditional expectation quenching the inhomogeneity,
and write $ \normrt{\,\Cdot\,}{k} := (\Exrt[\,(\Cdot)^k\,])^{1/k} $ for the corresponding norm.
Given any deterministic $ f:\T\to\R $,
\begin{align*}
	\Normrt{ \int_{i}^{i'} \sum_{x\in\T} f(s,x)d\mg(s,x) }{k}^2
	\leq
	\frac{c(k)}{N} \sum_{i\leq j<i'} \sum_{x\in\T}\Big(\sup_{s\in[j,j+1]} f^2(s,x)\Big) \normrt{ Z(j,x) }{k}^2,
\end{align*}
for all $ i<i'\in\Z_{\geq 0} $.
\end{lemma}
\begin{proof}
The conditional expectation $ \Exrt[\,\Cdot \,]:=\Ex[ \, \Cdot \,  | \rt(x),x\in\T] $
amounts to fixing a realization of $ \{\rt(x)\}_{x\in\T} $ that satisfies Assumption~\ref{assu:rt}.
In fact, only Assumption~\ref{assu:rt}\ref{assu:rt:bdd} will be relevant toward the proof.
With this in mind, throughout this proof we view $ \rt(x) $ as \emph{deterministic} functions satisfying~Assumption~\ref{assu:rt}\ref{assu:rt:bdd}.

For fixed $ i\in\Z_{\geq 0} $, consider the discrete-time martingale $ \til{\mg}(i') := \sum_{j=i}^{i'-1} J(j) $, $ i'=i+1,i+2,\ldots $,
with increment $ J(j):= \int_{j}^{j+1} \sum_{x\in\T} f(s,x)d\mg(s,x) $.
Write $ \mathscr{F}(i'):=\sigma(J(i),\ldots,J(i')) $ for the canonical filtration.
Burkholder's inequality applied to $ \til{\mg} $ gives
\begin{align}
	\label{eq:burkholder}
	\normrt{ \til{\mg}(i') }{k}^2
	\leq
	c(k)
	\Normrt{
		\sum_{i\leq j<i'}  \Exrt\big[ J(j)^2 \big| \mathscr{F}(j) \big]
	}{k}.
\end{align}
We may compute
$ \Exrt[ J(j)^2 | \mathscr{F}(j) ] = \Exrt[ \int_{j}^{j+1} \sum_{x,x'} f(s,x)f(s,x') d\langle \mg(s,x),\mg(s,x')\rangle | \mathscr{F}(j) ] $.
The quadratic variation $ \langle \mg(s,y),\mg(s,y')\rangle $ is calculated in~\eqref{eq:qv}.
Under Assumption~\ref{assu:rt}\ref{assu:rt:bdd}, $ \rtt(x) $ is uniformly bounded,
and weak asymmetry scaling~\eqref{eq:was} gives $ (\tau-1)^{2}, (\tau^{-1}-1)^2 \leq \frac{1}{N} $.
Using these properties in~\eqref{eq:qv} gives
\begin{align}
	\label{eq:qv:bd}
	|\tfrac{d~}{dt}\langle \mg(t,x), \mg(t,x') \rangle |
	\leq
	\tfrac{c}{N}\ind_\set{x=x'}Z^2(t,x),
\end{align}
whereby
\begin{align}
	\label{eq:bdg:J}
	\Exrt[ J(j)^2 | \mathscr{F}(j) ] \leq \frac{c}{N} \sum_{x\in\T}\Exrt\Big[ \int_{j}^{j+1} f(s,x)^2 Z(s,x)^2 ds \Big| \mathscr{F}(j)\Big].
\end{align}
Fix $ x\in\T $.
Assumption~\ref{assu:rt}\ref{assu:rt:bdd} asserts that the Poisson clocks $ \PoiL(t,x) $ and $ \PoiR(t,x) $ that dictate jumps between $ x $ and $ x+1 $ have bounded rates.
Each jump changes $ Z(t,x) $ by a factor of $ \exp(\pm\frac{c}{\sqrt{N}}) $ (see~\eqref{eq:Z} and~\eqref{eq:was}).
This being the case, we have
\begin{align}
	\label{eq:locZbd>}
	Z(s,x) &\leq e^{\frac{X(j,x)}{\sqrt{N}}} Z(j,x),
	\qquad
	s\in[j,j+1),
\\
	\label{eq:locZbd<}
	Z(s,x) &\geq e^{-\frac{\til{X}(j,x)}{\sqrt{N}}} Z(j,x),
	\qquad
	s\in[j,j+1),
\end{align}
for some $ X(j,x),\til{X}(j,x) $ that are stochastically dominated by Poisson($ c $),
and are \emph{independent} of the sigma algebra $ \filZ(t) $ defined in~\eqref{eq:filZ}.
Now, use \eqref{eq:locZbd>} in \eqref{eq:bdg:J} to get
\begin{align*}
	\Exrt[ J(j)^2 | \mathscr{F}(j) ] \leq \frac{c}{N} \sum_{x\in\T}\Big(\sup_{s\in[j,j+1]} f(s,x)^2 \Big) Z(j,x)^2.
\end{align*}
Inserting this back into~\eqref{eq:burkholder} concludes the desired result.
\end{proof}

Recall from~\eqref{eq:nearst} that $ \uic>0 $ is the H\"{o}lder exponent of $ Z_\ic(\Cdot) $.
\begin{proposition}\label{prop:mom}
Fixing $ u\in(0,1) $,
$ v\in(0,\urt) $, $ k>1 $, and $ \Lambda,T<\infty $.
Let $ \Exrt[\,\Cdot \,] $ be as in Lemma~\ref{lem:BDG},
and further, write $ \ExO[\,\Cdot \,]:=\Exrt[ (\, \Cdot \, )\ind_{\Omega(u,v,\Lambda,T,N)}]=\Exrt[ \, \Cdot \, ]\ind_{\Omega(1,v,\Lambda,T)} $,
and let $ \normO{\,\Cdot\,}{k}:= \ExO[(\,\Cdot\,)^k]^{1/k} $ denote the corresponding norm.
There exists $ c=c(u,v,k,\Lambda,T) $ such that, for all $ x,x'\in\T $ and $ t,t'\in[0, N^2T] $,
\begin{subequations}
\begin{align}
	\label{eq:Zmom}
	\normO{ Z(t,x) }{k} &\leq c,
\\
	\label{eq:gZmom}	
	\normO{ Z(t,x)-Z(t,x') }{k} &\leq c \Big(\frac{\dist_\T(x,x')}{N}\Big)^{\frac{u}{2}\wedge\uic\wedge v},
\\
	\label{eq:tZmom}	
	\normO{ Z(t',x)-Z(t,x) }{k} &\leq c \Big(\frac{|t'-t|\vee 1}{N^2}\Big)^{\frac{u}{4}\wedge\frac{\uic}{2}\wedge\frac{v}{2}},
\end{align}
\end{subequations}
\end{proposition}
\begin{proof}
Fixing 
$ v\in(0,\urt) $, $ k>1 $, and $ \Lambda,T<\infty $,
throughout this proof we write $ c=c(v,k,T,\Lambda) $ to simplify notation.
As declared previously, the value of the constant may change from line to line.
Following the same convention as in the proof of Lemma~\ref{lem:BDG},
throughout this proof we view $ \rt(x) $ and $ \sg(t;x,\tilx) $ as \emph{deterministic} functions and,
(with $ \Omega(1,v,\Lambda,T) $ as in~\eqref{eq:Omega} being conditioned) assume
the properties in Proposition~\ref{prop:sg}\ref{prop:sg:sum}--\ref{prop:sg:loc} hold.

Let us begin by considering discrete time $ i\in\Z\cap[0,N^{2}T] $.
The starting point of the proof is the microscopic, mild equation~\eqref{eq:Lang:int}.
Recall that $ Z_\ic(x) $ is deterministic by assumption.
In~\eqref{eq:Lang:int}, set $ t=i $, take $ \normO{\,\Cdot\,}{k} $ on both sides, and square the result.
We have
\begin{align}
	\label{eq:Lang:iter}
	\normO{ Z(i,y) }{k}^2
	\leq
	2 \Big( \sum_{\tilx\in\T}\sg(i;x,\tilx)Z_\ic(\tilx) \Big)^2
	+
	2
	\NormO{ \int_0^i \sum_{\tilx\in\T}\sg(i-s;x,\tilx)d\mg(s,\tilx) }{k}^2.
\end{align}
To bound the last term in~\eqref{eq:Lang:iter},
apply Lemma~\ref{lem:BDG} with $ (i,i')\mapsto (0,i) $ and $ f(s,\tilx)=\sg(i-s;x,\tilx) $
(recall that $ \sg $ is  deterministic here),
and then use Proposition~\ref{prop:sg}\ref{prop:sg:loc} to bound $ \sup_{s\in[j,j+1]} \sg(i-s;x,\tilx)^2 $ by $ c\,\sg(i-j;x,\tilx)^2 $.
This gives
\begin{align}
	\label{eq:Lang:iter:}
	\normO{ Z(i,x) }{k}^2
	\leq
	2 \Big( \sum_{\tilx\in\T}\sg(i;x,\tilx)Z_\ic(\tilx) \Big)^2
	+
	\frac{c}{N} \sum_{j=0}^{i-1} \sum_{x\in\T} \sg(i-j;x,\tilx)^2 \normO{ Z(i,\tilx) }{k}^2.
\end{align}
Using the assumption $ Z_\ic(x) \leq c $ from~\eqref{eq:nearst} and the bound from Proposition \ref{prop:sg}\ref{prop:sg:sum},
we have $ \sum_{\tilx\in\T}\sg(i;x,\tilx)Z_\ic(\tilx) \leq c $,
and using the bound from Proposition \ref{prop:sg}\ref{prop:sg:sup},
we write $ \sg(i-j;x,\tilx)^2 \leq \sg(i-j;x,\tilx) c\,(i-j)^{-1/2} $.
Inserting these bounds into~\eqref{eq:Lang:iter:}, we arrive at
\begin{align}
	\label{eq:Lang:iter:1}
	\normO{ Z(i,x) }{k}^2
	\leq
	c
	+
	c \sum_{j=0}^{i-1}  \frac{N^{-2}}{\sqrt{N^{-2}(i-j)}} \sum_{x\in\T}\sg(i-j;x,\tilx) \normO{ Z(i,\tilx) }{k}^2,
\end{align}
Iterating~\eqref{eq:Lang:iter:1} gives
\begin{align}
	\label{eq:Lang:iter:2}
	\normO{ Z(i,x) }{k}^2
	\leq
	c
	+
	\sum_{n=1}^\infty c^n \sum_{\vec{\ell}\in \sigma_n(i)}  \prod_{j=0}^n \frac{N^{-2}}{\sqrt{N^{-2}\ell_j}}
	\sum_{x_1,\ldots,x_n\in\T} \prod_{k=1}^n \sg(\ell_j;x_{j-1},x_{j}).
\end{align}
Here, we adopt the convention that $ x_0:=x $,
and  $ \sigma_n(i) :=\{ (\ell_0,\ldots,\ell_n) \in \Z_{>0}^{n+1} : \ell_0+\ldots+\ell_n=i \} $.
In~\eqref{eq:Lang:iter:2},
we sum over $ x_{n},\ldots,x_1 $ in order, and use the bound from Proposition~\ref{prop:sg}\ref{prop:sg:sum} at each step to bound the result by $ c $.
Then, approximating the sum over $ \vec{\ell}\in \sigma_n(t) $ by an integral, we have
\begin{align}
	\label{eq:Lang:iter:3}
	\normO{ Z(i,x) }{k}^2
	\leq
	c
	+
	\sum_{n=1}^\infty c^n \int_{\Sigma_n(iN^{-2})} \prod_{j=0}^n s_i^{-\frac12} \cdot d^n\vec{s}.
\end{align}
We now apply the Dirichlet integral formula~\eqref{eq:dirichlet} with $ v_0=\ldots=v_n=\frac12 $.
Given that $ i \leq TN^{2} $, upon summing the result over $ n=1,2,\ldots $, we obtain
\begin{align}
	\tag{\ref*{eq:Zmom}'}
	\label{eq:Zmom:}
	\normO{ Z(i,x) }{k}  &\leq c.
\end{align}
This is exactly the first desired bound~\eqref{eq:Zmom} for $ t=i\in\Z_{>0} $,
and hence the label~\eqref{eq:Zmom:}.
We will have similar labels for \eqref{eq:gZmom}--\eqref{eq:tZmom}.

We now turn to the gradient moment estimates~\eqref{eq:gZmom}.
Set
\begin{align}
	\label{eq:mom:I}
	I(x)&:=\sum_{\tilx\in\T}\sg(i;x,\tilx)Z_\ic(\tilx),
\\
	\label{eq:mom:J}
	J(x,x')&:= \frac{1}{N} \sum_{j=0}^{i-1} \sum_{\tilx\in\T} |\sg(i-j;x,\tilx)-\sg(i-j;x',\tilx)|^2 \normO{ Z(i,\tilx) }{k}^2.
\end{align}
Following the same procedure leading to~\eqref{eq:Lang:iter:},
but starting with $ Z(i,x)-Z(i,x') $ instead of $ Z(i,x) $, here we have
\begin{align}
	\label{eq:Lang:iter::}
	\normO{ Z(i,x) -  Z(i,x') }{k}^2
	\leq
	2 \big( I(x)-I(x') \big)^2
	+
	c \, J(x,x').
\end{align}
To bound the term $ J(x,x') $, in~\eqref{eq:mom:J}, use
\begin{align*}
	|\sg(i-j;x,\tilx)-\sg(i-j;x',\tilx)|^2
	\leq
	\Big( \sup_{\tilx}|\sg(i-j;x,\tilx)-\sg(i-j;x',\tilx)| \Big) \Big( \sg(i-j;x,\tilx)+\sg(i-j;x',\tilx) \Big).
\end{align*}
Then, sum over $ \tilx\in\T $,
using the bound \eqref{eq:Zmom:} on $ \normO{ Z(i,\tilx) }{k}^2 $
and the bounds from Proposition~\ref{prop:sg}\ref{prop:sg:sum} and \ref{prop:sg:gsup} on $ \sg $.
With $ i\leq N^2T $, we have
\begin{align}
	\label{eq:mom:J:}
	J(x,x') \leq \frac{c}{N} \sum_{j=0}^{i-1} \frac{\dist_\T(x,x')^u}{(i-j+1)^{(u+1)/2}} \leq c\Big(\frac{\dist_\T(x,x')}{N}\Big)^u.
\end{align}
We now proceed to bound $ I(x)-I(x') $.
Recall that $ \sg(t)=\hk(t)+\sgr(t) $.
Decompose $ I(x)=I_\hk(x)+I_\sgr(x) $ into the corresponding contributions of $ \hk(t) $ and $ \sg(t) $:
$ I_\hk(x):=\sum_{\tilx\in\T}\hk(i;x,\tilx)Z_\ic(\tilx) $ and $ I_\sgr(x):=\sum_{\tilx\in\T}\sgr(i;x,\tilx)Z_\ic(\tilx) $.
For $ I_\hk $, using translation invariance of $ \hk $ (i.e., $ \hk(t;x,\tilx)=\hk(t;x+i,\tilx+i) $),
we have $ I_\hk(x)-I_\hk(x')=\sum_{\tilx\in\T} \hk(t;x,\tilx)(Z_\ic(\tilx)-Z_\ic(\tilx+(x'-x))) $.
Given this expression, together with the H\"{o}lder continuity of $ Z_\ic(\Cdot) $ from our assumption \eqref{eq:nearst},
we have
\begin{align}
	\label{eq:mom:Ihk}
	\big| I_\hk(x)-I_\hk(x') \big| \leq \big(\tfrac{\dist_\T(x,x')}{N}\big)^{\uic} c.
\end{align}
As for $ I_\sgr $, using the bound from Proposition~\ref{prop:sg}\ref{prop:sgr:gsum} for $ u=v $ and the boundedness of $ Z_\ic(x) $ gives
\begin{align}
	\label{eq:mom:Isgr}
	\big| I_\sgr(x)-I_\sgr(x') \big| \leq \big(\tfrac{\dist_\T(x,x')}{N}\big)^v c.
\end{align}
Combining~\eqref{eq:mom:J:}--\eqref{eq:mom:Isgr} with~\eqref{eq:Lang:iter::} yields
\begin{align}
	\tag{\ref*{eq:gZmom}'}
	\label{eq:gZmom:}
	\normO{Z(i,x)-Z(i,x') }{k}
	\leq
	\Big(
			\big(\tfrac{\dist_\T(x,x')}{N}\big)^u
			+
			(\tfrac{\dist_\T(x,x')}{N})^{2\uic}
			+
			(\tfrac{\dist_\T(x,x')}{N})^{2v}
	\Big)^{1/2}c
	\leq
	\big(\tfrac{\dist_\T(x,x')}{N}\big)^{\frac{u}{2}\wedge \uic\wedge v} c.
\end{align}

Finally we turn to the gradient moment estimate~\eqref{eq:tZmom}.
Fix $ i<i'\in\Z\cap[0,N^2T] $, $ x\in\T $, and set
\begin{align*}
	\til{I}(i,i',x):=\sum_{\tilx\in\T}\sg(i'-i;x,\tilx)Z(i,\tilx) - Z(i,x),
	\qquad
	\til{J}(i,i',x) := \frac1N\sum_{j=i}^{i'-1} \sum_{\tilx\in\T} \sg(i'-j;x,\tilx)^2 \normO{ Z(i,\tilx) }{k}.
\end{align*}
To alleviate heavy notation, hereafter we omit dependence on $ (i,i',x) $
and write $ \til{I} $ and $ \til{J} $ in places of $ \til{I}(i,i',x) $ and $ \til{J}(i,i',x) $.
Following the same procedure leading to~\eqref{eq:Lang:iter:},
starting from $ t=i $ instead of $ t=0 $, here we have
\begin{align}
	\label{eq:Lang:iter:::}
	\normO{ Z(i',x)-Z(i,x) }{k}^2
	\leq
	2 \normO{ \til{I} }{k}^2
	+
	c \til{J}.
\end{align}
Using the bound~\eqref{eq:Zmom:} on $ \normO{ Z(i,\tilx) }{k} $ and the bounds from
Proposition~\ref{prop:sg}\ref{prop:sg:sum} and \ref{prop:sg:sup} for $ u=1 $ on $ \sg $,
we have
\begin{align*}
	\til{J} \leq \frac{c}{N} \sum_{j=i}^{i'} \frac{1}{\sqrt{i'-j+1}} \leq \Big(\frac{i'-i}{N^2}\Big)^{\frac12} c.
\end{align*}
As for $ \til{I} $, decompose it into $ \til{I}=\til{I}_\hk+\til{I}_\sgr $,
where
\begin{align*}
	\til{I}_\hk
	&:=\sum_{\tilx\in\T}\hk(i'-i;x,\tilx)Z(i,\tilx) - Z(i,x)
	=
	\sum_{\tilx\in\T}\hk(i'-i;x,\tilx)\big(Z(i,\tilx) - Z(i,x)\big),
\\
	\til{I}_\sgr
	&:=\sum_{\tilx\in\T}\sgr(i'-i;x,\tilx)Z(i,\tilx).
\end{align*}
Taking $ \normO{\,\Cdot\,}{k} $ of $ \til{I}_\hk $, with the aid of~\eqref{eq:gZmom:},
we have $ \normO{\til{I}_\hk}{k} \leq c \sum_{\tilx\in\T}\hk(i'-i;x,\tilx)(\dist_\T(x,\tilx)/N)^{\frac12\wedge\uic} $.
For $ \hk $ it is straightforward to show that
$ \sum_{\tilx\in\T}\hk(i'-i;x,\tilx)\dist_\T(x,\tilx)^u \leq c(u) (i'-i)^{u/2} $, so
$
	\normO{\til{I}_\hk}{k} \leq (\frac{i'-i}{N^2})^{\frac12\wedge\uic}c.
$
As for $ \til{I}_\sgr $, taking $ \normO{\,\Cdot\,}{k} $ using~\eqref{eq:Zmom:} and
the bound from Proposition~\ref{prop:sg}\ref{prop:sgr:sum} gives
$
	\normO{\til{I}_\hk}{k} \leq (\tfrac{i'-i}{N^2})^vc. 	
$
Inserting the preceding bounds on $ \til{J} $, $ \til{I}_\hk $, and $ \til{I}_\sgr $ into~\eqref{eq:Lang:iter:::},
we obtain
\begin{align}
	\tag{\ref{eq:tZmom}'}
	\label{eq:tZmom:}
	\normO{ Z(i,y)-Z(i',y) }{k}
	\leq
	\Big( \big(\tfrac{i'-i}{N^2}\big)^{\frac12} + \big(\tfrac{i'-i}{N^2}\big)^{\frac12\wedge\uic} c+\big(\tfrac{i'-i}{N^2}\big)^{v} \Big)^{1/2}c
	\leq
	\big(\tfrac{i'-i}{N}\big)^{\frac{u}{4}\wedge\frac{\uic}{2}\wedge\frac{v}{2}} c.
\end{align}

So far we have obtained the relevant bounds~\eqref{eq:Zmom:}--\eqref{eq:tZmom:} for integer time.
To go from integer to continuum, we consider generic $ \lfloor t\rfloor \leq t\in[0,N^{2}T] $,
and estimate $ \normO{ Z(t,x)-Z(\lfloor t\rfloor, x)}{k} $.
To this end, recall we have the local (in time) bounds~\eqref{eq:locZbd>}--\eqref{eq:locZbd<} on the growth of $ Z(s,y) $,
where $ X(j,x),\til{X}(j,x) $ that are stochastically dominated by Poisson($ c $), and are \emph{independent} of $ \filZ(t) $ (defined in~\eqref{eq:filZ}).
In~\eqref{eq:locZbd>}--\eqref{eq:locZbd<},
subtract $ Z(j,x) $ from both sides,
and take $ \normO{ \,\Cdot\,}{k} $ on both sides to get
\begin{align}
	\notag
	&\NormO{ \sup_{t\in[j,j+1]}|Z(t,x)-Z(j, x)|}{k}
	\leq
	\NormO{ (e^{\frac{X(j,x)}{\sqrt{N}}}-1)Z(j,x) }{k}
	+
	\NormO{ (1-e^{-1\frac{\til{X}(j,x)}{\sqrt{N}}})Z(j, x) }{k}
\\
	\label{eq:locZbd}
	&\quad\quad\quad
	=
	\NormO{ (e^{\frac{X(j,x)}{\sqrt{N}}}-1) }{k} \, \normO{ Z(j, x) }{k}
	+
	\NormO{ (1-e^{-1\frac{\til{X}(j,x)}{\sqrt{N}}})}{k} \, \normO{ Z(j, x) }{k}
	\leq
	\tfrac{1}{\sqrt{N}}c.
\end{align}
Since $ (\frac{\dist_\T(x,x')}{N})^{\frac{u}{2}\wedge\uic\wedge v}, (\frac{|t-t'|\vee 1}{N^2})^{\frac{u}{4}\wedge\frac{\uic}{2}\wedge\frac{v}{2}} \geq \frac{1}{\sqrt{N}} $ for all $ x\neq x' $ and $ t,t'\geq 0 $,
we may use~\eqref{eq:locZbd} to approximate $ Z(\lfloor t\rfloor,x) $ with $ Z(t,x) $,
and hence infer \eqref{eq:Zmom}--\eqref{eq:tZmom} from \eqref{eq:Zmom:}--\eqref{eq:tZmom:}.
\end{proof}

Recall that $ D([0,T],C(\limT)) $ denotes the space of right-continuous-with-left-limits functions $ [0,T]\to C(\R) $,
equipped with Skorohod's $ J_1 $-topology.

\begin{corollary}
\label{cor:tight}
For any given $ T<\infty $, $ \{\scZ_N\}_{N} $ is tight in the space of $ D([0,T],C(\limT)) $,
and its limits concentrate in $ C([0,T],C(\limT)) $.
\end{corollary}
\begin{proof}
First, to avoid the jumps (in $ t $) of $ \scZ_N(t,x) $,
consider the process $ \til{\scZ}_N(t,x) := \scZ(t,x) $, for $ t\in \frac{1}{N^2}\Z_{\geq 0} $,
and linearly interpolate in $ t\in[0,\infty) $.
For fixed $ v,\Lambda,T $ as in Proposition~\ref{prop:mom},
the moment bounds obtained in Proposition~\ref{prop:mom}, together with the Kolmogorov continuity theorem,
implies that $ \{\til{\scZ}_N \ind_{\Omega(1,v,\Lambda,T)}\}_{N} $ is tight in $ C([0,T]\times\limT)=C([0,T],C(\limT)) $.
Further, Proposition~\ref{prop:sg} asserts that $ \Pr[\Omega(1,v,\Lambda,T)] \to 1 $ under the iterative limit $ (\lim_{\Lambda\to\infty} \lim_{N\to\infty}\Cdot) $,
so $ \{\til{\scZ}_N \}_{N} $ is tight in $ C([0,T],C(\limT)) $.

To relate $ \scZ_N $ to $ \til{\scZ}_N $,
we proceed to bound the difference $ \til{\scZ}_N - \scZ_N $.
Fix $ u\in(0,1) $, $ v\in(0,\urt) $ and set $ I_j:=[\frac{j}{N^2},\frac{j+1}{N^2}] $.
From~\eqref{eq:locZbd}, we have that
\begin{align}
	\notag
	\ExO&\big[ \ \norm{ \til{\scZ}_N - \scZ_N }_{L^\infty(I_j\times\limT)}^k \ \big]
\\
	\label{eq:markov:}	
	&:=\Ex\Big[ \ \norm{ \til{\scZ}_N - \scZ_N }_{L^\infty(I_j\times\limT)}^k \ind_{\Omega(u,v,\Lambda,T,N)} \ \Big| \rt(x),x\in\T \Big]
	\leq
	c(u,v,\Lambda,k,T) N^{-k/2},
	\qquad
	j=0,1,\ldots, \lceil TN^2 \rceil.
\end{align}
The r.h.s.\ of~\eqref{eq:markov:} is deterministic (i.e., not depending on $ \rt $).
This being the case, take $ \Ex[\,\Cdot\,] $ in~\eqref{eq:markov:}, and apply Markov inequality
$ \Pr[|X|^k>\e] \leq \e^{-k}\Ex[|X|^k] $ with $ X=\norm{ \til{\scZ}_N - \scZ_N }_{L^\infty(I_j\times\limT)} \ind_{\Omega(u,v,\Lambda,T,N)} $.
We obtain
\begin{align*}
	\Pr\Big[ \ \norm{ \til{\scZ}_N - \scZ_N }_{L^\infty(I_j\times\limT)} \ind_{\Omega(u,v,\Lambda,T,N)} > \varepsilon \Big]
	\leq
	c(u,v,\Lambda,k,T) \e^{-k} N^{-k/2},
	\qquad
	j=0,1,\ldots, \lceil TN^2 \rceil.
\end{align*}
Setting $ k=5 $ and take union bounds over $ j=0,1,\ldots, \lceil TN^2 \rceil $ yields
\begin{align}
	\label{eqcor:tight}
	\Pr\big[ \norm{ \til{\scZ}_N-\scZ_N }_{L^\infty([0,T]\times\limT)}\ind_{\Omega(u,v,\Lambda,T,N)} > \varepsilon \big]
	\leq
	c(u,v,\Lambda,T,\e) N^{2-5/2}.
\end{align}
Further, Proposition~\ref{prop:sg} asserts that $ \Pr[\Omega(u,v,\Lambda,T,N)] \to 1 $ under the iterative limit
$ (\lim_{\Lambda\to\infty} \lim_{N\to\infty}\Cdot) $.
Hence, passing~~\eqref{eqcor:tight} to $ N\to\infty $ along a suitable sequence $ \Lambda=\Lambda_n\to\infty $ gives
\begin{align*}
	\lim_{N\to\infty} \Pr\big[ \norm{ \til{\scZ}_N-\scZ_N }_{L^\infty([0,T]\times\limT)} > \varepsilon \big] = 0.
\end{align*}
From this, we conclude that $ \scZ_N $ and $ \til{\scZ}_N $ must have the same limit points in $ D([0,T],C(\limT)) $.
Knowing that $ \{\til{\scZ}_N \}_{N} $ is tight in $ C([0,T],C(\limT)) $,
we thus conclude the desired result.
\end{proof}

\section{Proof of Theorem~\ref{thm:main}}
\label{sect:pfmain}
Given Corollary~\ref{cor:tight}, to prove Theorem~\ref{thm:main},
it suffices to identify limit points of $ \{\scZ_N\}_N $.
We achieve this via a martingale problem.

\subsection{Martingale problem}
\label{sect:mgpb}
Recall that, even though $ \Ham $ and its semigroup $ \Sg(t):= e^{t\Ham} $ are possibly random, they are independent of the driving noise $ \xi $.
This being the case, conditioning on a generic realization of $ \Rtlim $,
throughout this subsection, we assume $ \Sg(t) $ and $ \Ham $ are \emph{deterministic},
(constructed from a deterministic $ \Rtlim\in C^{\urt}[0,1] $).

It is shown in~\cite[Section~2]{fukushima77} that, for bounded $ \Rtlim $,
the self-adjoint operator $ \Ham = \frac12\partial_{xx}+\Rtlim'(x) $ has discrete spectrum.
More explicitly, $ \Ham \eigf_n = \eigv_n \eigf_n $, $ n=1,2,\ldots $,
with $ \eigf_n \in D(\Ham) \subset H^1(\limT) $ and $ \eigv_1\geq \eigv_2 \geq \cdots \to-\infty $,
and with $ \{\eigf_n\}_{n=1}^\infty $ forming a Hilbert basis (i.e., dense orthonormal set) of $ L^2(\limT) $.
Let $ \eigsp := \{ \sum_{i=1}^m \alpha_i \eigf_i: m\in\Z_{>0},\alpha_1,\ldots,\alpha_m\in\R \} $
denote the linear span of eigenfunctions.
Recall that $ \langle f, g\rangle := \int_{\limT} f(x) g(x) dx $ denotes the inner product on $ L^2(\limT) $.

We say that a $ C([0,\infty)\times\limT) $-valued process $ \limZ $ solves the \textbf{martingale problem} corresponding to~\eqref{eq:spde} if,
for any $ f\in\eigsp $,
\begin{align}
	\label{eq:linmg}
	\Mg(t;f)
	&:=
	\langle f, \limZ(s) \rangle \Big|_{s=0}^{s=t}
	-
	\int_{0}^t \langle \Ham f, \limZ(s) \rangle ds ,
\\	
	\label{eq:qdmg}
	\Mgg(t;f)
	&:=
	(\Mg_f(t))^2
	-
	\int_{0}^t \langle f^2, \limZ^2(s) \rangle ds
\end{align}
are local martingales in $ t $.

\begin{proposition}
\label{prop:mgpb}
A $ C([0,\infty)\times\limT) $-valued process $ \limZ $ that solves
the aforementioned martingale problem is a mild solution~\eqref{eq:spde} of the \ac{SPDE}~\eqref{eq:spde}.
\end{proposition}
%
%
\begin{proof}
Fix $ \limZ\in C([0,\infty)\times\limT) $ that solves the martingale problem.
The first step is to show that $ \limZ $ is a weak solution.
That is, extending the probability space if necessary,
there exists a white noise measure $ \xi(t,x)dtdx $ such that,
for any given $ f\in\eigsp $,
\begin{align}
	\label{eq:weak}
	\Mg(t;f)
	=
	\langle f, \limZ(s) \rangle \big|_{s=0}^{s=t}
	-
	\int_{0}^t \langle \Ham f, \limZ(s) \rangle ds
	=
	\int_0^t\int_{\limT} f(x) \limZ(s,x) \xi(s,x)dx.
\end{align}
With $ \eigsp $ being dense in $ L^2(\limT) $,
the statement is proven by the same argument of \cite[Proposition~4.11]{bertini97}.
We do not repeat it here.

Next, for given $ n \geq 1 $,
consider the process $ F(t):=e^{-\eigv_n t} \langle \eigf_n, \limZ(t) \rangle  $.
Using It\^{o} calculus, with the aid of~\eqref{eq:weak} (for $ f=\eigf_n $),
we have
\begin{align*}
	F(t)-F(0)
	=
	\int_0^t \big( -\eigv_n F(s) + e^{-\eigv_n s} \langle \Ham\eigf_n, \limZ(s) \rangle \big) ds
	+
	\int_0^t \int_{\limT} e^{-\eigv_n s} \eigf_n(x) \limZ(s,x) \xi(s,x)dxds.
\end{align*}
With $ \Ham\eigf_n=\eigv_n\eigf_n $, the first term on the r.h.s.\ is zero.
This being the case, multiplying both sides by $ e^{\eigv_n t} $ gives
\begin{align*}
	\langle \eigf_n, \limZ(t) \rangle
	-
	\langle e^{t\eigv_n}\eigf_n, \limZ(0) \rangle	
	=
	\int_0^t \int_{\limT} e^{\eigv_n (t-s)} \eigf_n(\tilx) \limZ(s,\tilx) \xi(s,\tilx)d\tilx ds.
\end{align*}
Further, write $ e^{t\eigv_n}\eigf_n= \Sg(t)\eigf_n $
and $ e^{\eigv_n (t-s)} \eigf_n(\tilx) = \int_{\limT} \Sg(t-s;x,\tilx) \eigf_n(x) dx $,
and use the fact that $ \Sg(t-s;x,\tilx)=\Sg(t-s;\tilx,x) $,
we now have
\begin{align}
	\label{eq:mild1}
	\langle f, \limZ(t) \rangle
	-
	\langle \Sg(t)f, \limZ(0) \rangle	
	=
	\Big\langle f, \int_0^t \int_{\limT} \Sg(t-s;\Cdot,\tilx) \limZ(s,\tilx) \xi(s,\tilx) d\tilx ds \Big\rangle,
	\qquad
	f=\eigf_1,\eigf_2,\ldots.
\end{align}
Equation~\eqref{eq:mild1} being linear in $ f $ readily generalizes to all $ f\in\eigsp $.
With $ \eigsp $ being dense in $ L^2(\limT) $ and hence in $ C(\limT) $,
we conclude that $ \limZ $ satisfies~\eqref{eq:spde:mild}.
\end{proof}

For convenience of subsequent analysis,
let us rewrite the martingale problem \eqref{eq:linmg}--\eqref{eq:qdmg} in a slightly different but equivalent form:
for all $ n,n' \geq 1 $,
\begin{align}
	\tag{\ref*{eq:linmg}'}
	\label{eq:linmg:}
	\Mg_n(t)
	&:=
	\Mg(t;\eigf_n)
	=
	\langle \eigf_n, \limZ(s) \rangle \big|_{s=0}^{s=t}
	-
	\eigv_n\int_{0}^t \langle \eigf_n, \limZ(s) \rangle ds,
\\	
	\tag{\ref*{eq:qdmg}'}
	\label{eq:quadmg:}
	\Mgg_{n,n'}(t)
	&:=
	\Mg(t;\eigf_n)\Mg(t;\eigf_{n'})
	-
	\int_{0}^t \langle \eigf_n\eigf_{n'}, \limZ^2(s) \rangle ds
\end{align}
are local-martingales in $ t $.

As stated previously, to prove Theorem~\ref{thm:main},
it now suffices to identify limit points of $ \{\scZ_N\}_N $.
This being the case, after passing to a subsequence,
hereafter we assume $ \scZ_N \Rightarrow \limZ $,
for some $ C([0,\infty),C(\T)) $-valued process $ \limZ $.
By Skorokhod's representation theorem, extending the probability space if necessary,
we further assume $ \limZ $ and $ \scZ_N $ inhabit the same probability space, with
\begin{align}
	\label{eq:Zcnvg}
	\norm{\scZ_N - \limZ }_{L^\infty([0,T]\times\limT)} \longrightarrow_\text{P} 0,
\end{align}
for each given $ T<\infty $.
Our goal is to show that $ \limZ $ solves the martingale problem~\eqref{eq:linmg:}--\eqref{eq:quadmg:}.
We further refer to \eqref{eq:linmg:} and \eqref{eq:quadmg:} as the linear and quadratic martingale problems, respectively.

\subsection{Linear martingale problem}
\label{sect:linmg}
Here we show that $ \limZ $ solves the linear martingale problem~\eqref{eq:linmg:}.
Let
\begin{align*}
	\big\langle f, g \big\rangle_N := \frac{1}{N} \sum_{x\in\T} f(\tfrac{x}{N}) g(\tfrac{x}{N})
\end{align*}
denote the discrete analog of $ \langle f,g\rangle $,
$ \Delta_N f (x) := N^2 (f(x+\frac{1}N)+f(x-\frac1N)-f(2x)) $ denote the scaled discrete Laplacian,
and $ \ham_N  := \frac12 \Delta_N + N^2 \nu \rt(Nx) $  denote the scaled operator.
Multiply both sides of~\eqref{eq:Lang} by $ \eigf_n(Nx) $,
integrate over $ t\in[0,N^2\bar{t}] $ and sum over $ x\in\T $.
We have that
\begin{align}
	\label{eq:mgn}
	\mg_n(N^2t)
	:=
	\int_0^{N^2t} \frac1N \sum_{x\in\T} \eigf_n(Nx) d\mg(t,x)
	=
	\langle \eigf_n, \scZ_N(s) \rangle_N \Big|_{s=0}^{s=t}
	-
	\int_0^t \langle \ham_N\eigf_n,  \scZ_N(s) \rangle_N ds
\end{align}
is a martingale.

Indeed, the r.h.s.\ of~\eqref{eq:mgn} resemble the r.h.s.\ of~\eqref{eq:linmg:},
and one would hope to show convergence of the former to the latter in order to establish $ \Mg_n(t) $ being a local martingale.
For the case of homogeneous \ac{ASEP}, we have $ \frac12\Delta_N $ in place of $ \ham_N $,
and the eigenfunctions $ \eigf_n $ are $ C^2 $.
In this case, using Taylor expansion it is straightforward to show that
$ \int_0^t \langle \frac12\Delta_N\eigf_n,  \scZ_N(s) \rangle_N ds $ converges to
its continuum counterpart $ \int_0^t\int_{\limT} \langle \frac12 \eigf_n'', \limZ(s) \rangle ds $.
Here, on the other hand, we only have $ \eigf_n \in H^1(\limT) $,
and $ \rt(x) $ and $ \scZ_N(t,x) $ lack differentiability in $ x $.
Given the situation,
a direct proof of $ \int_0^t \langle \ham_N\eigf_n,  \scZ_N(s) \rangle_N ds $ converging to its continuum counterpart seems challenging.

To circumvent the aforementioned issue, we route through the integrated (i.e., mild) equation~\eqref{eq:Lang:int:}.
For a given $ t \geq 0 $ and $ k\in\Z_{>0} $, put $ t_i:= \frac{i}{k}t $, set $ (t_*,t)=(N^2t_{i-1},N^2t_i) $ in~\eqref{eq:Lang:int:},
and subtract $ Z(N^2t_{i-1},x) $ from both sides.
This gives
\begin{align*}
	Z(s,x)\big|_{s=N^2t_{i-1}}^{s=N^2t_i}
	=
	\Big(\big(\sg(N^2\tfrac{t}{k}) -\Id\big) Z(N^2t_{i-1}) \Big)(x) + \sum_{\tilx\in\T}\int_{N^2t_{i-1}}^{N^2t_i} \sg(N^2t_i-s;x,\tilx) d\mg(s,\tilx),
\end{align*}
where `$ \Id $' denotes the identity operator.
Multiply both sides by $ \eigf_n(Nx) $, and sum over $ x\in\T $ and $ i=1,\ldots,k $.
After appropriate scaling, we obtain
\begin{align}
	\label{eq:linmg:pf}
	\big\langle \eigf_n, \scZ_N(s)\big\rangle_N \big|_{s=0}^{s=t}
	-
	\linmgD_{k,N}(t)
	-
	\sum_{i=1}^k\sum_{x,\tilx\in\T}\int_{N^2t_{i-1}}^{N^2t_i} \frac1N \sum_{x\in\T} \eigf_n(Nx)\sg(N^2t_i-s;x,\tilx) d\mg(s,\tilx)
	=
	0,
\end{align}
where, with $ (\sg_N(t) f)(x) := \frac{1}{N}\sum_{\tilx\in\frac1N\T} N\sg(N^2t;Nx,N\tilx) f(x) $ denoting the scaled semigroup,
\begin{align}
	\label{eq:linmgD}
	\linmgD_{k,N}(t) := \sum_{i=1}^k \Big\langle \eigf_n, (\sg_N(\tfrac{t}{k})-\Id) \scZ_N(t_{i-1}) \Big\rangle_N.
\end{align}
Further adding and subtracting $ \mg_n(t) $ on both sides of \eqref{eq:linmg:pf} gives
\begin{align}
	\label{eq:linmg:pf:}
	&\langle \eigf_n, \scZ_N(s)\rangle_N \big|_{s=0}^{s=t}
	-
	\linmgD_{k,N}(t)
	-
	\linmgR_{k,N}(t)
	=
	\mg_n(t),
\\
	\label{eq:lingmgR}
	&\linmgR_{k,N}(t)
	:=
	\sum_{i=1}^k
	\int_{N^2t_{i-1}}^{N^2t_i} \frac1N\sum_{\tilx\in\T}\Big(  \sum_{x\in\T} \eigf_n(Nx)\sg(N^2t_i-s;x,\tilx) -\eigf_n(N\tilx) \Big) d\mg(s,\tilx).
\end{align}

In the following we will invoke convergence to zero in probability under \emph{iterated} limits.
For random variables $ X_{k,N} $ indexed by $ k,N $,
we write
\begin{align*}
	\lim_{k\to\infty} \lim_{N\to\infty} X_{k,N} \stackrel{\text{P}}{=} 0
\end{align*}
if $ \limsup\limits_{k\to\infty} \limsup\limits_{N\to\infty}\Pr[ |X_{k,N}| > \e ] =0 $, for each $ \e >0 $.
Given~\eqref{eq:linmg:pf:}, we proceed to show
\begin{lemma}
\label{lem:linmg}
For any given $ T<\infty $,
\begin{enumerate}[label=(\alph*),leftmargin=7ex]
\item \label{lem:linmg1} \
	$
		\displaystyle
		\lim_{N\to\infty} \Big( \sup_{t\in[0,T]} \big| \langle \eigf_n, \scZ_N(t)\rangle_N - \langle \eigf_n, \limZ(t)\rangle \big| \Big) \stackrel{\text{P}}{=} 0,
	$
\item \label{lem:linmg2} \
	$
		\displaystyle
		\lim_{k\to\infty}\lim_{N\to\infty}
		\sup_{t\in[0,T]} \big| \linmgD_{k,N}(t) - \eigv_n \int_0^t \langle \eigf_n, \limZ(s)\rangle ds \big| \stackrel{\text{P}}{=} 0,
	$
\item \label{lem:linmg3} \
	$
		\displaystyle
		\lim_{k\to\infty}\lim_{N\to\infty}
		\sup_{t\in[0,T]} \big| \linmgR_{k,N}(t) \big| \stackrel{\text{P}}{=} 0.
	$
\end{enumerate}
\end{lemma}
\begin{proof}
\ref{lem:linmg1} Given~\eqref{eq:Zcnvg} and $ \eigf_n \in H^1(\limT) \subset C(\limT) $, this follows straightforwardly.

\ref{lem:linmg2} Given~\eqref{eq:Zcnvg} and Proposition~\ref{prop:sgtoSg}, we have,
for each $ s,\delta\in[0,\infty) $,
\begin{align}
	\label{eqlem:lingmg2}
	\lim_{N\to\infty} \norm{ \big(\sg_N(\delta)-\Id) \scZ_N(s)\big)- \big(\Sg(\delta)-\Id\big) \limZ(s) }_{L^\infty(\limT)}
	\stackrel{\text{P}}{=} 0.
\end{align}
Using~\eqref{eqlem:lingmg2} for $ s=t_{j-1} $ and $ \delta=\frac{t}{k} $, and plugging it into~\eqref{eq:linmgD},
together with $ \eigf_n \in H^1(\limT) \subset L^1(\limT) $, we have
\begin{align*}
	&\lim_{N\to\infty} \sup_{t\in[0,T]} |\linmgD_{k,N}(t) - \LinmgD_k(t,k)| \stackrel{\text{P}}{=} 0,
\\
	&\LinmgD_k(t) := \sum_{i=1}^k \Big\langle \eigf_n, (\Sg(\tfrac{t}{k})-\Id) \limZ(t_{i-1}) \Big\rangle
	=
	\sum_{i=1}^k (e^{\frac{t}{k}\eigv_n}-1) \big\langle \eigf_n, \limZ(t_{i-1}) \big\rangle.
\end{align*}
Further taking the $ k\to\infty $ limit using the continuity of $ \limZ(t) $ gives
\begin{align*}
	\lim_{k\to\infty} \sup_{t\in[0,T]} \Big|\LinmgD_k(t) - \eigv_n \int_0^t \big\langle \eigf_n, \limZ(s) \big\rangle ds\Big| \stackrel{\text{P}}{=} 0.
\end{align*}
This concludes the proof for~\ref{lem:linmg2}.

\ref{lem:linmg3} Given the moment bounds from Proposition~\ref{prop:mom},
it is not hard to check that $ \{\linmgR_{k,N}(\Cdot)\}_{k,N} $ is tight in $ D[0,T] $.
This being the case, it suffices to establish one point convergence:
\begin{align}
	\label{eq:linmggoal}
	\lim_{k\to\infty}\limsup_{N\to\infty}
	\big| \linmgR_{k,N}(t) \big| \stackrel{\text{P}}{=} 0.
\end{align}

To this end, fix $ u\in(0,1) $, $ v\in(0,\urt) $, $ \Lambda,T<\infty $,
recall the definition of $ \Omega=\Omega(u,v,\Lambda,T,N) $ from~\eqref{eq:Omega},
and recall the notation $ \Exrt[\,\Cdot\,]:= \Ex[\,\Cdot\,|\rt(x),x\in\T] $.
Multiply both sides of~\eqref{eq:lingmgR} by $ \ind_{\Omega}\ind_\set{\eigv_n<\Lambda} $,
and calculate the second moment (with respect to $ \Exrt[\,\Cdot\,] $) of $ \linmgR_{k,N}(t) $.
With the aid of~\eqref{eq:qv:bd} and the moment bounds from Proposition~\ref{prop:mom}, we have
\begin{align}
	\label{eq:linmg2}
	\Exrt&\big[\linmgR_{k,N}(t)^2\big] \ind_{\Omega} \ind_\set{\lambda_n<\Lambda}
\\
	\notag
	&\leq
	c(u,v,T,\Lambda)
	\sum_{i=1}^k
	\int_{N^2t_{i-1}}^{N^2t_i} \frac{1}{N^3}\sum_{\tilx\in\T}
	\Bigg( \Big(  \sum_{x\in\T} \eigf_n(Nx)\sg(N^2t_i-s;x,\tilx) -\eigf_n(N\tilx) \Big)^2
	\Exrt\big[Z(s,\tilx)^2\big]\ind_{\Omega} \ind_\set{\lambda_n<\Lambda} \Bigg)
	ds
\\
	\label{eq:linmg3}
	&\leq
	c(u,v,T,\Lambda)
	\sum_{i=1}^k
	\frac{1}{N^2}\int_{N^2t_{i-1}}^{N^2t_i} \frac{1}{N}\sum_{\tilx\in\T}\Big(  \sum_{x\in\T} \eigf_n(Nx)\sg(N^2t_i-s;x,\tilx) -\eigf_n(N\tilx) \Big)^2 ds
	\ind_\set{|\lambda_n|<\Lambda}.
\end{align}
Let $ N\to\infty $ in~\eqref{eq:linmg3}.
Given that $ \eigf_n \in H^1(\limT) \subset L^1(\limT) $,
with the aid of Proposition~\ref{prop:sgtoSg},
we have
\begin{align*}
	\sum_{x\in\T} \eigf_n(Nx)\sg(N^2(t_i-s);x,N\tilx) \to \int_{\limT} \eigf_n(x)\Sg(t_i-s;x,\tilx) dx,
	\qquad
	\text{uniformly in } \tilx\in\limT.
\end{align*}
Hence
\begin{align}
	\notag
	\limsup_{N\to\infty} \eqref{eq:linmg3}
	&\leq
	c(u,v,T,\Lambda)
	\sum_{i=1}^k
	\int_{t_{i-1}}^{t_i} \int_{\limT}\Big( \int_{\limT}\eigf_n(x)\Sg(t_i-s;x,\tilx) d\tilx - \eigf_n(\tilx) \Big)^2 d\tilx \ind_\set{\lambda_n<\Lambda}
\\
	\notag
	&=
	c(u,v,T,\Lambda)
	\sum_{i=1}^k
	\int_{t_{i-1}}^{t_i} \int_{\limT}\Big( (e^{(t_i-s)}\eigv_n-1) \eigf_n(\tilx) \Big)^2 d\tilx \ind_\set{\lambda_n<\Lambda}
\\
	\label{eq:linmg5}
	&=
	c(u,v,T,\Lambda) k \int_{0}^{\frac{t}{k}} (e^{s\eigv_n}-1)^2 ds
	\leq
	k^{-2} (u,v,T,\Lambda).
\end{align}
Now, combine~\eqref{eq:linmg2}--\eqref{eq:linmg5},
take $ \Ex[\,\Cdot\,] $ of the result, and let $ k\to\infty $.
We arrive at
\begin{align}
	\label{eq:linmg6}
	\lim_{k\to\infty} \limsup_{N\to\infty}\Ex\big[\linmgR_{k,N}(t)^2\ind_{\Omega}\ind_\set{\eigv_n<\Lambda} \big] = 0.
\end{align}
Indeed, $ \Pr[\set{\eigv_n<\Lambda}] \to 1 $ as $ \Lambda \to \infty $,
and Proposition~\ref{prop:sg} asserts that $ \Pr[\Omega]=\Pr[\Omega(u,v,\Lambda,T,N)] \to 1 $ under the iterative limit $ (\lim_{\Lambda\to\infty} \lim_{N\to\infty}\Cdot) $.
Combining these properties with \eqref{eq:linmg6} yields the desired result~\eqref{eq:linmggoal}.
\end{proof}

Lemma~\eqref{lem:linmg} together with~\eqref{eq:linmg:pf:} gives
\begin{align}
	\label{eq:mgcnvg}
	\sup_{t\in[0,T]} |\Mg_n(t)-\mg_n(t)|
	\longrightarrow_\text{P}
	0.
\end{align}
Knowing that~$ \mg_n(t) $ is an $ \filZ $-martingale, we conclude that $ \Mg_n(t) $ is a local martingale.

\subsection{Quadratic martingale problem}
\label{sect:quadmg}

Our goal here is to show that $ \limZ $ solves the quadratic martingale problem~\eqref{eq:qdmg}.
With $ \mg_n(t) $ given in~\eqref{eq:mgn},
the first step is to calculate the cross variation of $ \mg_n(t)\mg_{n'}(t) $:
\begin{align}
	\label{eq:qdmg:qv}
	\langle \mg_n, \mg_{n'} \rangle(N^2t)
	=
	\int_0^{N^2t} \frac{1}{N^2} \sum_{x,x'\in\T} \eigf_n(Nx)\eigf_{n'}(Nx') d\langle \mg(s,x), \mg(s,x') \rangle.
\end{align}
Given~\eqref{eq:qv}, the r.h.s.\ of~\eqref{eq:qdmg:qv} permits an explicit expression in terms of $ \eta(s,x) $ and $ Z(s,x) $.
Relevant to our purpose here is an expansion of the expression that exposes the $ N\to\infty $ asymptotics.
To this end,
with $ Z(t,x) $ defined in~\eqref{eq:Z},
note that
\begin{align}
	\label{eq:taylor}
	\eta(t,x)Z(t,x) &= \tfrac{\tau^{1/2}-1}{\tau^{1/2}-\tau^{1/2}} Z(t,x) + \tfrac{1}{\tau^{1/2}-\tau^{1/2}} \nabla Z(t,x-1),
\\
	\label{eq:taylor:}
	\eta(t,x+1)Z(t,x) &= \tfrac{1-\tau^{-1/2}}{\tau^{1/2}-\tau^{1/2}} Z(t,x) + \tfrac{1}{\tau^{1/2}-\tau^{1/2}} \nabla Z(t,x).
\end{align}
Recall the filtration~$ \filZ(t) $ from \eqref{eq:filZ}.
In the following we use $ \bdd(t,x)=\bdd^{(N)}(t,x) $ to denote a \emph{generic} $ \filZ $-adopted process
that may change from line to line (or even within a line), but is bounded uniformly in $ t,x,N $. 
Set
\begin{align}
	\label{eq:W}
	W(t,x) := N(\nabla Z(t,x))(\nabla Z(t,x-1)).
\end{align}
Using the identities~\eqref{eq:taylor}--\eqref{eq:taylor:} in~\eqref{eq:qv},
together with $ r=\frac{1-1/\sqrt{N}}{2} $, $ \ell=\frac{1+1/\sqrt{N}}{2} $, $ \tau:=r/\ell $
and $ |\rtt(x)| \leq c $ (from \eqref{eq:Z}, \eqref{eq:was}, and Assumption~\ref{assu:rt}\ref{assu:rt:bdd}),
we have
\begin{align}
	\notag
	\tfrac{d~}{ds}&\langle \mg(s,x), \mg(s,x') \rangle
	=(r-\ell)^2 \rtt(x) \Big( \tfrac{1}{\ell}\eta(s,x) + \tfrac{1}{r}\eta(s,x+1) - \big(\tfrac1r+\tfrac1\ell\big)\eta(t,x)\eta(s,x+1)) \Big)Z(s,x)^2
\\
	\label{eq:qdmg:qv::}
	&= \tfrac{\rtt(x)}{N} \Big( \big( Z^2(s,x) + W(s,x) ) \big) + N^{-\frac12}\bdd(s,x) Z^2(s,x) \Big),
\end{align}
From~\eqref{eq:taylor}--\eqref{eq:taylor:}, it is readily checked that
\begin{align}
	\label{eq:Wapbd}
	|W(t,x)| \leq c Z^2(t,x).
\end{align}
In~\eqref{eq:qdmg:qv::}, write $ \rtt(x)=1+\rt(x) $,
and use~\eqref{eq:Wapbd} to get $ \frac{\rt(x)}{N}(Z^2(s,x)+W(s,x)) = \frac{\rt(x)}{N} \bdd(s,x)Z^2(s,x) $.
Also, since $ \rtt(x) $ is bounded (from Assumption~\ref{assu:rt}\ref{assu:rt:bdd}),
we have $ \frac{\rtt(x)}{N}N^{-\frac12}\bdd(s,x) Z^2(s,x) = \frac{1}{N} N^{-\frac12} \bdd(s,x) Z^2(s,x) $.
From these discussions we obtain
\begin{align}
	\tfrac{d~}{ds}\langle \mg(s,x), \mg(s,x') \rangle
	\label{eq:qdmg:qv:}	
	&= \tfrac{1}{N} \Big( \big( Z(s,x)^2 + W(s,x) ) + (\rt(x) + N^{-\frac12}) \bdd(s,x) Z^2(s,x) \Big).
\end{align}
Inserting~\eqref{eq:qdmg:qv:} into~\eqref{eq:qdmg:qv} gives
\begin{subequations}
\label{eq:qvv}
\begin{align}
	\label{eq:qdmg:qvv}
	\langle \mg_n, \mg_{n'} \rangle(N^2t)
	&=
	\frac{1}{N^2}\int_0^{N^2t} \frac{1}{N} \sum_{x\in\T} \eigf_n(Nx)\eigf_{n'}(Nx') Z^2(s,x) ds
	+
	\mgg_1(t) + \mgg_2(t),
\\
	\label{eq:qdmg:qvv1}
	\mgg_1(t)
	&:=	\frac{1}{N^2}\int_0^{N^2t} \frac{1}{N} \sum_{x\in\T} \eigf_n(Nx)\eigf_{n'}(Nx') (\rt(x) + N^{-\frac12}) \bdd(s,x) Z^2(s,x) ds,
\\	
	\label{eq:qdmg:qvv2}
	\mgg_2(t)
	&:=
	\frac{1}{N^2}\int_0^{N^2t} \frac{1}{N} \sum_{x\in\T} \eigf_n(Nx)\eigf_{n'}(Nx) W(s,x) ds.
\end{align}
\end{subequations}

Indeed, the r.h.s.\ of~\eqref{eq:qdmg:qvv} is the discrete analog of
$ \int_0^{t} \langle \eigf_n\eigf_{n'}, \limZ^2(s) \rangle ds  $ that appears in~\eqref{eq:quadmg:}.
By~\eqref{eq:rtbd}, $ \norm{\rt}_{L^\infty(\T)} \leq N^{-\urt} $ with probability $ \to_{\Lambda,N} 1 $.
With the aid of moment bounds from Proposition~\ref{prop:mom},
it is conceivable $ \mgg_1(t) $ converges in $ C[0,T] $ to zero in probability.
On the other hand, $ W(s,x) $ does \emph{not} converge to zero for fixed $ (s,x) $.
In order to show that $ \mgg_2(t) $ converges to zero,
we  capitalize on the spacetime averaging in~\eqref{eq:qdmg:qvv2}.
The main step toward showing such an averaging is a decorrelation estimate on $ W(s,x) $,
stated in Proposition~\ref{prop:selfav}.

To prove the decorrelation estimate,
we follow the general strategy of \cite{bertini97}.
The idea here is to develop an integral equation for $ \ExO[ W(t+s,x) | \filZ(s) ] $ and try to `close' the equation.
Closing the equation means bounding terms on the r.h.s.\ of the integral equation,
so as to end up with an integral inequality for $ \ExO[ W(t+s,x) | \filZ(s) ] $.
Crucial to success under this strategy are certain nontrivial inequalities involving the kernel $ \sg(t;x,\tilx) $,
which we now establish. These are considerably more difficult to demonstrate in the inhomogeneous case (versus the homogeneous case).

\begin{remark}
\label{rmk:selfav}
Self-averaging properties like Proposition~\ref{prop:selfav} are often encountered in the context of convergence of particle systems to \acp{SPDE}.
In particular, in addition to the approach of \cite{bertini97} that we are following,
alternative approaches have been developed in different contexts.
This includes hydrodynamic replacement \cite{quastel11} and the Markov duality method \cite{corwin18a}.
The last two approaches do not seem to apply in the current context.
For hydrodynamic replacement, one needs two-block estimates to relate the fluctuation of $ h(t,x) $ to the quantity $ W(t,x) $.
Inhomogeneous \ac{ASEP} under Assumption~\ref{assu:rt} sits beyond the scope of existing proofs of two-block estimates.
As for the duality method, it is known \cite{borodin14} that inhomogeneous \ac{ASEP} enjoys a duality
via the function $ \til{Q}(t,\vec{x}):=\prod_{i=1}^n \eta(t,x_i)\tau^{h(t,x_i)} $.
(Even though \cite{borodin14} treats \ac{ASEP} in the full-line $ \Z $,
duality being a local property, readily generalizes to $ \T $.)
For the method in~\cite{corwin18a} to apply, however,
one also needs a duality via the function $ Q(t,\vec{x}):=\prod_{i=1}^n \tau^{h(t,x_i)} $,
which is lacking for the inhomogeneous \ac{ASEP}.
\end{remark}

In what follow, for $ f,g\in[0,\infty)\times\T^2\to\R $, we write
\begin{align}
	\label{eq:sgmg:sgmgg}
	\sgmg_{f,g}(t;x,\tilx) := (\nabla_{x}f(t;x,\tilx))(\nabla_x g(t;x-1,\tilx)),
	\qquad
	\sgmgg_{f,g}(t;x) := \sum_{\tilx\in\T}|\sgmg_{f,g}(t;x,\tilx)|.
\end{align}
Recall also that $ \sgr(t):=\sg(t)-\hk(t) $ denotes the difference of $ \sg(t)=e^{t\ham} $ to $ \hk(t) $, the kernel of the homogeneous walk.

\begin{lemma}
\label{lem:sgkey}
Fix $ u\in(\frac12,1) $, $ v\in(0,\urt) $, $ \Lambda,T<\infty $.
We have, for all $ t\in[0,N^2T]  $ and $ x\in\T $,
\begin{enumerate}[label=(\alph*),leftmargin=7ex]
\item \label{lem:sgrkey=0} \
	$
		\displaystyle
		\int_0^{N^2T} \sgmgg_{f,g}(s;x) ds \, \ind_{\Omega(u,v,\Lambda,T,N)}
		\leq
		c(u,v,T,\Lambda)N^{-(u\wedge v)}\log(N+1) \qquad \textrm{when  }(f,g)=(\hk,\sgr), (\sgr,\hk), (\sgr,\sgr),
	$
\item \label{lem:sgrkeytail} \
	$
		\displaystyle
		\sgmgg_{\sg,\sg}(t;x)\ind_{\Omega(u,v,\Lambda,T,N)}
		\leq
		c(u,v,T,\Lambda)(1+t)^{-(u+\frac12)},
	$
\item \label{lem:sgkey=0} \
	$
		\displaystyle
		\Big| \int_0^{t} \sum_{\tilx\in\T}\sgmg_{\hk,\hk}(s;x,\tilx)ds \Big| \leq  \frac{c}{\sqrt{t+1}},
	$
\item \label{lem:sgkey<1} \
	There exists a universal $ \beta<1 $ and $ N_0=N_0(u,v,T,\Lambda) $ such that\\
	$
		\displaystyle
		\int_0^{N^2T} \sgmgg_{\sg,\sg}(s;x) ds \, \ind_{\Omega(u,v,\Lambda,T,N)} \leq \beta,
	$
	for all $ N\geq N_0 $.
\end{enumerate}
\end{lemma}
\begin{proof}
Throughout this proof we assume that $ s,t\leq TN^{2} $,
and, to simplify notation, write $ c=c(u,v,\Lambda,T) $
and $ \Omega=\Omega(u,v,\Lambda,T,N) $. At times below we will apply earlier lemmas or propositions wherein variables were labeled $x$ or $u$. We will not, however, always apply them with the values of $x$ and $u$ specified in our proof (for instance, we may want to apply a result with $u=1$). In order to avoid confusion, when we specify the value $\Cdot$ of $x$ or $u$ (or other variables) used in that application of an earlier result we will write $x\mapsto \Cdot$ or $u\mapsto \Cdot$.

\ref{lem:sgrkey=0}
Our first aim is to bound the expression
$
	\sum_{\tilx} |\nabla f(s;x,\tilx)| |\nabla g(s;x',\tilx)|
$
when $ (f,g)=(\hk,\sgr), (\sgr,\hk), (\sgr,\sgr) $.
To this end, bound $ |\nabla f(s;x,\tilx)| $ by it supremum over $ \tilx\in\Z $,
and use \eqref{eq:hkgd} or Proposition~\ref{prop:sg}\ref{prop:sgr:gsup} (with $ x'\mapsto x-1 $ and $u\mapsto u$),
and for the remaining sum use \eqref{eq:hkgd:sum} or Proposition~\ref{prop:sg}\ref{prop:sgr:gsum}.
This gives
\begin{align}
	\label{eqlem:sgrkey=0:}
	\sum_{\tilx\in\T} |\nabla f(s;x,\tilx)\,\nabla g(s;x',\tilx)| \ind_{\Omega}
	&\leq
	c\,\Bigg(
		\frac{1}{s+1} \bigg( \frac{N^{-v}}{(s+1)^{u/2}}+N^{-u} \bigg)
		+
		\bigg( \frac{N^{-v}}{(s+1)^{(u+1)/2}}+\frac{N^{-u}}{\sqrt{s+1}} \bigg) \frac{1}{\sqrt{s+1}}
\\
	\label{eqlem:sgrkey=0::}
	&\hphantom{\leq c\Big(\frac{1}{s+1} }	
		+
		\bigg( \frac{N^{-v}}{(s+1)^{(u+1)/2}}+\frac{N^{-u}}{\sqrt{s+1}} \bigg)
		\bigg( \frac{N^{-v}}{(s+1)^{u/2}}+N^{-u} \bigg)
	\Bigg),
\end{align}
when $ (f,g)=(\hk,\sgr), (\sgr,\hk), (\sgr,\sgr) $.
Expand the terms on the r.h.s.\ of~\eqref{eqlem:sgrkey=0:},
and (using $ u<1 $), bound $ N^{-v}/(s+1)^{1+\frac{u}2} \leq N^{-v}/(s+1)^{u+\frac12} $.
In~\eqref{eqlem:sgrkey=0::}, use $ u>\frac12 $ and $ s\leq TN^2 $ to bound
$ N^{-u}/\sqrt{s+1} \leq c\,(s+1)^{-1} $.
We then have, when $ (f,g) = (\hk,\sgr), (\sgr,\hk), (\sgr,\sgr) $,
\begin{align}
	\label{eqlem:sgrkey=0}
	\sum_{\tilx\in\T} |\nabla f(s;x,\tilx)\,\nabla g(s;x',\tilx)| \ind_{\Omega}
	\leq
	\frac{cN^{-v}}{(s+1)^{u+\frac12}} + \frac{cN^{-u}}{s+1}.
\end{align}
Integrate~\eqref{eqlem:sgrkey=0} over $ s\in[0,t] $.
Given that $ u>\frac12 $, we have $ \int_0^t N^{-v}/(s+1)^{u+\frac12} ds \leq c N^{-v} $;
Given that $ t\leq N^2T $,
we have $ \int_0^t N^{-u}/(s+1) ds \leq c N^{-u}\log(N+1) $.
From these considerations we conclude the desired bound.

\ref{lem:sgrkeytail}
Using~\eqref{eq:hkgd}--\eqref{eq:hkgd:sum} gives $ \sum_{\tilx\in\T} |\nabla \hk(t;x,\tilx)\,\nabla \hk(t;x',\tilx)| \leq c\,(t+1)^{-3/2} $.
Combining this with~\eqref{eqlem:sgrkey=0},
and using $ N^{-u}/(t+1) \leq c\,(t+1)^{-(1+\frac{u}{2})} $, we conclude the desired result.

\ref{lem:sgkey=0}
Recall that $ \hk $ solves the lattice heat equation~\eqref{eq:lhe}.
Multiply both sides of~\eqref{eq:lhe} by $ \hk(s;x',\tilx) $, sum over $ x\in\T $, and integrate over $ s\in[0,\infty) $. We have
\begin{align*}
	\sum_{\tilx\in\T} \int_0^\infty \tfrac12\partial_s\big( \hk(s;x',\tilx) \hk(s;x,\tilx) \big) ds
	&=
	\int_0^\infty \sum_{\tilx\in\T} \tfrac14 \big( \hk(s,x',\tilx)\Delta_x \hk(s;x,\tilx) + (\Delta_{x'}\hk(s,x',\tilx))\,\hk(s;x,\tilx) \big) ds
\\
	&=
	-\int_0^\infty \sum_{\tilx\in\T} \tfrac12 \nabla_{x'}\hk(s;x',\tilx) \nabla_x \hk(s;x,\tilx)ds.
\end{align*}
With $ \hk(0;x,\tilx)=\ind_\set{x=\tilx} $ and $ \hk(\infty;x,\tilx)=\frac{1}{N} $,
the l.h.s.\ is equal to $ \frac12(\frac{1}{N}-\ind_\set{x=x'}) $.
This gives
\begin{align*}
	\int_0^\infty \sum_{\tilx\in\T} \nabla_{x'}\hk(s,x',\tilx) \nabla_x \hk(s;x,\tilx)ds
	=
	\ind_\set{x=x'} - \frac1N.
\end{align*}
Set $ x'\mapsto x-1 $ gives $ \int_0^t \sum_{\tilx\in\T} \sgmg_{\hk,\hk}(s;x,\tilx)ds = \frac1N-\int_{t}^\infty \sum_{\tilx\in\T} \sgmg_{\hk,\hk}(s;x,\tilx)ds $.
To bound the last term, use~\eqref{eq:hkgd}--\eqref{eq:hkgd:sum} (with $ u\mapsto 1 $) to get
\begin{align*}
	\Big| \int_0^t \sum_{\tilx\in\T} \sgmg_{\hk,\hk}(s;x,\tilx)ds \Big|
	\leq
	\frac1N+
	\Big| \int_{t}^\infty \sum_{\tilx\in\T} \sgmg_{\hk,\hk}(s;x,\tilx)ds \Big|
	\leq
	\frac1N+
	\int_{t}^\infty \frac{c}{(s+1)^{3/2}} ds
	\leq
	\frac1N
	+
	\frac{c}{\sqrt{t+1}}.
\end{align*}
This together with $ \frac1N \leq \frac{c}{\sqrt{t+1}} $ completes the proof.

\ref{lem:sgkey<1}
Since $ \sg=\hk+\sgr $, we have $ \sgmgg_{\sg,\sg}(s,x)\leq (\sgmgg_{\hk,\hk}+\sgmgg_{\sgr,\sg}+\sgmgg_{\sg,\sgr}+\sgmgg_{\sgr,\sgr})(s,x) $.
The bounds established in part \ref{lem:sgrkey=0} of this lemma gives
\begin{align*}
	\sup_{x\in\T}
	\int_0^{N^2T}
	\big( \sgmgg_{\sgr,\sg}+\sgmgg_{\sg,\sgr}+\sgmgg_{\sgr,\sgr} \big)(s,x) ds \ind_{\Omega}
	\longrightarrow_\text{P}
	0.
\end{align*}
Granted this, it now suffices to show that there exists $\beta'<1$ and $N_0(u,v,T,\Lambda)$ such that
\begin{align}
	\label{eqlem:sgkey<1}
	\int_0^{N^2T} \sgmgg_{\hk,\hk}(s;x) ds
	:=
	\int_0^{N^2T} \sum_{\tilx\in\T} |\nabla\hk(s;x,\tilx) \, \nabla\hk(s;x-1,\tilx)| ds
	\leq \beta',
	\qquad
	\textrm{for }N\geq N_0(u,v,T,\Lambda).
\end{align}
for some universal constant $ \beta'<1 $.
Recall that $ \hk^\Z(t;x) $ denotes the kernel $ \hk $ of the homogeneous walk on $ \Z $,
and that $ \hk $ is expressed in terms of $ \hk^\Z $ by~\eqref{eq:hk:hkZ}.
Let $ I := (-\frac{N}{2},\frac{N}{2}]\cap\Z \subset\Z $ denote an interval in $ \Z $ of length $ N $ centered at $ 0 $.
Under this setup we have
\begin{align}
	\notag
	\sum_{\tilx\in\T} |\nabla\hk(s;x,\tilx) \, \nabla\hk(s;x-1,\tilx)|
	&=
	\sum_{y\in I} \Big|\sum_{j\in\Z}\nabla\hk^\Z(s;y+jN) \, \sum_{j'\in\Z}\nabla\hk^\Z(s;y-1+j'N)\Big|	
\\
	\label{eq:nabnab}
	&\leq
	\sum_{y\in I}\sum_{j,j'\in\Z} \Big|\nabla\hk^\Z(s;y+jN) \, \nabla\hk^\Z(s;y-1+j'N)\Big|.
\end{align}
Within \eqref{eq:nabnab}, the diagonal terms $ j=j'$, after being summed over $ y\in I $,
jointly contribute to
\begin{align*}
	V(s) :=  \sum_{y\in \Z}\big|\nabla\hk^\Z(s;y) \, \nabla\hk^\Z(s;y-1)\big|.
\end{align*}
We set the contribution of off-diagonal terms to be
\begin{align}
	\label{eq:nabnabS}
	S(s)
	:=
	\sum_{y\in I}\sum_{j\neq j'\in\Z} \big|\nabla\hk^\Z(s;y+jN) \, \nabla\hk^\Z(s;y-1+j'N)\big|,
\end{align}
Integrating \eqref{eq:nabnab} over $ s\in[0,N^2T] $ then gives
\begin{align}
	\label{eq:nabnab:}
	\int_0^{N^2T}\sum_{\tilx\in\T} |\nabla\hk(s;x,\tilx) \, \nabla\hk(s;x-1,\tilx)| ds
	\leq
	\int_0^{N^2T} V(s) ds
	+
	\int_0^{N^2T} S(s) ds.
\end{align}
For the first term on the r.h.s.\ of~\eqref{eq:nabnab:}, it is known~\cite[Lemma~A.3]{bertini97} that
\begin{align}
	\label{eq:bgkey}
	\int_0^{N^2T} V(s) ds = \int_0^\infty \sum_{y\in\Z} |\nabla\hk^\Z(s;y) \, \nabla\hk^\Z(s;y-1)| ds
	=: \beta''
	<1.
\end{align}
To bound the last term in~\eqref{eq:nabnab:},
we use the bound from
\cite[Eq~(A.13)]{dembo16}, which in our notation reads
$ |\nabla\hk^\Z(s;y+iN)| \leq \frac{1}{s+1} e^{-\frac{|y+iN|}{\sqrt{s+1}}} $.
Further, for all $ y\in I $ we have $ |y|\leq \frac{N}{2} $, which gives $ |y+iN| \geq \frac{1}{c} (|y|+ |i|N) $, and hence
$
	|\nabla\hk^\Z(s;y+iN)| \leq \tfrac{1}{s+1} e^{-\frac{|y|+|i|N}{c\sqrt{s+1}}},
$
for all $ y\in I $.
Using this bound on the r.h.s.\ of~\eqref{eq:nabnabS} gives
\begin{align*}
	S(s)
	\leq
	c\,\Big( \sum_{j\neq j'\in\Z} e^{-\frac{(|j|+|j'|)N}{\sqrt{s+1}}} \Big) \Big( \sum_{y\in \Z} \frac{ e^{-\frac{|y|}{c\sqrt{s+1}}} }{(s+1)^2} \Big)
	\leq
	c \, e^{-\frac{N}{c\sqrt{s+1}}} \, (s+1)^{-3/2}.
\end{align*}
Integrating this inequality over $ s\in[0,N^2T] $,
and combining the result with~\eqref{eq:nabnab:}--\eqref{eq:bgkey} yields
\begin{align*}
	\int_0^{N^2T}\sum_{\tilx\in\T} |\nabla\hk(s;x,\tilx) \, \nabla\hk(s;x-1,\tilx)| ds
	\leq
	\beta'' + c \int_0^{N^2 T} e^{-\frac{N}{c\sqrt{s+1}}}(s+1)^{-3/2} ds.
\end{align*}
Fix $ \alpha \in (0,1) $ and divide the last integral into integrals over $ s\in[0,N^{-\alpha}] $ and $ s\in[N^{-\alpha},N^2T] $.
We see that $ \int_0^{N^2 T} \exp(-\frac{N}{c\sqrt{s+1}})(s+1)^{-3/2} ds \leq (\exp(-\frac{1}c N^{-1+\alpha})+ N^{-\alpha/2}) c \to 0 $.
Hence we conclude~\eqref{eqlem:sgkey<1} for $ \beta'=\frac{\beta''+1}{2} <1 $.
\end{proof}

Given Lemma~\ref{lem:sgkey}, we now proceed to establish an integral inequality of the conditional expectation of $ W(t+s,x) $.

\begin{lemma}
\label{lem:selfav:}
Fix $ u\in(\frac12,1) $, $ v\in(0,\urt) $, $ \Lambda,T<\infty $.
Let $ \Omega':=\Omega(u,v,T,\Lambda,N)\cap \Omega(1,v,T,\Lambda,N) $
and $ \ExO[\,\Cdot\,] := \Ex[\,\Cdot\,|\rt(x),x\in\T] \ind_{\Omega'} $.
We have, for all $ 	s,t\in[0,N^2T] $ and $ x\in\T $,
\begin{align}
\label{eq:decorIter}
\begin{split}
	\ExO&\Big[\,\big| \ExO\big[ W(t+s,x) \big| \filZ(s) \big] \big| \, \Big]
	\leq
	c(u,v,\Lambda,T) \big( N^{-(\frac{u}2\wedge\uic\wedge v)}\log(N+1) + \tfrac{1}{\sqrt{t+1}} + \tfrac{N}{t+1} \big)
\\
	&+
	\int_0^t
	\sum_{\tilx\in\T} \sgmg_{\sg,\sg}(t';x,\tilx) \ExO\Big[\,\big| \ExO\big[ W(t'+s,x) \big| \filZ(s) \big] \big| \,\Big] dt'.
\end{split}
\end{align}
\end{lemma}
\begin{proof}
Throughout this proof we assume $ s,t\leq TN^{2} $,
and, to simplify notation, we write $ c=c(u,v,\Lambda,T) $.
Recall from~\eqref{eq:W} that $ W(t+s,x) $ involves $ x $-gradients of $ Z $.
The idea is to derive equations for $ \nabla_x Z(t,x) $.
To this end, set $ (t_*,t)\mapsto(s,s+t) $ in~\eqref{eq:Lang:int:} and take $ \nabla_x $ on both sides to get
\begin{align}
	\label{eq:gZeq}
	&\nabla_x Z(t+s,x) = D(x) + F(x),
\\
	\label{eq:DF}
	&
	D(x) := \sum_{\tilx\in\T}\nabla_x\sg(s;x,\tilx)Z(s,\tilx),
	\qquad
	F(x) := \int_{s}^{t+s} \sum_{\tilx\in\T}\nabla_x\sg(t-\tau;x,\tilx)d\mg(\tau,\tilx).
\end{align}
Note that we have omit dependence on $ s,t $ in the notation $ D(x), F(x) $.
Similar convention is practiced in the sequel.
Use~\eqref{eq:gZeq} twice with $ x\mapsto x $ and $ x\mapsto x-1 $  to express $ W $ in terms of $ D $ and $ F $.
Since $ F(x) $ is a martingale integral and since $ D(x) $ is $ \filZ(s) $-measurable,
upon taking $ \ExO[\,\Cdot\,|\filZ(s)] $, we have
\begin{align}
	\label{eq:quadmg:ExW}
	&\ExO[W(t+s,x)|\filZ(s)]
	=
	N D(x)D(x-1) + N \ExO[F(x)F(x+1)|\filZ(s)].
\end{align}
To evaluate the last term in~\eqref{eq:quadmg:ExW},
recall that $ \bdd(t,x) $ denotes a generic $ \filZ $-adopted uniformly bounded process,
and note that, from~\eqref{eq:rtbd}, we have $ |\rt(x)| \leq \Lambda N^{-\urt} $ under $ \Omega $.
Recall the notation $ \sgmg_{f,g}, \sgmgg_{f,g} $ from~\eqref{eq:sgmg:sgmgg}.
Using~\eqref{eq:qdmg:qv:} we write
\begin{align*}
	N \ExO[F(x)F(x+1)|\filZ(s)] = F_1(x)+F_2(x)+F_3(x),
\end{align*}
where
\begin{align}
	\label{eq:quadmg:F1}
	F_1(x)
	&:=
	\int_{s}^{s+t} \sum_{\tilx\in\T}\sgmg_{\sg,\sg}(t-t';x,\tilx) \ExO[ Z^2(s+t',\tilx)|\filZ(s)] dt',
\\
	\notag
	F_2(x)
	&:=
	\int_{0}^{t} \sum_{\tilx\in\T}\sgmg_{\sg,\sg}(t-t';x,\tilx) \ExO[ W(s+t',\tilx)|\filZ(s)] dt',
\\
	\label{eq:quadmg:F3}
	F_3(x)
	&:=
	N^{-(\frac12\wedge\urt)}
	\int_{s}^{s+t} \sum_{\tilx\in\T}\sgmg_{\sg,\sg}(t-t';x,\tilx) \ExO[ \bdd(t',\tilx)Z^2(s+t',\tilx)|\filZ(s)] dt'.
\end{align}
Note that $ F_2(x) $ is expressed in terms of $ \ExO[W(t+s,x)|\filZ(s)] $, $ t\geq 0 $.
Insert \eqref{eq:quadmg:F1}--\eqref{eq:quadmg:F3} into the last term in~\eqref{eq:quadmg:ExW},
and take $ \ExO[\,|\Cdot|\,] $ on both sides.
We now obtain
\begin{align}
\label{eq:quadmg:iter}
\begin{split}
	\ExO&\big[\, \big| \ExO[W(t+s,x)|\filZ(s)] \big| \, \big]
	\leq
	N \ExO\big[ \, |D(x)D(x-1)| \, \big] + \ExO\big[\,|F_1(x)| \, \big]
\\
	&\quad
	+ \int_{0}^{t} \sum_{\tilx\in\T}\sgmg_{\sg,\sg}(t-t';x,\tilx) \ExO\big[\, |\ExO[ W(s+t',\tilx)|\filZ(s)]| \, \big] dt'
	+ \ExO\big[\,|F_3(x)|\,\big].
	\end{split}
\end{align}
To proceed, we bound the residual terms in~\eqref{eq:quadmg:iter} that involves $ D $, $ F_1 $, and $ F_3 $.

We begin with the term $ N\ExO[\,|D(x)D(x-1)|\,] $.
Using the expression~\eqref{eq:DF} for $ D(x) $,  we take $ \normO{\,\Cdot\,}{2} $ on both sides,
and write $ \sg(t)=\hk(t)+\sgr(t) $.
With the aid of the moment bound on $ \normO{Z(s,\tilx)}{2} $ from Proposition~\ref{prop:mom}, we obtain
\begin{align}
	\normO{D(x)}{2}
	\leq
	\sum_{\tilx\in\T} |\nabla_x\sg(t;x,\tilx)| \, \normO{Z(s,\tilx)}{2}
	\leq
	c\sum_{\tilx\in\T} |\nabla_x\hk(t;x,\tilx)| + c\sum_{\tilx\in\T} |\nabla_x\sgr(t;x,\tilx)|.
\end{align}
Further using~\eqref{eq:hkgd:sum} (with $ x'\mapsto x-1 $ and $ u\mapsto 1 $) and using
the bound from Proposition~\ref{prop:sg}\ref{prop:sgr:gsum} (with $ u\mapsto 1 $
and $v\mapsto v$) gives
$
	\normO{D(x)}{2}
	\leq
	( \tfrac{1}{\sqrt{t+1}} + N^{-1} )c
	\leq
	\tfrac{c}{\sqrt{t+1}},
$
where, in the last inequality, we used $ t \leq TN^{2} $. We were able to take $u\mapsto 1$ in our application of \eqref{eq:hkgd:sum} because we are on the event $\Omega':=\Omega(u,v,T,\Lambda,N)\cap\Omega(1,v,T,\Lambda,N)$.
Given this bound, applying Cauchy--Schwarz inequality we have
\begin{align}
	\label{eq:quadmg:Dbd:}
	N \ExO\big[\,|D(x)D(x-1)|\,\big]
	\leq
	N\normO{D(x)}{2} \normO{D(x-1)}{2}
	\leq
	\tfrac{cN}{t+1}.
\end{align}

Next we turn to bounding $ \ExO[\,|F_1(x)|\,] $.
First, given the decomposition $ \sgmg =\sgmg_{\hk,\hk}+\sgmg_{\hk,\sgr}+\sgmg_{\sgr,\hk}+\sgmg_{\sgr,\sgr} $,
we write $ F_1(x) = F_{11}(x)+F_{12}(x) $, where
\begin{align*}
	F_{11}(x)
	&=
	\int_{0}^{t} \sum_{\tilx\in\T}\big(\sgmg_{\hk,\sgr}+\sgmg_{\sgr,\hk}+\sgmg_{\sgr,\sgr}\big)(t-t';x,\tilx) \ExO[ Z^2(s+t',\tilx)|\filZ(s)] dt',
\\
	F_{12}(x)
	&=
	\int_{0}^{t} \sum_{\tilx\in\T}\sgmg_{\hk,\hk}(t-t';x,\tilx) \ExO[ Z^2(s+t',\tilx)|\filZ(s)] dt'.
\end{align*}
For $ F_{11}(x) $, we use bounds from Lemma~\ref{lem:sgkey}\ref{lem:sgrkey=0}
and moment bounds from Proposition~\ref{prop:mom} to get $ \ExO[\,|F_{11}(x)|\,] \leq cN^{-(u\wedge v)}\log(N+1) $.
As for $ F_{12}(x) $, we further decompose $ F_{12}(x)=F_{121}(x)+F_{122}(x) $,
where
\begin{align*}
	F_{121}(x)
	&=
	\ExO[ Z^2(s+t,x)|\filZ(s)] \int_{0}^{t} \sum_{\tilx\in\T}\sgmg_{\hk,\hk}(t-t';x,\tilx) dt',
\\
	F_{122}(x)
	&=
	\int_{0}^{t} \sum_{\tilx\in\T}\sgmg_{\hk,\hk}(t-t';x,\tilx) \ExO[ Z^2(s+t',\tilx)-Z^2(s+t,x)|\filZ(s)] dt'.
\end{align*}
For $ F_{121}(x) $, taking $ \ExO[\,|\,\Cdot\,|\,] $ and using the moment bound on $ \normO{Z(s+t,\tilx)}{2} $ from Proposition~\ref{prop:mom},
followed by using Lemma~\ref{lem:sgkey}\ref{lem:sgkey=0}, we have $ \ExO[\,|F_{121}(x)|\,] \leq \frac{c}{\sqrt{t+1}} $.
As for $ F_{122}(x) $, write
\begin{align*}
	|Z^2(s+t',\tilx)-Z^2(s+t,x)|
	\leq
	\big(Z(s+t',\tilx)+Z(s+t,x)\big) \big( |Z(s+t',\tilx)-Z(s+t,\tilx)| + |Z(s+t,\tilx)-Z(s+t,x)|	\big).
\end{align*}
Set $ \alpha=\frac{u}{2}\wedge\uic\wedge v $ to simplify notation.
Using the moment bounds from Proposition~\ref{prop:mom}, here we have
\begin{align*}
	\ExO\big[\,|F_{122}(x)|\,\big]
	\leq
	c \int_{0}^{t} \sum_{\tilx\in\T}|\sgmg_{\hk,\hk}(t-t';x,\tilx)|
	\Big(
		\Big(\frac{\dist_\T(x,\tilx)}{N}\Big)^{\alpha}
		+
		\Big(\frac{|t-t'|\vee 1}{N^2}\Big)^{\frac{\alpha}{2}}
	\Big)
	dt'.
\end{align*}
Further using the bounds~\eqref{eq:hkgd}--\eqref{eq:hkgd:sum} (with $ x'\mapsto x-1 $ and $ u\mapsto 1 $)
and \eqref{eq:hk:hold:} (with $ u\mapsto \alpha $) gives
\begin{align*}
	\ExO\big[\,|F_{122}(x)|\,\big]
	\leq
	c \int_{0}^{t}
	\frac{1}{(t-t'+1)}
	\frac{N^{-\alpha}}{(t-t'+1)^{(1-\alpha)/2}}
	dt'
	\leq
	c N^{-\alpha}
	=c N^{-(\frac{u}2\wedge\uic\wedge v)}.
\end{align*}
Collecting the preceding bounds on the $ F $'s, we conclude
\begin{align}
	\label{eq:quadmg:F1bd}
	\ExO\big[\,|F_1(x)|\,\big]
	\leq
	\frac{c}{\sqrt{t+1}} + c N^{-(\frac{u}2\wedge\uic\wedge v)} \log(N+1).
\end{align}

As for $ F_3(x) $.
Recall that $ \bdd(t,x) $ denotes a (generic) uniformly bounded process.
Taking $ \ExO[\,|\,\Cdot\,|\,] $ in~\eqref{eq:quadmg:F3} and using the moment bounds from Proposition~\ref{prop:mom}
and using Lemma~\ref{lem:sgkey}\ref{lem:sgkey<1},
we have
\begin{align}
	\label{eq:quadmg:F3bd}
	\ExO\big[\,|F_3(x)|\,\big]
	\leq
	cN^{-(\frac12\wedge\urt)}.
\end{align}

Inserting~\eqref{eq:quadmg:Dbd:}--\eqref{eq:quadmg:F3bd} into~\eqref{eq:quadmg:iter} completes the proof.
\end{proof}

We now establish the required decorelation estimate on $ W $.

\begin{proposition}
\label{prop:selfav}
Let $ u,v,\Lambda,T, \Omega' $, $ \ExO[\,\Cdot\,] $ be as in Lemma~\ref{lem:selfav:}.
There exists $ c=c(u,v,\Lambda,T) $ such that, for all $ 	s,t\in[0,N^2T]$ and $ x\in\T $,
\begin{align}
	\label{eq:selfav:goal}
	\ExO\Big[\,\big|\ExO\big[ W(t+s,x) \big| \filZ(s) \big]\big| \, \Big]
	\leq
	c\, \Big( N^{-(\frac{u}{2}\wedge\uic\wedge v)}\log(N+1) + \tfrac{1}{\sqrt{t+1}} + \tfrac{N}{t+1} \Big).
\end{align}
\end{proposition}
\begin{proof}
Through the proof, we write $ c=c(u,v,T,\Lambda) $ to simplify notation, and assume $ t\in[0,N^2T] $.
For fixed $ s \in[0,N^2T] $, set $ w(t) := \sup_{x\in\T} \ExO[\,|\ExO[ W(t+s,x) \big| \filZ(s) ]|\,] $,
which is the quantity we aim to bound,
and consider also $ w(t,x) := \ExO[\,|\ExO[ W(t+s,x) \big| \filZ(s) ]|\,] $.
Taking supremum over $ x\in\T $ in~\eqref{eq:decorIter} gives
\begin{align*}
	w(t,x)
	\leq
	c \Big( N^{-(\frac{u}2\wedge\uic\wedge v)}\log(N+1) + \tfrac{1}{\sqrt{t+1}} + \tfrac{N}{t+1} \Big)
	+
	\int_0^t
	\sgmgg_{\sg,\sg}(t-t';x) w(t') dt'.
\end{align*}
Iterating this inequality gives
\begin{align}
	\label{eq:decorIter:}
	w(t,x)
	\leq
	c \Big( N^{-(\frac{u}2\wedge\uic\wedge v)} + \tfrac{1}{\sqrt{t+1}} + \tfrac{N}{t+1} \Big)
	+
	\sum_{n=1}^\infty \big(w_{1,n}(t,x) + w_{2,n}(t,x) + w_{3,n}(t,x)\big),
\end{align}
where, with the notation $ \Sigma_n(t) $ from~\eqref{eq:Sigman} and $ d^n\vec{s} $ from~\eqref{eq:ds}, we have
\begin{align*}
	w_{i,n}(t,x)
	:=
	\int_{\Sigma_n(t)}
	\Big(\prod_{i=1}^n\sgmgg_{\sg,\sg}(s_i;x) \Big)
	\cdot
	\left\{\begin{array}{l@{,}l}
		N^{-(\frac{u}2\wedge\uic\wedge v)}\log(N+1)	& \text{ for } i=1
		\\
		\frac{1}{\sqrt{s_0+1}}	& \text{ for } i=2
		\\
		\frac{N}{s_0+1} 			& \text{ for } i=3
	\end{array}
	\right\}
	\cdot
	d^n\vec{s}.
\end{align*}
Let $ \beta:= \sup_{x\in\T} \int_0^{N^2T} \sgmgg_{\sg,\sg}(t,x) dt $,
which, by Lemma~\ref{lem:sgkey}\ref{lem:sgkey<1}, is strictly less than $ 1 $ (uniformly in $ N $).
For $ w_{1,n}(t,x) $, noting that the integral does not involve the variable $ s_0 $, we bound
\begin{align}
\label{eq:w1bd}
\begin{split}	
	\sum_{n=1}^\infty w_{1,n}(t,x)
	&\leq
	N^{-(\frac{u}2\wedge\uic\wedge v)} \log(N+1)
	\sum_{n=1}^\infty
	\int_{[0,N^2T]^n}
	\prod_{i=1}^n\sgmgg_{\sg,\sg}(s_i;x) ds_i
\\
	&=
	N^{-(\frac{u}2\wedge\uic\wedge v)} \log(N+1)
	\tfrac{\beta}{1-\beta}
	=
	c N^{-(\frac{u}2\wedge\uic\wedge v)}\log(N+1).
\end{split}
\end{align}

To bound $ w_{2,n} $ and $ w_{3,n} $, we invoke the argument from \cite[Proof of Proposition~3.8]{labbe17}.
We begin with $ w_{2,n} $.
Split $ w_{2,n}(t,x) $ into integrals over $ \Sigma_{n}(t)\cap\set{s_0>\frac{t}{n+1}} $
and over $ \Sigma_{n}(t)\cap\set{s_0\leq\frac{t}{n+1}} $, i.e., $ w_{2,n}(t,x)=w'_{2,n}(t,x)+w''_{2,n}(t,x) $, where
\begin{align*}
	w'_{2,n}(t,x)
	:=
	\hspace{-10pt}
	\int\limits_{\Sigma_n(t)\cap\set{s_0>\frac{t}{n+1}}}
	\hspace{-10pt}
	\Big(\prod_{i=1}^n\sgmgg_{\sg,\sg}(s_i;x) \Big)
	\cdot
	\frac{1}{\sqrt{s_0+1}}
	d^n\vec{s},&
	&
	w''_{i,n}(t,x)
	:=
	\hspace{-10pt}
	\int\limits_{\Sigma_n(t)\cap\set{s_0\leq\frac{t}{n+1}}}
	\hspace{-10pt}
	\Big(\prod_{i=1}^n\sgmgg_{\sg,\sg}(s_i;x) \Big)
	\cdot
	\frac{1}{\sqrt{s_0+1}}
	\cdot
	d^n\vec{s}.
\end{align*}
For $ w'_{i,n} $, we bound $ \frac{1}{\sqrt{s_0+1}} $ by $ c(\frac{n+1}{t+1})^{1/2} $.
Doing so releases the $ s_0 $ variable from the integration, yielding
\begin{align}
	\label{eq:w'}
	w'_{i,n}(t,x)
	\leq
	c
	\Big(\frac{n+1}{t+1}\Big)^{1/2}	
	\int_{[0,N^2T]^n}
	\Big(\prod_{i=1}^n\sgmgg_{\sg,\sg}(s_1;x) \Big)
	d^n\vec{s}
	=
	c n \beta^n \Big(\frac{n+1}{t+1}\Big)^{1/2}.	
\end{align}
As for $ w''_{i,n} $,
we note that the integration domain is necessarily a subset of $ \Sigma_n(t)\cup_{i=1}^n\set{s_i>\frac{t}{n+1}} $.
At each encounter of $ s_i>\frac{t}{n+1} $, we invoke the bound from Lemma~\ref{lem:sgkey}\ref{lem:sgrkeytail}.
This gives
\begin{align}
	\label{eq:w''}
	w''_{i,n}(t,x)
	\leq
	c\sum_{i=1}^n
	\Big( \frac{n+1}{t+1} \Big)^{u+\frac12}
	\int_{\Sigma_n(t)}
	\Big(\prod_{i'\in\set{1,\ldots,n}\setminus\set{i}}\sgmgg_{\sg,\sg}(s_{i'};x) \Big)
	\frac{1}{\sqrt{s_0+1}}
	d^n\vec{s}.
\end{align}
For each $ i=1,\ldots,n $, the integral in~\eqref{eq:w''} does not involve the variable $ s_i $.
We then bound
\begin{align}
	\label{eq:w'':}
	w''_{i,n}(t,x)
	\leq
	c
	\sum_{i=1}^n
	\Big( \frac{n+1}{t+1} \Big)^{u+\frac12}
	\int_{[0,t]^n}
	\Big(\prod_{i'\in\set{1,\ldots,n}\setminus\set{i}}\sgmgg_{\sg,\sg}(s_{i'};x) ds_i \Big)
	\frac{1}{\sqrt{s_0+1}} ds_0
	\leq
	\frac{c(n+1)^{u+\frac12}\beta^{n-1}}{(t+1)^{u}}.
\end{align}
Combine~\eqref{eq:w'} and~\eqref{eq:w'':}, and sum the result over $ n\geq 1 $.
With $ \beta<1 $ and $ u>\frac12 $, we conclude
\begin{align}
	\label{eq:w2bd}
	\sum_{n=1}^\infty w_{2,n}(t,x)
	\leq
	\frac{c}{\sqrt{t+1}}.
\end{align}
As for $ w_{3,n}(t) $, the same calculations as in the preceding gives
\begin{align*}
	w_{3,n}(t,x)
	\leq
	cN\cdot
	\Big(
		n \beta^n \frac{n+1}{t+1}
		+
		\frac{c(n+1)^{u+\frac12}\beta^{n-1}}{(t+1)^{u+\frac12}} \log(t+2)
	\Big)
	\leq
	cN\cdot\frac{(n+1)^{u+\frac12} \beta^{n-1}}{t+1},
\end{align*}
where the factor $ \log(t+2) $ arises from integrating $ \frac{1}{s_0+1} $,
and the second inequality follows since $ u>\frac12 $.
Summing over $ n\geq 1 $, with $ \beta<1 $, we have
\begin{align}
	\label{eq:w3bd}
	\sum_{n=1}^\infty w_{3,n}(t,x)
	\leq
	\frac{cN}{t+1}.
\end{align}

Inserting~\eqref{eq:w1bd}, \eqref{eq:w2bd}--\eqref{eq:w3bd} into~\eqref{eq:decorIter:} completes the proof.
\end{proof}

Having established the decorrelation estimate in Proposition~\ref{prop:selfav},
we continue to prove that $ \limZ $ solves the quadratic martingale problem~\eqref{eq:quadmg:}.
Recall the definition of $ \mg_n(t) $ from~\eqref{eq:mgn}.
Consider the discrete analog $ \mgg_{n,n'}(t) $ of $ \Mgg_{n,n'}(t) $ (defined in~\eqref{eq:quadmg:}):
\begin{align*}
	\mgg_{n,n'}(t)
	:=
	\mg_n(t)\mg_{n'}(t)
	-
	\int_0^t \langle \eigf_{n} \eigf_{n'}, \scZ^2_N(s) \rangle_N ds.
\end{align*}
Recall from \eqref{eq:mgcnvg} that $ \mg_n $ converges in $ C[0,T] $ to $ \Mg_n $ in probability.
Also, from~\eqref{eq:Zcnvg},
$
	\int_0^t \langle \eigf_{n} \eigf_{n'}, \scZ^2_N(s) \rangle_N ds
$
converges in $ C[0,T] $ to its continuum counterpart
$
	\int_0^t \langle \eigf_{n} \eigf_{n'}, \limZ^2(s) \rangle ds
$
in probability.
Consequently, $ \mgg_{n,n'} $ converges in $ C[0,T] $ to $ \Mgg_{n,n'} $ in probability.

On the other hand, we know that
$
	\mgg'_{n,n'}(t)
	:=
	\mg_n(t)\mg_{n'}(t) - \langle\mg_n,\mg_{n'} \rangle (t)
$
is a martingale, and, given the expansion~\eqref{eq:qvv},
we have
$	
	\mgg_{n,n'}(t) - \mgg'_{n,n'}(t) = \mgg_1(t) + \mgg_2(t).
$
In the following we will show that $ \mgg_1 $ and $ \mgg_2 $ converges in $ C[0,T] $ to zero in probability.
Given this, \eqref{eq:mgcnvg}, and the fact that $ \mgg'_{n,n'}(t) $ is a martingale,
it then follows that \eqref{eq:quadmg:} is a local martingale.

It now remains only to show that $ \mgg_1 $ and $ \mgg_2 $ converges in $ C[0,T] $ to zero in probability.
Given the moment bounds Proposition~\ref{prop:mom},
it is not hard to check that $ \mgg_1, \mgg_2 $ is tight in $ C[0,T] $.
This being the case, it suffices to establish one point convergence:
\begin{lemma}
For a fixed $ t\in\R_{\geq 0} $, we have that $ \mgg_1(t),\mgg_2(t) \to_\text{P} 0 $.
\end{lemma}

\begin{proof}
Fixing $ u\in(0,1) $, $ v\in(0,\uic) $, $ t\in[0,T] $, $ \Lambda<\infty $,
throughout this proof we write
$ \Omega=\Omega(u,v,\Lambda,T,N) $,
$ \Omega'=\Omega(u,v,\Lambda,T,N)\cap\Omega(1,v,\Lambda,T,N) $,
$ c=c(u,v,T,\Lambda) $, and $ \Exrt[\,\Cdot\,]:= \Ex[\,\Cdot\,|\rt(x),x\in\T] $.

We begin with $ \mgg_1 $.
Recall that $ \bdd $ denotes a generic uniformly bounded process.
By~\eqref{eq:rtbd}, $ \norm{\rt}_{L^\infty(\T)} \leq N^{-\urt} $.
Given this, taking $ \Exrt[|\,\Cdot\ind_{\Omega}\,|] $ in~\eqref{eq:qdmg:qvv1}
using the moment bound on $ Z(t,x) $ from Proposition~\ref{prop:mom}
and using $ \norm{\eigf_n}_{L^\infty(\limT)} \leq c\norm{\eigf_n}_{H^1(\limT)} $ (from~\ref{fn:LinfinH1}),
we have
\begin{align*}
	\Exrt[| \mgg_1(t)\ind_{\Omega}|]
	:=
	\Ex\big[\,| \mgg_1(t) |\ind_{\Omega}\,\big| \rt(x) \big]	
	\leq
	c
	N^{-(\frac12\wedge \urt)}\,\norm{\eigf_n}_{H^1(\limT)}\norm{\eigf_{n'}}_{H^1(\limT)}.
\end{align*}
Set $ \Gamma = \Gamma(n,n',\Lambda):= \set{\,\norm{\eigf_n}_{H^1(\limT)}\norm{\eigf_{n'}}_{H^1(\limT)} \leq \Lambda} $.
Multiply both sides by $ \ind_\Gamma $, and take $ \Ex[\,\Cdot\,] $ on both sides to get
$
	\Ex[\,| \mgg_1(t) |\ind_{\Omega}\ind_{\Gamma}\,]	
	\leq
	c
	N^{-(\frac12\wedge \urt)}.
$
This gives $ | \mgg_1(t) |\ind_{\Omega}\ind_{\Gamma} \to_\text{P} 0 $ as $ N\to\infty $.
More explicitly, writing $ \mgg_1(t)=\mgg_1(t;N) $, $ \Omega=\Omega(u,v,T,\Lambda,N) $, and $ \Gamma=\Gamma(\Lambda) $,
we have, for each fixed $ \e>0 $,
\begin{align}
	\label{eq:mgg1to0}
	\lim_{N\to\infty} \Pr\big[ \set{|\mgg_1(t;N)|>\e}\cap\Omega(u,v,T,\Lambda,N)\cap\Gamma(\Lambda) \big] =0.
\end{align}
Indeed, with $ n,n' $ being fixed, we have
\begin{align}
	\label{eq:Gammato1}
	\lim_{\Lambda\to\infty}
	\Pr[\,\Gamma(\Lambda)^\text{c}\,] = \Pr[\,\norm{\eigf_n}_{H^1(\limT)}\norm{\eigf_{n'}}_{H^1(\limT)} > \Lambda\,] = 0.
\end{align}
Also, Proposition~\ref{prop:sg} asserts that
\begin{align}
	\label{eq:iterlimit}
	\limsup_{\Lambda\to\infty}
	\limsup_{N\to\infty}\Pr[\Omega(u,v,T,\Lambda,N)^\text{c}]=0.
\end{align}
Use the union bound to write
\begin{align*}
	\Pr\big[ |\mgg_1(t;N)|>\e \big]
	\leq
	\Pr\big[ \set{|\mgg_1(t;N)|>\e}\cap\Omega(u,v,T,\Lambda,N)\cap\Gamma(\Lambda) \big]
	+
	\Pr[\Omega(u,v,T,\Lambda,N)^\text{c}]
	+
	\Pr[\,\Gamma(\Lambda)^\text{c}\,],
\end{align*}
and send $ N\to\infty $ and $ \Lambda\to\infty $ \emph{in order} on both sides.
With the aid of \eqref{eq:mgg1to0}--\eqref{eq:iterlimit}, we conclude that\\
$ \lim_{N\to\infty} \Pr[ |\mgg_1(t;N)|>\e ] = 0 $, for each $ \e >0 $.
That is, $ \mgg_1(t;N) \to_\text{P} 0 $, as $ N\to\infty $.

Turning to $ \mgg_2 $, in~\eqref{eq:qdmg:qvv2},
we take $ \ExO[(\,\Cdot\,)^2\ind_{\Omega'}] $ on both sides to get
\begin{align*}
	\Exrt[(\mgg_2(t))^2\ind_{\Omega'}]
	\leq
	\big( \norm{\eigf_n}_{H^1(\limT)}\norm{\eigf_{n'}}_{H^1(\limT)} \big)^2
	\frac{2}{N^4}\int_{s_1<s_2\in[0,N^2t]^2} \frac{1}{N^2} \sum_{x_1,x_2\in\T} |\Exrt[ W(s_2,x_2) W(s_1,x_1)\ind_{\Omega'} ]| ds_1ds_2.
\end{align*}
Multiplying both sides by $ \ind_{\Gamma} $,
we replace $ (\norm{\eigf_n}_{H^1(\limT)}\norm{\eigf_{n'}}_{H^1(\limT)})^2 $ with $ \Lambda^2=c $ on the r.h.s.\ to get
\begin{align}
	\label{eq:mgg2:bd}
	\Exrt[(\mgg_2(t))^2\ind_{\Omega'}\ind_\Gamma]
	\leq
	\frac{c}{N^4}\int_{s_1<s_2\in[0,N^2t]} \frac{1}{N^2} \sum_{x_1,x_2\in\T} |\Exrt[ W(s_2,x_2) W(s_1,x_1)\ind_{\Omega'} ]| ds_1ds_2.
\end{align}

To bound the expectation on the r.h.s.\ of~\eqref{eq:mgg2:bd},
we fix a threshold $ \kappa>0 $, and split the expectation into
$ \Exrt[ W(s_2,x_2) W(s_1,x_1)\ind_{\Omega'} ] = f_1+ f_2 $, where
\begin{align*}
	f_1:=
	\Exrt[ W(s_2,x_2) W(s_1,x_1)\ind_{\Omega'}\ind_{|W(s_1,x_1)|\leq\kappa} ],
	\qquad
	f_2:=
	\Exrt[ W(s_2,x_2) W(s_1,x_1)\ind_{\Omega'}\ind_{W(s_1,x_1)>\kappa} ].
\end{align*}
For $ f_1 $, insert the conditional expectation $ \Ex[\,\Cdot\,|\filZ(s_1)] $, and then use Proposition~\ref{prop:selfav} to show
\begin{align*}
	|f_1|
	\leq
	\kappa \Exrt\big|\Exrt\big[ W(s_2,x_2)\ind_{\Omega'} \big| \filZ(s_1) \big] \big|
	\leq
	c\kappa \,
	\Big( N^{-(\frac{u}{2}\wedge\uic\wedge v)}\log(N+1)+\tfrac{1}{\sqrt{s_2-s_1+1}} + \tfrac{N}{s_2-s_1+1} \Big).
\end{align*}
As for $ f_2 $, apply Markov's inequality followed by using~\eqref{eq:Wapbd} to get
\begin{align*}
	|f_2| 
	\leq
	c \kappa^{-1} \sup_{s\in[0,N^2T]} \sup_{x\in\T} \Exrt[ Z^4(s,x)\ind_{\Omega'} ]
	\leq
	c\kappa^{-1},
\end{align*}
where the last inequality follows from the moment bound on $ Z(s,x) $ from Proposition~\ref{prop:mom}.
Inserting  the bounds on $ |f_1| $ and $ |f_2| $ into~\eqref{eq:mgg2:bd} now gives
\begin{align}
\label{eq:mgg2:bd:}
\begin{split}
	\Exrt[(\mgg_2(t))^2\ind_{\Omega'}\ind_\Gamma]
	&\leq
	\frac{c}{N^4}\int\limits_{s_1<s_2\in[0,N^2t]}
	\hspace{-15pt}
	\Big( \kappa\,\Big( N^{-(\frac{u}{2}\wedge\uic\wedge v)}\log(N+1)+\frac{1}{\sqrt{s_2-s_1+1}} + \frac{N}{s_2-s_1+1} \Big) + \kappa^{-1} \Big) ds_1ds_2
\\
	& \leq c \kappa N^{-(\frac{u}{2}\wedge\uic\wedge v)}\log(N+1) + c\kappa^{-1}.
\end{split}
\end{align}
Now, choose $ \kappa= N^{-(\frac{u}{4}\wedge\frac{\uic}{2}\wedge\frac{v}{2})} $,
and take $ \Ex[\,\Cdot\,] $ on both sides of~\eqref{eq:mgg2:bd:}.
This gives
\begin{align*}
	\Ex[(\mgg_2(t))^2\ind_{\Omega'}\ind_\Gamma]
	\leq c
	N^{-(\frac{u}{4}\wedge\frac{\uic}{2}\wedge\frac{v}{2})}\,\log(N+1) \to 0.
\end{align*}
Given this, similarly to the preceding, after passing $ \Lambda $ to a suitable subsequent $ \Lambda_N\to\infty $ in~\eqref{eq:mgg1to0},
we conclude that $ \mgg_2(t) \to_\text{P} 0 $.
\end{proof}

\bibliographystyle{alphaabbr}
\bibliography{inhmgASEP}
\end{document}